\documentclass{amsart}

 \usepackage{kantlipsum} 
 \calclayout
 \usepackage{microtype} 
\usepackage[colorlinks,citecolor=blue,urlcolor=blue]{hyperref}
\usepackage{amssymb,mathrsfs,enumerate,xcolor,color}

\usepackage{graphicx}

\newtheorem{theorem}{Theorem}[section]
\newtheorem{lemma}[theorem]{Lemma}
\newtheorem{corollary}[theorem]{Corollary}

\theoremstyle{definition}

\theoremstyle{remark}
\newtheorem{remark}[theorem]{Remark}

\newcommand*{\dif}{\ensuremath{\mathop{}\!\mathrm{d}}}

\def\R{\mathbb{R}}

\def\P{\mathbb{P}}

\numberwithin{equation}{section}

\begin{document}

 \title[From $0$ to $3$: Intermediate phases of  the maximum of two-type BBM]{From $0$ to $3$: Intermediate phases between normal and anomalous spreading  of two-type branching Brownian motion}


\author{Heng Ma}
\address{School of Mathematical Sciences, Peking
University}
\curraddr{}
\email{hengmamath@gmail.com}
\thanks{}

\author{Yan-Xia Ren}
\address{LMAM School of Mathematical Sciences \& Center for
Statistical Science, Peking
University}
\curraddr{}
\email{yxren@math.pku.edu.cn}
\thanks{Yan-Xia Ren is the corresponding author.}
\thanks{
The research of this project is supported
   by the National Key R\&D Program of China (No. 2020YFA0712900).
 The research of Yan-Xia Ren is supported by NSFC (Grant Nos. 12071011 and 12231002) and  the Fundamental Research Funds for Central Universities, Peking University LMEQF}

\date{}

\dedicatory{}

\begin{abstract}
  The logarithmic correction for the order of the maximum of a two-type reducible branching
  Brownian motion on the real line
  exhibits a double jump when the parameters (the ratio of the diffusion
  coefficients of the two types of particles,
  and the ratio of the branching rates
  the two types of particles)
  cross the boundary of the anomalous spreading region identified by Biggins.

 In this paper, we further examine this double jump phenomenon  by
 studying a two-type reducible branching
  Brownian motion on the real line with its parameters depend on the time horizon $t$. We show that when the parameters
  approach the boundaries of the anomalous spreading region in an appropriate way, the order of the maximum can interpolate smoothly between different surrounding regimes. We also determine the asymptotic law of the maximum and characterize the extremal process.
\end{abstract}

\maketitle


\section{Introduction and main results}
Branching Brownian motion (BBM) is a probabilistic model that describes the evolution of a population of individuals.
This model has been intensively studied and continues to be the subject of many recent researches.
A large literature focused on the link between BBM and the F-KPP reaction-diffusion equation introduced in \cite{Fisher37} and \cite{KPP37}.
 For results in this direction, see
\cite{BGKM23, Bramson83, CR88, Harris99, Kyprianou04, McKean75, MRR22} and the references therein.
Understanding the spatial  spread of such a  population, particularly the propagation of the front, is a fundamental question
and has attracted significant interest, see e.g.
 \cite{ABBS13, ABK11, ABK13,  BBGMRR22,  BH17, Bramson78, CHL19, LS87}.
The insights and methods used in studying the extreme values of BBM are applicable to a large class of probabilistic models, including the two-dimensional discrete Gaussian free field, epsilon-cover time of the two-dimensional torus by Brownian motion, and characteristic polynomials of random matrices.
For further details, we refer our reader to the lecture notes
 \cite{Kistler14, Zeitouni12} and the reviews \cite{Arguin16, BK22}.

Multitype branching Brownian motion is a natural extension of BBM that
can be used to describe the evolution a population composed of different types or species.
In the irreducible case (where each type of individual can have descendants of
all types),
it behaves in some sense like  an \textit{effective} single-type BBM, see e.g. \cite{HRS23, RY}.
This paper focuses on a \textit{two-type reducible} case
(where individuals of type 2 cannot have descendants of type 1), which is the simplest setting
in which the maximum exhibits a phase transition
that is not observed in the case of single-type BBM.

In \cite{Biggins10,Biggins12}, Biggins gave a comprehensive description of the leading order of the maximum of the two-type reducible BBM.
    There are three cases $\mathscr{C}_{I},\mathscr{C}_{II}$ and $\mathscr{C}_{III}$:
 In $\mathscr{C}_{I}$, the maximum equals the speed of individuals of type $1$; in $\mathscr{C}_{II}$, the maximum is equal to the speed of individuals of type $2$; and in $\mathscr{C}_{III}$, the spreading speed of the two-type process is strictly larger than
the speeds of a single-type  BBM of particles of type $1$ or type  $2$.
In this paper, $\mathscr{C}_{I}$ and $\mathscr{C}_{II}$ are referred to as normal spreading regions, while $\mathscr{C}_{III}$ is referred to as the \textit{anomalous spreading} region.
The regime to which the process belongs
depends on  the ratio $\beta$ of the branching rates and the  ratio $\sigma^{2}$ of the diffusion coefficients for individuals of types $1$ and $2$.
Belloum and Mallein \cite{BM21} obtained the logarithmic correction of the  maximum and the limiting extremal process when $(\beta,\sigma^{2})$ are interior points of $\mathscr{C}_{I},\mathscr{C}_{II},\mathscr{C}_{III}$.
Further studies on the boundary case by Belloum in \cite{Belloum22} and by the authors in \cite{MR23} completed the phase diagram for the maximum (see Figure \ref{fig-phase}).
Notably, a double jump occurs in the logarithmic correction for
the order of the maximum when the parameters $(\beta,\sigma^{2})$
cross  the boundary of the anomalous spreading region $\mathscr{C}_{III}$.

\begin{figure}[htbp]
  \centering
 \includegraphics[width=\linewidth]{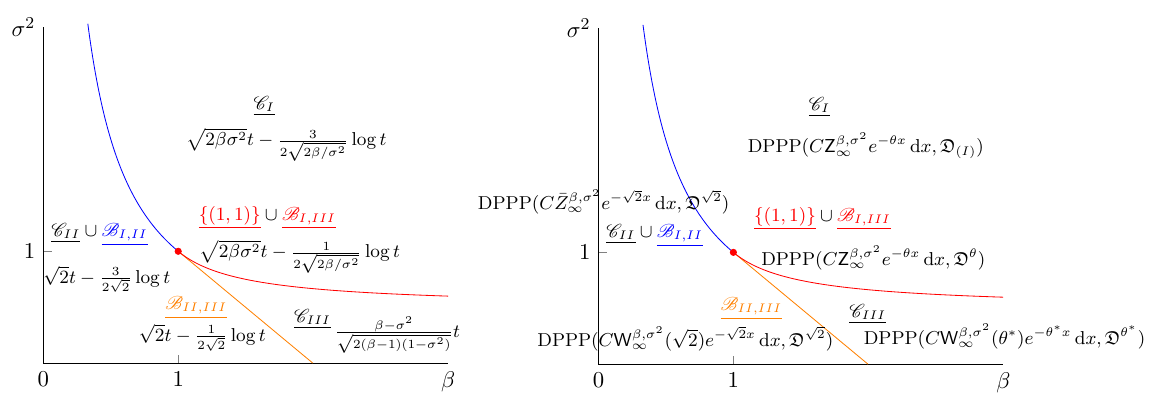}
  \caption{Phase diagram for the order of the  maximum and limiting extremal process of two-type BBM}\label{fig-phase}
  \end{figure}

A further interesting question (see \cite[Question 1]{MR23}) is to make the logarithmic correction smoothly interpolates between normal spreading cases and anomalous spreading case.
Similar problems  for variable speed BBM were investigated in \cite{Bh20} and \cite{KS15}:
The logarithmic correction for the order of the maximum for two-speed BBM changes discontinuously when approaching slopes $\sigma_1^2=$ $\sigma_2^2=1$, which corresponds to standard BBM. Bovier and Huang \cite{Bh20} further studied this transition by choosing $\sigma_1^2=1 \pm t^{-\alpha}$ and $\sigma_2^2=1 \mp t^{-\alpha}$, and showed that the logarithmic correction for the order of the maximum  smoothly interpolates between the correction in the i.i.d. case $\frac{1}{2 \sqrt{2}} \log t$, standard BBM case $\frac{3}{2 \sqrt{2}} \log t$, and $\frac{6}{2 \sqrt{2}} \log t$ when $\alpha \in (0,1/2)$.
Inspired by these two papers, we study in this paper a two-type reducible BMM with parameters depending on the time horizon $t$. We assume that the parameters $(\beta_{t},\sigma_{t}^{2})$  approach  the boundary of $\mathscr{C}_{III}$ appropriately. We show that  the logarithmic correction for the maximum smoothly interpolates between the normal spreading cases  and anomalous spreading case
(see Figure \ref{fig-phase-2}).
Moreover, we find the asymptotic law of the maximum and characterize the extremal process, which turns out to coincide (up to a constant)
with that of a two-type reducible BBM with parameters $(\lim\limits_{t \to \infty} \beta_{t}, \lim\limits_{t \to \infty} \sigma^{2}_{t})$.

\subsection{Branching Brownian motions}

A  BBM on the real line can be described as follows:
Initially, there is a particle which moves as a Brownian motion with diffusion coefficient $\sigma^{2}$ starting from the origin.
 At rate $\beta$,  the initial particle splits into two particles.
The offspring particles start moving from their place of birth independently, with same diffusion coefficient and obeying the same branching rule.
We denote this process by $\{ ( \mathsf{X}^{\beta,\sigma^{2}}_{u}(t), u \in \mathsf{N}_{t})_{t \geq 0}, \mathsf{P} \}$, where $\mathsf{N}_{t}$ is the set of all particles alive at time $t$ and $\mathsf{X}^{\beta,\sigma^{2}}_u(t)$ is the position of an individual $u \in \mathsf{N}_{t}$. If  $\beta=\sigma^{2}=1$, we  call $\{\mathsf{X}^{1,1}_{u}(t)\}$ the standard BBM and write $\{\mathsf{X}_{u}(t)\}$ for short.
The scaling property of Brownian motion implies that
\begin{equation*}
\left(\mathsf{X}_u^{\beta, \sigma^2}(t): u \in \mathsf{N}_t\right) \stackrel{\text { law }}{=}\left(\frac{\sigma}{\sqrt{\beta}} \mathsf{X}_u(\beta t): t \in \mathsf{N}_{\beta t}\right).
\end{equation*}

Let $\mathsf{M}^{\beta,\sigma^{2}}_t=\max_{u \in \mathsf{N}_{t}}\mathsf{X}_{u}^{\beta,\sigma^{2}} (t)$ be  the maximum of BBM at time $t$. It is well-known  that  the
centered maximum
$\mathsf{M}^{\beta,\sigma^{2}}_t$ converges in distribution to a randomly shifted Gumbel random variable  (see \cite{Bramson78,Bramson83,LS87}).
More precisely, letting
\begin{equation*}
  v= \sqrt{2 \beta \sigma^{2}} \quad \text{ and }  \quad \theta =
  \sqrt{2 \beta / \sigma^{2}},
\end{equation*}
then
\begin{equation}\label{eq-converge-to-Gumbel}
\lim_{t \to \infty}
  \mathsf{P} \left( \mathsf{M}^{\beta,\sigma^{2}} _{t}-  v t + \frac{3}{2 \theta} \log t \leq x \right)
  =  \mathsf{E} \left[ \exp \left(-  C  \mathsf{Z}^{\beta,\sigma^{2}}_{\infty} e^{-\theta x}\right)
  \right]
\end{equation}
for some  constant $C$ depending on $\beta, \sigma^{2}$, where
$\mathsf{Z}^{\beta,\sigma^{2}}_{\infty}$   is the almost sure limit of the \textit{derivative martingale} $(\mathsf{Z}^{\beta,\sigma^{2}}_{t})_{t>0}$ defined by $ \mathsf{Z}^{\beta,\sigma^{2}}_{t} :=\sum_{u \in \mathsf{N}_t} [ v t-\mathsf{X}^{\beta,\sigma^{2}}_u(t) ] \exp( \theta \mathsf{X}^{\beta,\sigma^{2}}_u(t)-2 \beta t ) $. The name ``derivative martingale" comes from the fact that $\mathsf{Z}^{\beta,\sigma^{2}}_{t}= -\frac{\partial}{\partial \lambda}|_{\lambda = \theta}\mathsf{W}^{\beta,\sigma^{2}}_t(\lambda) $, where
 $\mathsf{W}^{\beta,\sigma^{2}}_t(\lambda):=\sum_{u \in \mathsf{N}_t} \exp\{\lambda \mathsf{X}^{\beta,\sigma^{2}}_u(t)- (\beta+\frac{\lambda^2 \sigma^{2}}{2} )t\}$ are the \textit{additive martingales} for BBM.

 The construction of the limiting extremal process for BBM,
 obtained independently in \cite{ABBS13} and \cite{ABK13},
 gives a deeper understanding of the extreme value statistics for BBM. Precisely,
\begin{equation}
  \label{eq-converge-to-DPPP}
   \lim_{t \to \infty} \sum_{u \in \mathsf{N}_t} \delta_{\mathsf{X}_{u}(t)-\sqrt{2}t + \frac{3}{2\sqrt{2}}\log t }    =\operatorname{DPPP}\left(\sqrt{2} C_{\star} \mathsf{Z}_{\infty} \mathrm{e}^{-\sqrt{2} x} \mathrm{~d} x, \mathfrak{D}^{\sqrt{2}}\right)
  \end{equation}
  in law in the vague topology\footnote{
   We say that point measures $ \mathcal{E}_{t}$ converges to  $\mathcal{E}_{\infty}$ as $t\to\infty$ in the vague topology, if
 for every continuous and compactly supported test function $\phi$,
  $ \int \phi \dif \mathcal{E}_{t} $ converges weakly to
  $\int \phi \dif \mathcal{E}_{\infty}$.
  In this paper  when dealing with the weak convergence
 of point processes,
  we always assume that we are using this setting.},
where DPPP $(\mu, \mathfrak{D})$   stands for
a \textit{decorated Poisson point process} with intensity measure $\mu$ and  decoration law $\mathfrak{D}$.
Given a (random) measure $\mu$ on $\mathbb{R}$ and a point process $\mathfrak{D}$ on $\mathbb{R}$,
let $\sum_{i} \delta_{x_{i}} $ be a Poisson point process with intensity $\mu$
and let $\left( \sum_{j} \delta_{d^{i}_{j}} :i \geq 0\right)$  be an independent family of i.i.d. point processes with common law $\mathfrak{D}$, then $  \sum_{i,j}   \delta_{x_i+d_j^i}$ is a decorated Poisson point process with intensity measure $\mu$ and  decoration law $\mathfrak{D}$.
The decoration law $\mathfrak{D}^{\sqrt{2}}$ in \eqref{eq-converge-to-DPPP} belongs to a family of  ``gap processes" $\left(\mathfrak{D}^{\varrho}, \varrho \geq \sqrt{2}\right)$
(see \cite{BBCM22,BH14}),
defined as
\begin{equation}\label{eq-decoration-process}
\mathfrak{D}^{\varrho}(\cdot):=\lim _{t \rightarrow \infty}
\mathsf{ P}
 \left(\sum_{u \in \mathsf{N}_t} \delta_{\mathsf{X}_{u}(t)-\mathsf{M}_t } \in \cdot \mid \mathsf{M}_t \geq   \varrho t\right).
\end{equation}

\subsection{Two-type reducible branching Brownian motions}
In this paper, we study the following two-type reducible branching Brownian motion:
 Type $1$ particles move according to a Brownian motion with diffusion coefficient $\sigma^2$, branch at rate $\beta$ into two children of type $1$ and give birth to particles of type $2$ at rate $1$; type $2$ particles move as a standard Brownian motion and branch at rate $1$ into $2$ children of type $2$, but cannot give birth to offspring of type $1$.
 For $t\geq 0$, we use $N_{t}$ to denote the total number of particles alive at time $t$.
  We can further categorize these particles into type $1$ and $2$, represented by $N^{1}_{t}$ and $N^{2}_{t}$ respectively.
 The position of an individual $u \in N_{t}$ is denoted by $X_{u}(t)$.
 The maximum position at time $t$ is represented by $M_{t}= \max_{u \in N_{t}} X_{u}(t)$. Finally, the law that the two-type BBM starts with a type $1$ particle at the origin is denoted by $  \mathbb{P}^{\beta,\sigma^{2}}$.

The extremal value statistics of the two-type system behaves like that of the single-type BBM.
The centered maximum $(M_{t} - m^{\beta,\sigma^{2}}(t) , \mathbb{P}^{\beta,\sigma^{2}})$  converges in law to a random shifted Gumbel distribution, with a proper centering $m^{\beta,\sigma^{2}}(t)$ of the  form
$l(\beta,\sigma^{2}) \, t- s(\beta,\sigma^{2})\, \log t$. Additionally, the extremal process $(\sum_{u \in N_{t}} \delta_{X_{u}(t)-m^{\beta,\sigma^{2}}(t) }, \mathbb{P}^{\beta,\sigma^{2}})$ converges in law to a certain decorated Poisson point process.
However, an intriguing phase transition occurs in the centering $m^{\beta,\sigma^{2}}(t)$ of the  maximum of the two-type BBM (see Table \ref{tab-phase} and Figure \ref{fig-phase}), but not in single type BBMs, due to the  significant contribution of the added  type $2$ particles to the maximum in some situations.

Divide the parameter space $\left(\beta, \sigma^2\right) \in \mathbb{R}_{+}^2$ into three regions (see Figure \ref{fig-phase}):
  \begin{align*}
    \mathscr{C}_I  &=\left\{\left(\beta, \sigma^2\right): \sigma^2>\frac{1}{\beta}1_{\{\beta \leq 1\}}+\frac{\beta}{2 \beta-1}1_{\{\beta>1\}} \right\} ,\\
    \mathscr{C}_{I I} & =\left\{\left(\beta, \sigma^2\right): \sigma^2<\frac{1}{\beta}1_{\{\beta \leq 1\}}+(2-\beta)1_{\{\beta>1\}}\right\} ,\\
    \mathscr{C}_{I I I} & =\left\{\left(\beta, \sigma^2\right): \sigma^2+\beta>2 \text { and } \sigma^2<\frac{\beta}{2 \beta-1}\right\} ;
    \end{align*}
and define $\mathscr{B}_{i,j} = \partial \mathscr{C}_{i} \cap \partial \mathscr{C}_{j} \backslash \{(1,1)\} $ for $i \neq j$ and $i,j \in \{I,II,III\} $.
If $(\beta,\sigma^2) \in \mathscr{C}_{I} $, the order of the maximum  of the two-type process
is the same as that of particles of type $1$ alone,
and  the asymptotic behavior of the extremal process is dominated by the long-time behavior of particles of type 1. If $(\beta,\sigma^2) \in \mathscr{C}_{II} \cup \mathscr{B}_{I,II}  $
the asymptotic behavior of particles of type 2 dominates the extremal process.
If $(\beta,\sigma^2) \in \mathscr{C}_{III} \cup  \partial \mathscr{C}_{III}$, the so-called  anomalous spreading region, the speed of the two-type process is strictly larger than the speeds of both single type particle systems.
Extreme values can only be achieved by descendants of  first-generation type $2$ particles  born during a certain time interval and within a certain space interval.
To present the known results in a clear and precise manner
 we summarize the different regimes of the maximum and extremal process of the two-type BBM in a table.
For  cases $\mathscr{C}_{I},\mathscr{C}_{II},\mathscr{C}_{III} $, we refer to \cite{BM21}.
The case {(1,1)} was discussed in \cite{Belloum22}, and cases  $\mathscr{B}_{I,II},\mathscr{B}_{I,III},\mathscr{B}_{II,III} $ were covered in \cite{MR23}.
 Recall that the family of  decoration laws $(\mathfrak{D}^{\rho}: \rho \geq \sqrt{2})$ are defined in \eqref{eq-decoration-process}.
\begin{table}[h!]
  \centering
    \begin{tabular}{|l | l | l|}
      \hline
    Regime & Correct centering $m^{\beta,\sigma^{2}}(t)$ & Limiting extremal process  \\
    \hline
      $\mathscr{C}_{I}$ & $\sqrt{2 \beta \sigma^2}t-\frac{3}{2\theta }\log t $ &  $\mathrm{DPPP}(C \mathsf{Z}^{\beta,\sigma^2}_{\infty}e^{-\theta x } \dif x, \mathfrak{D}_{(I)}) $\\

      $\mathscr{C}_{II} \cup \mathscr{B}_{I,II}$  & $\sqrt{2}t-\frac{3}{2\sqrt{2}}\log t$ &  $\mathrm{DPPP}(C \bar{Z}^{\beta,\sigma^2}_{\infty}e^{-\sqrt{2}x } \dif x, \mathfrak{D}^{\sqrt{2}})$ \\

      $\mathscr{C}_{III}$ & $v^{*} t =  \frac{\beta-\sigma^2}{\sqrt{2(\beta-1)(1-\sigma^2)}}t$  &  $\mathrm{DPPP}(C \mathsf{W}^{\beta,\sigma^2}_{\infty}(\theta^{*})e^{-\theta^{*}x } \dif x, \mathfrak{D}^{\theta^{*}}) $  \\

      $\mathscr{B}_{II,III}$ &  $\sqrt{2}t-\frac{1}{2\sqrt{2}}\log t $   & $\mathrm{DPPP}(C \mathsf{W}^{\beta,\sigma^{2}}_{\infty}(\sqrt{2}) e^{-\sqrt{2}x} \dif x, \mathfrak{D}^{\sqrt{2}})$  \\
      $\mathscr{B}_{I,III} \cup \{(1,1)\} $ &  $\sqrt{2 \beta \sigma^{2}} t-\frac{1}{2\theta }\log t $   & $\mathrm{DPPP}(C \mathsf{Z}^{\beta,\sigma^{2}}_{\infty}e^{-\theta x} \dif x, \mathfrak{D}^{\theta})$ \\
  \hline
   \end{tabular}
  \caption{Five regimes of limiting behavior of $(\sum_{u \in N_{t}} \delta_{X_{u}(t)-m^{\beta,\sigma^{2}}(t) }, \mathbb{P}^{\beta,\sigma^{2}})$. The decoration process $\mathfrak{D}_{(I)}$
  was obtained implicitly in  \cite[Theorem 1.1]{BM21}. The random variable $\bar{Z}^{\beta,\sigma^{2}}_{\infty}$  is a composition of derivative martingale and additive martingale, see \cite[Lemma 5.3]{BM21}. Also,  $\theta^{*} = \sqrt{2(\beta-1)/(2-\sigma^{2})}$.}
  \label{tab-phase}
  \end{table}
One can also see Figure \ref{fig-phase} for a more visual representation of these results.
Note that the leading coefficient $l(\beta,\sigma^{2})$ is a continuous function of $(\beta,\sigma^{2})$. However the subleading coefficient $s(\beta,\sigma^{2})$  exhibits discontinuity.
 More interestingly,
 a double jump in the maximum is observed   when $(\beta,\sigma^{2})$ crosses the boundary of the anomalous spreading region $\mathscr{C}_{III}$.
Inspired by papers \cite{Bh20} and \cite{KS15},
we further study the apparent discontinuities in the maximum
that occur when
the parameters $(\beta,\sigma^{2})$ cross  the boundary of $\mathscr{C}_{III}$.
For this, we assume that the parameters $(\beta_{t},\sigma_{t}^{2})$ depend on the time horizon $t$ in an explicit way and approach the boundary of $\mathscr{C}_{III}$ appropriately.
Then we show  that the logarithmic correction for  the maximum now smoothly interpolates between
the normal spreading case ($\mathscr{C}_{I}$ or $\mathscr{C}_{II}$), the boundary case ($\mathscr{B}_{I,III}$, $\{(1,1)\}$ or $\mathscr{B}_{I,III}$) and
the anomalous spreading case $\mathscr{C}_{III}$, shown in the following Figure \ref{fig-phase-2}. (In Figure \ref{fig-phase-2} we only give the negative coefficient of the logarithmic correction, since the leading coefficient changes continuously.)

 \begin{figure}[htbp]
  \centering
 \includegraphics[width=\linewidth]{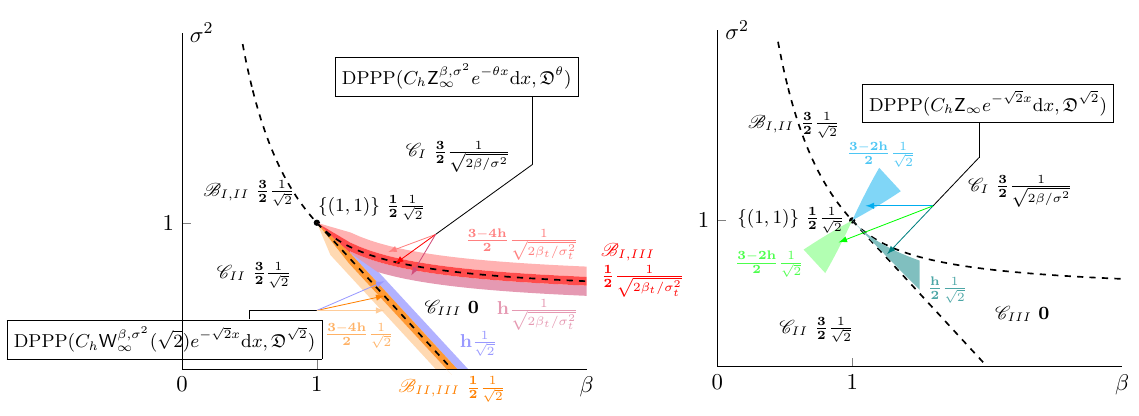}
  \caption{Intermediate phases between normal and anomalous spreading of two-type BBM.
    In the \textcolor{red}{red regime around $\mathscr{B}_{I,III}$} and   \textcolor{orange}{orange regime around  $\mathscr{B}_{II,III}$}
  (roughly speaking, this is the case that the $\mathrm{dist}((\beta_{t},\sigma^{2}_{t}), \mathscr{B}_{I,III})$ and  $\mathrm{dist}( (\beta_{t},\sigma^{2}_{t}), \mathscr{B}_{II,III}) $ respectively, is of order $o(1/\sqrt{t})$), the perturbation is negligible so that
everything is the same as on the boundary $\textcolor{red}{\mathscr{B}_{I,III}}$ and  $\textcolor{orange}{\mathscr{B}_{II,III}}$ respectively.
In the southeast anomalous spreading regime $\mathscr{C}_{III}$, the coefficient of the log-correction term is zero.
 In the two regimes corresponding to
$\textcolor{purple!70}{\mathrm{dist}((\beta_{t},\sigma^{2}_{t}), \mathscr{B}_{I,III})=\Theta(t^{-h})}$ and   $\textcolor{blue!70}{\mathrm{dist}( (\beta_{t},\sigma^{2}_{t}), \mathscr{B}_{II,III})= \Theta(t^{-h})}$  respectively, with
$h \in (0,1/2)$ and  $(\beta_{t},\sigma^{2}_{t}) \in \mathscr{C}_{III}$,
the coefficient interpolates continuously between the surrounding regimes.
In the normal spreading regimes $\mathscr{C}_{I}$ and $\mathscr{C}_{II}$, the coefficient of log-correction term is the same as in a single-type-$1$ BBM and  a single-type-$2$ BBM respectively. In the two regimes determined by \textcolor{red!50}{$\mathrm{dist}((\beta_{t},\sigma^{2}_{t}), \mathscr{B}_{I,III})=\Theta(t^{-h})$, $(\beta_{t},\sigma^{2}_{t}) \in \mathscr{C}_{I}$}
and \textcolor{orange!50}{$\mathrm{dist}((\beta_{t},\sigma^{2}_{t}), \mathscr{B}_{II,III})=\Theta(t^{-h})$, $(\beta_{t},\sigma^{2}_{t}) \in \mathscr{C}_{II}$} respectively with $h \in (0,1/2)$,
the coefficient interpolates continuously between the surrounding regimes.
Similar stories happen when $(\beta_{t},\sigma^{2}_{t})$ approaches $(1,1)$ via 
 certain curves in the highlighted regions 
 on the right figure, 
but now the order of the threshold becomes $1/t$.}
 \label{fig-phase-2}
  \end{figure}

Before presenting our results, we provide a very simple example to illustrate the idea.
Consider the function $f_{t}(x)= x^{t}$ for $x \geq 0$, $t>0$.
Clearly for fixed $x$, as $t \to \infty$, it holds that  $f_{t}(x) \to  0$ if $x<1$,
$f_{t}(x) \to 1$ if $x=1$ and $f_{t}(x) \to \infty $ if $x>1$. Hence $f$ has a double jump at $x=1$. To get a continuous phase transition, we let $x$ depend on $t$ and approach the critical point $1$ appropriately.
We define $x_{t,h}= 1 - \frac{h}{t}$, where $h$ stands for the proximity of $x_{t,h}$ and $1$.
 Then $\lim_{t\to \infty} f_{t}(x_{t,h}) = e^{-h}$
 which continuously interpolates between $1$ and $0$  as $h$ runs over $(0,\infty]$. Similarly by letting $x_{t}=1+ \frac{h}{t}$, $\lim_{t \to \infty} f(x_{t,h})$ interpolates continuously between $1$ and $\infty$.
 Our main results bear similarities to this example, but achieving the same goal in our problem is nontrivial and poses greater challenges.

\subsection{Main results}

As suggested by the previous simple example,  we fix a time horizon $t>0$ and run a two-type reducible BBM $\{X_{u}(s):u \in N_{s}, s \leq t \}$ up to time $t$ under the law $\mathbb{P}^{\beta_{t},\sigma^{2}_{t}}$. That is,  during time $s \in [0,t]$,  type $1$ particles have branching rate $\beta_{t}$ and diffusion coefficient $\sigma^{2}_{t}$,
 while type $2$ particles are standard.
 We need  to find parameters $(\beta_{t},\sigma_{t}^{2})$ that properly approximate a given point $(\beta,\sigma^{2})$ on the boundary of the anomalous spreading region  $\mathscr{C}_{III}$. To do this, we introduce our choice for approximations.

Given parameters $(\beta,\sigma^{2}) \in \mathscr{B}_{I,III} =\{   \frac{1}{\beta}+ \frac{1}{\sigma^{2}} = 2, \beta > 1 \}$ and $h \in (0, \infty]$,
we say $(\beta_{t},\sigma^{2}_{t})_{t>0}$ is
an $h$-admissible approximation  for  $(\beta,\sigma^{2})$ from $\mathscr{C}_{III}$,
denoted by $ (\beta_{t},\sigma^{2}_{t})_{t>0} \in  \mathscr{A}^{h,+}_{(\beta,\sigma^2)} $,
 if
\begin{equation}\label{A13+}
 (\beta_{t},\sigma^2_{t}) \in \mathscr{C}_{III}  \ , \ \frac{1}{\beta_{t}} + \frac{1}{\sigma^2_{t}}=2+ \frac{1}{t^{h}}    \text{ for large } t   \text{ and }(\beta_{t},\sigma^{2}_{t}) \to (\beta,\sigma^2) \text{ as } t \to \infty.
\end{equation}
(Here $h=\infty$ we use the notation $\frac{1}{\infty} = 0$).
Similarly, we say $(\beta_{t},\sigma^{2}_{t})_{t>0}$ is
an $h$-admissible approximation  for  $(\beta,\sigma^{2})$ from $\mathscr{C}_{I}$,
 denoted by
  $ (\beta_{t},\sigma^{2}_{t})_{t>0} \in  \mathscr{A}^{h,-}_{(\beta,\sigma^2)} $, if
\begin{equation}\label{A13-}
  (\beta_{t},\sigma^2_{t}) \in \mathscr{C}_{I}  \ , \  \frac{1}{\beta_{t}} + \frac{1}{\sigma^2_{t}}=2- \frac{1}{t^{h}}    \text{ for large } t   \text{ and }(\beta_{t},\sigma^{2}_{t}) \to (\beta,\sigma^2) \text{ as } t \to \infty.
 \end{equation}

Throughout this paper, for given $(\beta_{t},\sigma_{t}^{2})$, we set
\begin{equation*}
  \theta_{t}=  \sqrt{2 \beta_{t}/ \sigma_{t}^{2}}, \quad v_{t}= \sqrt{ 2 \beta_{t} \sigma_{t}^{2}} \quad   \text{ and }  \ v^{*}_{t} =\frac{\beta_{t}-\sigma^2_{t}}{\sqrt{2(\beta_{t}-1)(1-\sigma^2_{t})}} .
\end{equation*}

\begin{theorem} \label{thm-13-approximate}
  Let  $(\beta,\sigma^2) \in\mathscr{B}_{I,III}$, $h>0$. Define
   \begin{align*}
     m^{1,3}_{h,+}(t)  =  v^{*}_{t} t - \frac{\min\{h,1/2\} }{\theta_{t}} \log t   \ ; \
     m^{1,3}_{h,-}(t)  = v_{t}  t - \frac{3-4 \min\{h,1/2\} }{2\theta_{t}} \log t  .
     \end{align*}
 Then for  $(\beta_{t},\sigma^{2}_{t})_{t>0} \in \mathscr{A}^{h,\pm}_{(\beta,\sigma^2)}$,
 defined in \eqref{A13+} and \eqref{A13-},
 we have      \begin{equation*}
     \lim_{t \to \infty} \left(
       \sum_{u \in N_t} \delta_{X_u(t)-   m^{1,3}_{h,\pm}(t) } , \mathbb{P}^{\beta_{t},\sigma^{2}_{t}} \right)  = \operatorname{DPPP}\left(C_{h,\pm} \theta \mathsf{Z}_{\infty}^{\beta,\sigma^2}  e^{-\theta x}\dif  x, \mathfrak{D}^{\theta}\right)  \text{ in law},
    \end{equation*}
    for some constants $C_{h,\pm} $ depending only on
    $h$ and $(\beta, \sigma^2)$.
   \end{theorem}

\begin{remark}\label{rmk-interpolation-1}
   Recall Table \ref{tab-phase}.
   For $(\beta,\sigma^2) \in\mathscr{B}_{I,III}$ and $(\beta_{t}, \sigma_{t}^2) \in \mathscr{A}^{h,\pm}_{(\beta,\sigma^2)}$, Theorem \ref{thm-13-approximate} shows that, for $h \geq \frac{1}{2}$, the perturbation is so small  that
  the limiting behavior of the extreme values of BBM under $\mathbb{P}^{\beta_{t},\sigma^{2}_{t}}$ and  $\mathbb{P}^{\beta,\sigma^{2}}$ (i.e., no perturbation) are the same!
   For $(\beta_{t}, \sigma_{t}^2) \in \mathscr{A}^{h,+}_{(\beta,\sigma^2)}$,
   as  $h$ changes from $\frac{1}{2}$ to $0$,
   the coefficient for the log correction for the maximum  changes smoothly from $1$ (corresponding to the regime $\mathscr{B}_{I,III}$) to  $0$ (corresponding to the regime $\mathscr{C}_{III}$).   For $(\beta_{t}, \sigma_{t}^2) \in \mathscr{A}^{h,-}_{(\beta,\sigma^2)}$, as
   $h$ changes from $\frac{1}{2}$ to $0$,
    the coefficient  changes smoothly from $1$ (corresponding to the regime $\mathscr{B}_{I,III}$) and  $3$ (corresponding to the regime $\mathscr{C}_{I}$).
   \end{remark}

Given  $(\beta,\sigma^2) \in \mathscr{B}_{II,III} = \{ \beta+\sigma^{2}=2,\beta > 1 \} $ and $h \in (0,\infty]$, we say $(\beta_{t},\sigma^{2}_{t})_{t>0}$ is an  $h$-admissible approximation  for  $(\beta,\sigma^{2})$ from $\mathscr{C}_{III}$,
 denoted by
   $ (\beta_{t},\sigma^{2}_{t})_{t>0} \in
  \mathscr{A}^{h,+}_{(\beta,\sigma^2)} $,  if
  \begin{equation}\label{A23+}
    \beta_{t} + \sigma_{t}^{2} =2+ \frac{1}{t^{h}}   \ , \ (\beta_{t},\sigma^2_{t}) \in \mathscr{C}_{III} \text{ for large } t   \text{ and }(\beta_{t},\sigma^{2}_{t}) \to (\beta,\sigma^2) \text{ as } t \to \infty.
  \end{equation}
We say $(\beta_{t},\sigma^{2}_{t})_{t>0}$ is an  $h$-admissible approximation  for  $(\beta,\sigma^{2})$ from
$\mathscr{C}_{II}$, denoted by $ (\beta_{t},\sigma^{2}_{t})_{t>0} \in  \mathscr{A}^{h,-}_{(\beta,\sigma^2)} $, if
  \begin{equation}\label{A23-}
    \beta_{t} + \sigma_{t}^{2} =2- \frac{1}{t^{h}}   \ , \ (\beta_{t},\sigma^2_{t}) \in \mathscr{C}_{II} \text{ for large } t   \text{ and }(\beta_{t},\sigma^{2}_{t}) \to (\beta,\sigma^2) \text{ as } t \to \infty.
  \end{equation}

\begin{theorem} \label{thm-23-approximate}
Let $(\beta,\sigma^2) \in\mathscr{B}_{II,III}$ and $h \in (0,\infty]$.  Define
  \begin{equation*}
     m^{2,3}_{h,+}(t)   := v^{*}_{t} t - \frac{ \min \{h,1/2\} }{\sqrt{2}} \log t      \ ; \
     m^{2,3}_{h,-}(t)   := \sqrt{2} t - \frac{3-4  \min \{h,1/2\} }{2\sqrt{2}} \log t  .
  \end{equation*}
 Then for   $(\beta_{t},\sigma^{2}_{t})_{t>0} \in
 \mathscr{A}^{h,\pm}_{(\beta,\sigma^2)}$,
 defined in \eqref{A23+} and \eqref{A23-},
 we have   \begin{equation*}
  \lim_{t \to \infty} \left(
 \sum_{u \in N_t} \delta_{X_u(t)-   m^{2,3}_{h,\pm}(t) } , \mathbb{P}^{\beta_{t},\sigma^{2}_{t}} \right)   = \operatorname{DPPP}\left(C_{h,\pm} \sqrt{2}\mathsf{W}_{\infty}^{\beta,\sigma^2}(\sqrt{2}) e^{-\sqrt{2} x}\dif  x, \mathfrak{D}^{\sqrt{2}}\right) \text{ in law},
 \end{equation*}
 for some constants $C_{h,\pm} $ depending only on $h$ and $(\beta, \sigma^2)$.
\end{theorem}

Theorem \ref{thm-23-approximate} has a  similar
 explanation as in Remark \ref{rmk-interpolation-1}.

 Finally,  we introduce the   $h$-admissible approximation for $(1,1)$ from $\mathscr{C}_{I}, \mathscr{C}_{II}$, and $\mathscr{C}_{III}$ as follows respectively.
  \begin{itemize}
    \item Let $\mathscr{A}^{h,1}_{(1,1)} $ be the collection of all $(\beta_{t},\sigma^2_{t})_{t > 0}$ such that   $\frac{1}{\beta_{t}}+\frac{1}{\sigma^2_{t}}=2-t^{-h}$, $\beta_{t}= \sigma_{t}^2 $ for large $t$.
    \item  Let $\mathscr{A}^{h,2}_{(1,1)} $ be the collection of all $(\beta_{t},\sigma^2_{t})_{t > 0}$ such that   $\beta_{t}+ \sigma^2_{t}=2-t^{-h}$,  $\beta_{t}= \sigma_{t}^2 $  for large $t$.
    \item Let  $\mathscr{A}^{h,3}_{(1,1)} $ be the  collection of all $(\beta_{t},\sigma^2_{t})_{t > 0}$  such that  $\beta_{t}+\sigma^{2}_{t}=\frac{1}{\beta_{t}}+\frac{1}{\sigma^2_{t}}=  2+t^{-h}$ for large $t$.
  \end{itemize}

Our next theorem shows that the threshold  for negligible perturbation is $h = 1$, which is twice as much as that in Theorem \ref{thm-13-approximate} and \ref{thm-23-approximate}.
As $h$ changes from $1$ to $0$, the  coefficient
for the log correction
   changes smoothly from $1$ (corresponding to the regime $(1,1)$) to the target regime.

\begin{theorem}\label{thm-11-approximate}
Let $h \in (0,\infty]$. Define
\begin{align*}
    m^{(1,1)}_{h,1}(t) &  = v_{t} t- \frac{3-2 \min\{h,1\}}{2 \sqrt{2} }\log t  \ , \
      m^{(1,1)}_{h,2}(t)   = \sqrt{2}t- \frac{3-2 \min\{h,1\}}{2 \sqrt{2}} \log t , \text{ and } \\
     m^{(1,1)}_{h,3}(t) &  =
   v^{*}_{t}  t - \frac{\min\{h,1\}}{2\sqrt{2}} \log t.
\end{align*}
For  $i=1, 2,3$ and $(\beta_{t},\sigma^{2}_{t})_{t>0} \in \mathscr{A}_{(1,1)}^{h,i}$,
we have
 \begin{equation*}
  \lim_{t \to \infty} \left(
    \sum_{u \in N_t} \delta_{X_u(t)-   m^{(1,1)}_{h,i}(t) } , \mathbb{P}^{\beta_{t},\sigma^{2}_{t}} \right) = \operatorname{DPPP}\left(C_{h,i} \sqrt{2}\mathsf{Z}_{\infty} e^{-\sqrt{2} x}\dif  x, \mathfrak{D}^{\sqrt{2}}\right) \text{ in law},
\end{equation*}
  for some constants $C_{h,i} $ depending only on $ h$.
      \end{theorem}

\begin{remark}
  In the proof of Theorem \ref{thm-13-approximate} (similar for the proofs of Theorem \ref{thm-23-approximate} and Theorem \ref{thm-11-approximate}),
 it is sufficient to prove the convergence of the extremal process  consisting of particles of type $2$. Since by \eqref{eq-converge-to-Gumbel}, the highest type $1$ particles are located at level $v t- \frac{3}{2 \theta} \log t+O_{\P}(1)$, which is
   much below  type $2$ particles.
 Moreover, our results in Theorems \ref{thm-13-approximate}, \ref{thm-23-approximate} and \ref{thm-11-approximate}
  can be strengthened as the
  joint convergence of the extremal
   process
  and its maximum  $(\mathcal{E}_{t}, \max \mathcal{E}_{t})$ to
  $ (\mathcal{E}_{\infty}, \max \mathcal{E}_{\infty})$, where $\mathcal{E}_{\infty}$ is
the limiting point process
  and $\max \mathcal{E}_{\infty}$ is
the supremum of its support.
  The reason is that, by  \cite[Lemma 4.4]{BBCM22},   to prove the  convergence of $ (\mathcal{E}_t, \max \mathcal{E}_t)$,
   it suffices to show the convergence of Laplace functional $\mathbb{E}[e^{-\langle \phi, \mathcal{E}_{t} \rangle}]$ with certain test functions $\varphi \in \mathcal{T}$ introduced in the
    notation convention part below.
  \end{remark}

\textbf{Outline. }
The rest of the article is organized as follows. We discuss our results in the next sub-section, offering some heuristics of our proof and giving relation to coupled F-KPP equations.
In Section 2, we introduce several results on branching Brownian motions that will be needed in our proofs, in particular
some estimates for the Laplace functional of the point process associated to BBM
 and central limit theorems for the Gibbs measures associated to BBM.
In Sections 3, 4 and 5 we give the proofs of our theorems.
In section 3 we prove the case  $ \mathscr{A}^{h,-}_{(\beta,\sigma^2)}$ in Theorem \ref{thm-13-approximate} and the case $\mathscr{A}^{h,1}_{(1,1)}$ in Theorem \ref{thm-11-approximate}.  In section 4 we prove the case  $ \mathscr{A}^{h,+}_{(\beta,\sigma^2)}$ in Theorem \ref{thm-13-approximate},  \ref{thm-23-approximate} and the case $\mathscr{A}^{h,3}_{(1,1)}$ in Theorem \ref{thm-11-approximate}.  In section 5 we prove the case  $ \mathscr{A}^{h,-}_{(\beta,\sigma^2)}$ in Theorem \ref{thm-13-approximate} and the case $\mathscr{A}^{h,2}_{(1,1)}$ in Theorem \ref{thm-11-approximate}.

\textbf{Notation convention. }
Throughout this article $C$ (also $C_{h,+},   C_{h,-},\cdots$) are positive constants whose value may change from line to line.
Let  $\mathcal{T} $  be the set of  functions $\varphi \in C_{b}^{+}(\mathbb{R})$ such that $\inf \mathrm{supp}( \varphi)>-\infty$  and for some $a \in \mathbb{R}$,  $ \varphi(x) \equiv$ some positive constant for all $x > a$.
$\mathcal{T}$ will  serve as  test functions in the Laplace functional
(see \cite[Lemma 4.4]{BBCM22}). For two quantities $f$ and $g$, we write $f \sim g$ if $\lim f/g= 1$.  We write $f  \lesssim g$ if there exists a constant $C>0$ such that $f \leq C g $. We write $f  \lesssim_\lambda g$ to stress that the constant $C$ depends on parameter $\lambda$. We use the standard notation $\Theta(f)$  to denote a non-negative quantity such that there exists constant  $c_{1}, c_{2}>0$ such  that $c_{1} f \leq \Theta(f)  \leq c_{2}f $.
When this is no ambiguity, we use $\mathbb{P}$ and $\mathbb{E}$ to denote  $\mathbb{P}^{\beta_{t},\sigma^{2}_{t}}$  and $\mathbb{E}^{\beta_{t},\sigma^{2}_{t}}$, respectively.
We always use the front $\texttt{mathsf}$ to denote the probability or quantities related to single-type branching Brownian motion, like  $\mathsf{P},\mathsf{E}, \mathsf{X}_{u}, \mathsf{W}, \mathsf{Z}$ etc.
The probability  and expectation related to Brownian motion are denoted as  $\mathbf{P}$ and $\mathbf{E}$.

\subsection{Discussion of our results}

\subsubsection{Heuristics for localization of paths of extremal particles}

For each type $2$ particle $u \in N^{2}_{t}$, we define  $T_{u}$ as the time at which the oldest ancestor of type $2$ of $u$ was born. In other words, $T_{u}$ is the ``type transformation time"  of $u$. For convenience, we set $T_{u}=t$ for  $u \in N^{1}_{t}$.

We restate here the optimization problem introduced in \cite[Section 2.1]{BM21} (see also Biggins \cite{Biggins10}).
 For $p\in[0,1]$, let $\mathcal{N}_{p,a,b}(t)$ be the expected number of particles at time $t$
 that have speed $a$ before time $T_{u}\approx pt$ and speed $b$ after time $pt$ (under the law $\mathbb{P}^{\beta,\sigma^{2}}$).
Note that these particles are at level $[pa+(1-p)b]t$.
By first moment computations,
  $ \mathcal{N}_{p,a,b}(t)=\exp \left\{[(\beta -\frac{a^2}{2 \sigma^2})p+(1-\frac{b^2}{2})(1-p) ] t + o(t) \right\}$.  So the speed of
  the two-type
  BBM should be the maximum of $pa+(1-p)b$ among all the parameter $p,a,b$ such that $\left(\beta- \frac{a^2}{2 \sigma^2} \right)p  \geq 0$ and $\mathcal{N}_{p,a,b}(t) \geq 1$. That is,
\begin{equation}\label{eq-optimization-problem}
\begin{aligned}
    v^{*} = \max  \bigg\{p a+(1-p) b: & p \in [0,1],   \left(\beta-\frac{a^2}{2 \sigma^2}\right)p \geq 0,   \\
  & \left(\beta-\frac{a^2}{2 \sigma^2} \right)p+ \left(1-\frac{b^2}{2} \right)(1-p) \geq 0\bigg\} .
\end{aligned}
\end{equation}
Denote  by $\left(p^*, a^*, b^*\right)$  the maximizer of this optimization problem.
 If $\left(\beta, \sigma^2\right) \in \mathscr{C}_I$, then $p^*=1$, $a^{*}=v$, and $v^*=v$; if $\left(\beta, \sigma^2\right) \in \mathscr{C}_{II}$,  then  $p^*=0, b^{*}=\sqrt{2}$ and   $v^*=\sqrt{2}$; if   $\left(\beta, \sigma^2\right) \in \mathscr{C}_{III}$, then
 $p^{*}=\frac{\sigma^2+\beta-2}{2\left(1-\sigma^2\right)(\beta-1)}$, $b^{*}=\sqrt{2 \frac{\beta-1}{1-\sigma^2}}$, $a^{*}=\sigma^2 b^{*}$, and  $v^{*}= \frac{\beta-\sigma^2}{\sqrt{2\left(1-\sigma^2\right)(\beta-1)}} $.

 Inspired by the heuristics above, we are going to do
 some refined computations that provide more  precise  predictions for  localization of
  extremal particles, under the law $\mathbb{P}^{\beta_{t},\sigma^{2}_{t}}$.  To avoid duplication,
 we only do this under  the setting of Theorem \ref{thm-13-approximate}, i.e.,  $(\beta,\sigma^2)\in \mathscr{B}_{I,III}$  and $(\beta_{t},\sigma^2_{t})_{t>0}\in \mathscr{A}^{h,\pm}_{(\beta,\sigma^{2})} $. Under the setting of Theorem \ref{thm-23-approximate}, \ref{thm-11-approximate}, one can use a similar argument.

\paragraph*{The case $(\beta_{t},\sigma^2_{t})_{t>0}\in \mathscr{A}^{h,-}_{(\beta,\sigma^{2})}$. }

According to the  optimization problem \eqref{eq-optimization-problem}, $M_{t} \approx v_{t} t$ and   each individual $u \in N^{2}_{t}$ near the maximum should satisfy $T_u \approx t$.
The expected  number of type $1$ particles
 that are at level $v_{t} s-\delta(t)$ (where  $\delta(t)$ will be determined later)  at time $s=t-o(t)$
 is roughly $ e^{ \theta_{t} \delta(t) - \frac{\delta(t)^2}{2 \sigma^2_t s} +O(\log t)} $.
The probability that a typical particle of type $2$ has  a descendant at level $v_{t}(t-s)+\delta(t)$ at time $t-s$
is $ e^{-\left[ (\frac{v^2}{2}-1)(t-s)+  v_{t} \delta(t)+  \frac{\delta(t)^2}{2(t-s)}\right]+O(\log t)} $.
 Hence there are around
  \begin{equation}\label{eq-number-2}
    \exp\left\{  -\left[ \left(\frac{v_{t}^2}{2}-1\right)(t-s) + \frac{\delta(t)^2}{2 \sigma^2_t s}+\frac{\delta(t)^2}{2(t-s)}   \right] + (\theta_{t}-v_{t})\delta(t) +O(\log t) \right\}
  \end{equation}
   particles of type $2$ at level $v_{t} t $ at time $t$. In order for the limit of the the quantity in  \eqref{eq-number-2} to be non-zero as $  t\to\infty$,  using the prior knowledge $s \sim t$, we first have to  ensure that
  $\delta(t)$ has the same order as  $t-s$   or $\frac{\delta(t)^2}{t-s}$. So
  we get $\delta(t) =\Theta(t-s)$.
  We also need  to ensure that $\frac{\delta(t)^2}{2 \sigma^2_t s} = O(1)$, which implies  $\delta(t) =  O( \sqrt{t})$.
 Letting $\delta(t)=\lambda (t-s)$, we can rewrite \eqref{eq-number-2} as
  \begin{equation*}
\exp\left\{  \left[ -\frac{1}{2} [\lambda - (\theta-v_{t}) ]^{2} + \frac{ (\theta_{t}-v_{t})^{2} }{2}
   - \left( \frac{v_{t}^{2}}{2}-1 \right)          \right] (t-s)   +    O(\log t)  \right\}.
  \end{equation*}
As $\frac{1}{\beta_{t}}+\frac{1}{\sigma^{2}_{t}}=2-t^{-h}$, we have $
  (\theta_{t}-v_{t})^{2} -  (v_{t}^{2}-2) = 2 \beta_{t}(\frac{1}{\sigma^{2}_{t}}-2) + 2 =   - \frac{2\beta_{t}}{t^{h}}$. Then \eqref{eq-number-2} becomes
  \begin{equation*}
 \exp\left\{  \left[ -\frac{1}{2} [\lambda - (\theta-v_{t}) ]^{2} +  \frac{\beta_{t}}{t^{h}}    \right] (t-s) +  O(\log t)  \right\}.
 \end{equation*}
 To guarantee
 $\left[ -\frac{1}{2} [\lambda - (\theta-v_{t}) ]^{2} +  \frac{\beta_{t}}{t^{h}}    \right] (t-s) \geq 0$, we need $\lambda = (\theta_{t}-v_{t} )   $ and  $ t-s=O(t^{h})$.
In other words, the extremal particle $u \in N^2_{t}$ should satisfy
\begin{equation*}
  t-T_{u}= O(t^{h \wedge \frac{1}{2}}) \  \text{ and } \ X_{u}(T_{u}) \approx v_{t} T_{u}- (\theta_{t}-v_{t})(t-T_{u}) .
\end{equation*}

\paragraph*{The case $(\beta_{t},\sigma^2_{t})_{t>0}\in \mathscr{A}^{h,+}_{(\beta,\sigma^{2})}$. }
According to the   optimization problem \eqref{eq-optimization-problem}, $M_{t} \approx v^{*}_{t} t$ and each individual $u \in N^{2}_{t}$ near the maximum should satisfy $T_u \approx  p^{*}_{t} t$.
The expected  number of type $1$ particles
that are at level $a^{*}_{t} s- \delta(t)$ at time $s=p^{*}_{t} t-o(t)$
is roughly $ e^{ - \frac{ [a_{t}^{*} s - \delta(t) ]^{2} }{2 \sigma_{t}^{2} s} + \beta_{t} s +O(\log t)} $.
The probability that a typical particle of type $2$ has a descendant at level $v^{*}_{t} t - a^{*}_{t} s  + \delta(t)$ at time $t-s$
is $ e^{- \frac{ [ v^{*}_{t} t - a^{*}_{t} s  + \delta(t)  ]^{2} }{2 (t-s) } + (t-s)  +O(\log t)} $.
Hence there are around
\begin{equation}\label{eq-number-3}
  \exp\left\{  \beta_{t} s - \frac{ [a_{t}^{*} s - \delta(t) ]^{2} }{2 \sigma_{t}^{2} s}  +   (t-s)   - \frac{ [ v^{*}_{t} t - a^{*}_{t} s  + \delta(t)  ]^{2} }{2 (t-s) }       +O(\log t) \right\}
\end{equation}
particles of type $2$ at level $v^{*}_{t} t $ at time $t$. Let $ s = p^{*}_{t} t - \varepsilon(t)$. We have $ v^{*}_{t} t -a^{*}_{t} s= b^{*}_{t}(1-p^{*}_{t})t+a^{*}_{t} \varepsilon (t)=b^{*}_{t}(t-s)-(a^{*}_{t}-b^{*}_{t}) \varepsilon (t)  $.  Hence
   \begin{equation*}
    \frac{(v^{*}_{t} t -a^{*}_{t}s +\delta(t)   )^2}{2(t-s)}=\frac{(b^{*}_{t})^2}{2}(t-s) +  b^{*}_{t} \delta(t) - b^{*}_{t}(a^{*}_{t}-b^{*}_{t}) \varepsilon  (t)   + \frac{[\delta(t)-(a^{*}_{t}-b^{*}_{t})\varepsilon  (t)]^2}{2(t-s)}.
   \end{equation*}
  Applying the facts
  that  $ [\beta_{t} -\frac{(a^{*}_{t}) ^2}{2\sigma_{t}
  ^2 } ]p^{*}_{t}  + [1-\frac{(b^{*}_{t})^2}{2}](1-p^{*}_{t} )  =0 $  and $a^{*}_{t}=\sigma_{t}^2 b_{t}^{*}$, we get
   \begin{align*}
   & (\beta_{t}-\frac{(a^{*}_{t})^2}{2\sigma^{2}_{t}})s +\frac{a^{*}_{t}}{\sigma_{t}^{2}} \delta(t) - \frac{\delta(t)^{2}}{2\sigma_{t}^{2} s}  + (t-s)-\frac{[ v^{*}_{t} t -a^{*}_{t}s + \delta(t)   ]^2}{2(t-s)}   \\
   &=  (\beta_{t} -\frac{(a^{*}_{t})^2}{2 \sigma_{t}^2})p^{*}_{t} t +  (1-\frac{(b^{*}_{t})^2}{2} )(1-p^{*}_{t}) t \\
    & \qquad  \qquad  \qquad  - \left(\beta_{t}  -1+\frac{\sigma_{t}^2 -1}{2} (b^{*}_{t})^2 \right) \varepsilon (t)  -  \frac{[\delta(t)-(a^{*}_{t}-b^{*}_{t})\varepsilon  (t)]^2}{2(t-s)}  \\
   &=   -  \frac{[\delta(t)-(a^{*}_{t}-b^{*}_{t})\varepsilon  (t)]^2}{2(t-s)} - \frac{\delta(t)^{2}}{2\sigma_{t}^{2} s}.
   \end{align*}
Hence \eqref{eq-number-3} equals to  $\exp \{  -  \frac{[\delta(t)-(a^{*}_{t}-b^{*}_{t})\varepsilon  (t)]^2}{2(t-s)} - \frac{\delta(t)^{2}}{2\sigma_{t}^{2} s} + O(\log t)  \}$.
Thus we need $|\delta(t)|= O(\sqrt{t})$ and $| \delta(t)-(a^{*}_{t}-b^{*}_{t})\varepsilon  (t)| =O(\sqrt{t-s})$. Hence $|\varepsilon  (t) |=O(\sqrt{t})$.  In other words,
\begin{equation*}
  p^{*}_{t} t-T_{u} = O (\sqrt{t}) \  \text{ and } \ X_{u}(T_{u}) \approx a_{t}^{*} T_{u}- (a^{*}_{t}-b^{*}_{t})(t-T_{u}) .
\end{equation*}

\subsubsection{Application in F-KPP equations}
A multitype BBM, like standard BBM, is associated to
an F-KPP reaction diffusion equation.
For more details,  we refer to \cite[Section 2.3]{BM21}.
Specifically,  let $t>0$, and  $f, g: \mathbb{R} \rightarrow[0,1]$ be measurable functions.
  We define for all $x \in \mathbb{R}$ and $s\le t$ :
\begin{align*}
u(s, x) &
 =\mathbb{E}^{\beta_{t},\sigma^{2}_{t}}\left(\prod_{u \in N_s^1} f\left(X_u(t)+x\right) \prod_{u \in N_s^2} g\left(X_u(t)+x\right)\right), \\
 v(s,x) &
 = \mathsf{E} \left(  \prod_{u \in \mathsf{N}_s} g\left( \mathsf{X}_u(t)+x\right)\right), \text{ where }  (\mathsf{X}_{u}(t), u \in \mathsf{N}_{t}, \mathsf{P} ) \text{ is a standard BBM}.
\end{align*}
 Then  $(u, v)$ is a solution of the following coupled F-KPP equation
\begin{equation}\label{eq-coupled-pde}
   \left\{ \begin{aligned}
 & \partial_s u    =\frac{\sigma^2_{t}}{2} \Delta u  - \beta_{t}  u(1-u)-  u(1-v),  \  \
 0<s \leq t, \\
& \partial_s v  =\frac{1}{2} \Delta v-v(1-v),
\quad s>0,\\
& v(0, x)  =g(x)  \ , \   u(0, x)=f(x).
\end{aligned} \right.
\end{equation}

Our main results give the the existence of a function $m_t$ such that
(with good initial
functions $f, g$, e.g., $f= g =1$ on $(-\infty, -A]$ and $f =g  =0$ on $[A,\infty)$)    for all $x \in \mathbb{R}$,
\begin{equation*}
\lim _{t \rightarrow \infty}
 \left(u\left(t, x-m_t\right), v\left(t, x-m_t\right)\right)
=\left(w_1(x), w_2(x)\right),
\end{equation*}
where $(w_{1}, w_{2})$ is a solution of the the coupled  ordinary differential equations (ODEs):
\begin{equation}\label{eq-coupled-ode}
   \left\{ \begin{aligned}
& \frac{\sigma^2}{2} w_{1}''-\mathrm{c} w_{1}'   - \beta  w_{1}(1-w_{1})-  w_{1}(1-w_{2})=0,\\
& \frac{1}{2} w_{2}''-\mathrm{c} w_{2}'-w_{2}(1-w_{2})=0,
\end{aligned} \right.
\end{equation}
with $(\beta,\sigma^{2})=\lim_{t \to \infty} (\beta_{t},\sigma^{2}_{t})$ and $\mathrm{c}=\lim_{t \to \infty} m_{t}/t$.
In fact, $m_t$ is defined as follows:
\begin{itemize}
  \item if $\left(\beta, \sigma^2\right) \in \mathscr{B}_{I,III}$,
  $\left(\beta_{t}, \sigma^2_{t} \right)_{t>0} \in \mathscr{A}^{h,-}_{(\beta,\sigma^{2})}$, then $m_t=\sqrt{2 \beta_{t} \sigma_{t}^2} t-\frac{3-4 \min\{ h,1/2\}}{2 \sqrt{2 \beta_{t} / \sigma_{t}^2}} \log t$;
  if $(\beta,\sigma^{2})=(1,1)$ and   $\left(\beta_{t}, \sigma^2_{t} \right)_{t>0} \in \mathscr{A}^{h,1}_{(1,1)}$ then  $m_{t}= \sqrt{2 \beta_{t} \sigma_{t}^2} t-\frac{3-2 \min\{ h,1\}}{2 \sqrt{2 \beta_{t} / \sigma_{t}^2}} \log t   $;
  \item if $\left(\beta, \sigma^2\right) \in \mathscr{B}_{II,III}$ and
  $\left(\beta_{t}, \sigma^2_{t} \right)_{t>0} \in \mathscr{A}^{h,-}_{(\beta,\sigma^{2})}$, then $m_t=\sqrt{2  } t-\frac{3-4 \min\{ h,1/2\}}{2 \sqrt{2  }} \log t$; if $(\beta,\sigma^{2})=(1,1)$ and $\left(\beta_{t}, \sigma^2_{t} \right)_{t>0} \in \mathscr{A}^{h,2}_{(1,1)}$ then  $m_{t}=\sqrt{2  } t-\frac{3-2 \min\{ h,1\}}{2 \sqrt{2  }} \log t  $;
  \item
  if $\left(\beta, \sigma^2\right) \in \mathscr{B}_{I,III} \cup \in \mathscr{B}_{II,III}$ and
   $\left(\beta_{t}, \sigma^2_{t} \right)_{t>0} \in \mathscr{A}^{h,+}_{(\beta,\sigma^{2})}$,
  then $m_t= v^{*}_{t} t-\frac{  \min\{ h,1/2\}}{  \sqrt{2}} \log t$; if $(\beta,\sigma^{2})=(1,1)$ and $\left(\beta_{t}, \sigma^2_{t} \right)_{t>0} \in \mathscr{A}^{h,3}_{(1,1)}$ then $ m_{t}= v^{*}_{t}  t - \frac{\min\{ h,1\}}{2\sqrt{2}} \log t$.
\end{itemize}
Now we  show that $(w_{1}, w_{2})$ is a solution of \eqref{eq-coupled-ode}.
 By Theorems \ref{thm-13-approximate}, \ref{thm-23-approximate} and \ref{thm-11-approximate}, given $(\beta,\sigma^{2})$, for all $h$-admissible approximation $(\beta_{t},\sigma_{t})_{t>0}$ with $h \in (0,\infty]$, the limit $(w_{1}(x),w_{2}(x))$ are   the same (up to a translation depending on $h$).
So it suffices to consider the case $h=\infty$, i.e., $(\beta_{t},\sigma_{t}^{2})\equiv (\beta,\sigma^{2})$. Applying  the branching property, we have
\begin{equation*}
  u(t,x-m_{t}) = \mathbb{E}^{\beta,\sigma^{2}}\left(\prod_{u \in N_s^1} u\left(t-s,X_u(s)+x-m_{t}\right) \prod_{u \in N_s^2} v\left(t-s,X_u(s)+x-m_{t}\right)\right).
\end{equation*}
Letting $t \to \infty$, since $m_{t}= \mathrm{c} s + m_{t-s}+o(1)$, we get
\begin{equation*}
  w_{1}(x+ \mathrm{c} s)
  = \mathbb{E}^{\beta,\sigma^{2}}\left(\prod_{u \in N_s^1} w_{1}\left(X_u(s)+x \right) \prod_{u \in N_s^2} w_{2}\left(X_u(s)+x\right)\right), \quad s\geq 0;
\end{equation*}
 and similarly $ w_{2}(x+\mathrm{c} s) = \mathsf{E} \left( \prod_{u \in \mathsf{N}_s} w_{2}\left( \mathsf{X}_u(s)+x\right)\right) $.
Then, as the derivation of \eqref{eq-coupled-pde}, using again the argument in \cite[Section 2.3]{BM21},  $(w_{1}(x+\mathrm{c}t),w_{2}(x+\mathrm{c}t))$
solves the coupled F-KPP equation:
\begin{equation*}
   \left\{ \begin{aligned}
 & \partial_t u    =\frac{\sigma^2}{2} \Delta u  - \beta  u(1-u)-  u(1-v),  \  \
\\
& \partial_t v  =\frac{1}{2} \Delta v-v(1-v).
\end{aligned} \right.
\end{equation*}
That is, $(w_{1},w_{2})$ is a traveling wave solution of this coupled PDE;  and  \eqref{eq-coupled-ode} follows.

\section{Preliminary results}

\subsection{Brownian motion estimates}

The following lemma gives an upper bound for the probability that  a Brownian bridge below a straight line.

\begin{lemma}[{\cite[Lemma 2]{Bramson78}}]
\label{lem-bridge-estimate-0}
  Let $(\zeta^{[0,t]}_{s})_{ s \in [0,t]}$ be a Brownian bridge from $0$ to $0$. Let $x_{1}, x_{2} \geq 0$, then
  \begin{equation*}
    \P\left( \zeta_{s} \leq \frac{s}{t} x_{1}+ \frac{t-s}{t}x_{2}, \forall s \in [0,t]\right) = 1 - e^{-\frac{2 x_{1}x_{2}}{t}} \leq \frac{2x_{1}x_{2}}{t}.
  \end{equation*}
\end{lemma}

\subsection{Branching Brownian motion estimates}
Recall that $\{ ( \mathsf{X}^{\beta,\sigma^{2}}_{u}(t), u \in \mathsf{N}_{t})_{t \geq 0}, \mathsf{P} \}$ is
a BBM with branching rate $\beta$ and diffusion coefficient $\sigma^2$.
Let $v= \sqrt{2 \beta \sigma^2}$ and $\theta= \sqrt{\frac{2 \beta}{\sigma^2}}$. Then for some constant $C>0$ there holds
\begin{equation}\label{eq-upper-envelope-0}
 \mathsf{P}\left(\exists s >0, u \in \mathsf{N}_s: \mathsf{X}^{\beta,\sigma^{2}}_u(s) \geq v s +K\right) \leq  C e^{-\theta K}.
  \end{equation}
In fact,  this probability
is comparable with respect to this upper bound, see \cite[Lemma 3.4]{Madaule16tail}.
We state some fundamental results for the standard BBM (i.e., $\beta=\sigma^2=1$) that will be used later. The first one is the tail probability of the maximum of BBM. Applying the
 first moment method, we get  a trivial upper bound:
  for every $y \geq 1$ and  $t >0$,
\begin{equation}\label{eq-BBM-tail-2}
  \mathsf{P}\left( \max_{u \in \mathsf{N_{t}}}\mathsf{X}_{u}(t)  \geq   y \right)  \leq  e^{t} \mathbb{P}( B_{t} \geq  y) \leq \frac{1}{\sqrt{2 \pi } } \frac{\sqrt{t}}{y} e^{ t -\frac{y^2}{2t}}.
 \end{equation}
A much better estimate, especially when $y $  nears $ \mathsf{m}(t):= \sqrt{2}t- \frac{3}{2\sqrt{2}} \log t$,
 was given in \cite[Corollary 10]{ABK12} as follows.   For every $x \geq 1$ and  $t >0$,
\begin{equation*}
\mathsf{P}\left(  \max_{u \in \mathsf{N_{t}}}\mathsf{X}_{u}(t)   \geq \mathsf{m}(t)  + x \right)  \leq C   x   \exp \left(-\sqrt{2} x-\frac{x^2}{2 t} + \frac{3}{2\sqrt{2}} \frac{ x \log(t+1)}{ t}\right).
\end{equation*}
We shall use a slight modification of this inequality  as follows.
\begin{lemma}\label{lem-Max-BBM-tail}
  There exists some constant $C>0$ such that for every $x \geq 1$ and  $t >0$,
\begin{align*}
\mathsf{P}\left(  \max_{u \in \mathsf{N_{t}}}\mathsf{X}_{u}(t)   \geq  \mathsf{m}(t) + x \right)
\leq   C
x   \exp \left(-\sqrt{2} x- \frac{1}{2 t} \left[x- \frac{3}{2\sqrt{2}}\log (t+1) \right]^{2}  \right).
\end{align*}
\end{lemma}
Note that $ \inf\limits_{x \geq 1, t>0} \frac{1}{2t} [x-\frac{3}{2\sqrt{2}} \log (t+1)  ]^2 -  \frac{x^2}{3t} >-\infty$.  By enlarging the constant $C$,  we also have,  for $x \geq 1$ and $t>0$,
\begin{equation}\label{eq-Max-BBM-tail-2}
  \mathsf{P}\left(  \max_{u \in \mathsf{N_{t}}}\mathsf{X}_{u}(t)   \geq \mathsf{m}(t) + x \right)   \leq   C
  x   \exp \left(-\sqrt{2} x-\frac{x^2}{3 t}  \right).
\end{equation}

Secondly, we need some estimates about the Laplace functional of the following point processes  associated with BBM:
\begin{equation*}
	\sum_{u \in \mathsf{N}_t}  \delta_{\mathsf{X}_u(t)-\rho t}    \   \ \text{ for all } \rho \geq \sqrt{2}.
\end{equation*}
When looking at the long-time behavior of the
Laplace functionals of these  point processes, there are two distinct regimes: $\rho = \sqrt{2}$ and $\rho > \sqrt{2}$.

\begin{lemma}[{\cite[Corollary 2.9]{Belloum22},\cite[Lemma 3.7]{BM21}}]
\label{thm-Laplace-BBM-order}
Let  $\varphi \in \mathcal{T}$, $\epsilon>0 $ and $\rho \geq \sqrt{2}$.
Define
\begin{equation}\label{def-Phi-rho}
 \Phi_{\rho}(t,x): =1-\mathsf{E}\left(e^{-\sum_{u \in \mathsf{N}_t} \varphi\left(x+\mathsf{X}_u(t)-\rho t\right)}\right).
\end{equation}

\begin{enumerate}[(i)]
  \item If $\rho = \sqrt{2}$,  for $x \in [-t^{1-\epsilon}, -t^{\epsilon}]$ uniformly
  \begin{equation*}
    \Phi_{\sqrt{2}}(t,x)= (1+o(1))\gamma_{\sqrt{2}}(\varphi)\frac{ (-x )}{t^{3 / 2}}e^{\sqrt{2} x-\frac{x^2}{2 t}} , \text{ as } t \to \infty,
\end{equation*}
 where $\gamma_{\sqrt{2}}(\varphi)= \displaystyle \sqrt{2}C_{\star} \int e^{-\sqrt{2} z}\left(1-\mathbb{E}\left(e^{-\langle\mathfrak{D}^{\sqrt{2}}, \varphi(\cdot+z)\rangle}\right)\right) \dif z$.
\item If $\rho>\sqrt{2}$,
 for $|x| \leq t^{1-\epsilon}$ uniformly
\begin{equation*}
   \Phi_{\rho}(t,x) =(1+o(1))
   \gamma_\rho(\varphi) \frac{e^{\left(1-\rho^2 / 2\right) t}}{\sqrt{t}} e^{\rho x-\frac{x^2}{2 t}}  , \text{ as } t \to \infty,
\end{equation*}
where $ \gamma_\rho(\varphi)=  \displaystyle\frac{C(\rho)   }{\sqrt{2 \pi }}\int e^{-\rho z}\left(1- \mathbf{E} (e^{-\left\langle\mathfrak{D}^\rho, \varphi(\cdot+z)\right\rangle})  \right) \dif z$.
\end{enumerate}
\end{lemma}

In fact part (i) and part (ii) were proved for the case $x=- \Theta(\sqrt{t})$ in
\cite[Corollary 2.9]{Belloum22}  and for the case  $|x|= O(\sqrt{t})$ in \cite[Lemma 3.7]{BM21} respectively. However their proofs still work in our setting. We omit the repetitive proofs here.

Thirdly, we introduce several central limit theorems about the Gibbs measures associated with  standard BBM $\left\{\left(\mathsf{X}_u(t): u \in \mathsf{N}_t\right), \mathsf{P}\right\}$.
Conditioned on BBM at time $t$, we assign each  particle $u \in \mathsf{N}_t$ with probability
\begin{equation*}
  \frac{e^{\lambda \mathsf{X}_u(t)}}{\sum_{u \in \mathsf{N}_t} e^{\lambda \mathsf{X}_u(t)}}.
\end{equation*}
 Hence the additive martingale $\mathsf{W}_t(\lambda)=\sum_{u \in \mathsf{N}_t} e^{\lambda \mathsf{X}_u(t)-\left(\frac{\lambda^2}{2}+1\right) t}$ can be regarded as  a normalized partition function of the Gibbs measure.
The following law of large  numbers is well-known:
 for $0 \leq \lambda < \sqrt{2}$, and bounded continuous function $f$,
\begin{equation*}
  \lim _{t \rightarrow \infty}  \sum_{u \in \mathsf{N}_t} f\left( \frac{\mathsf{X}_u(t)}{t}\right)   e^{\lambda \mathsf{X}_u(t)-\left(\frac{\lambda^2}{2}+1\right) t}  = \mathsf{W}_{\infty}(\lambda) f(\lambda)    \   \text{ in probability,}
  \end{equation*}
  where $\mathsf{W}_{\infty}(\lambda)$ is the limit of the non-negative martingale $\mathsf{W}_{t}$. (See \cite[Proposition 2.5]{MR23} for a proof).
Furthermore,  a central limit theorem holds (see \cite[(1.14)]{Pain18}):   for $\lambda \in(0, \sqrt{2})$   and bounded continuous function $f$,
\begin{equation*}
\lim _{t \rightarrow \infty} \sum_{u \in \mathsf{N}_t} f\left(\frac{\mathsf{X}_u(t)-\lambda t}{\sqrt{t}}\right) e^{\lambda \mathsf{X}_u(t)-t\left(\frac{\lambda^2}{2}+1\right)}=W_{\infty}(\lambda) \int_{\mathbb{R}} f(z) e^{-\frac{z^2}{2}} \frac{\dif z}{\sqrt{2 \pi}}  \  \text{ in probability.}
\end{equation*}
In the following lemma, we generalize this central limit theorem to the case that the parameter $\lambda$ and test function $f$ both depend on $t$ in a certain way.
We   postpone its  proof to Appendix \ref{App-A}.

\begin{lemma}\label{lem-functional-convergence-derivative-martingale-2}   Let $G$ be a non-negative bounded  measurable function with compact support. Suppose $F_t(z)=$ $G\left(\frac{z-r_t}{h_t}\right)$
with $r_t$ and $h_t$ satisfying that for some $\epsilon>0$ and large
 $t$, $ |r_t | \leq \bar{r}<\infty$ and $|h_{t} | \leq \bar{h}<\infty$. Let $\lambda_{t} =\sqrt{2}(1-\frac{1}{\alpha_{t}})$,
  where $\alpha_{t}\geq 1$ and  $\sqrt{t}/\alpha_{t} \to \infty$.
 Define
  \begin{equation*}
   \mathsf{W}^{F_t}_{t}(\lambda_{t})  := \sum_{u \in \mathsf{N}_{t}}  F_{t} \left( \frac{\lambda_{t}t- \mathsf{X}_{u}(t) }{\sqrt{t}} \right)e^{\lambda_{t}\mathsf{X}_{u}(t)- (\frac{\lambda_{t}^2}{2}+1)t }   .
  \end{equation*}
Set $\mu_{\mathrm{Gau}}(\dif z)= \frac{1}{\sqrt{2\pi}}  e^{-\frac{z^2}{2}}  \dif z $. Write $\langle F_t, \mu_{\mathrm{Gau}} \rangle= \frac{1}{\sqrt{2 \pi}} \int_{0}^{\infty} F_t(z) e^{-\frac{z^2}{2}} \dif z $.
 \begin{enumerate}[(i)]
  \item If $\alpha_{t} \to \alpha \geq 1$,  and hence   $ \lambda_{t} \to \sqrt{2}(1-\frac{1}{\alpha})=: \lambda$,
    we have
 \begin{equation*}
    \lim_{t \to \infty} \frac{ \mathsf{W}^{F_t}_{t}(\lambda_{t})  }{ \langle F_{t},\mu_{\mathrm{Gau}} \rangle }
     =   \mathsf{W}_{\infty} (\lambda) \ \text{ in probability,}
    \end{equation*}

  \item  If $\alpha_{t} \to \infty$ and hence $\lambda_{t} \to \sqrt{2}$, we have
   \begin{equation*}
    \lim_{t \to \infty} \frac{ 1}{ \langle F_{t},\mu_{\mathrm{Gau}} \rangle }   \frac{ \mathsf{W}^{F_t}_{t}(\lambda_{t})}{ \sqrt{2}- \lambda_{t} }
   = 2 \mathsf{Z}_{\infty}\ \text{ in probability.}
  \end{equation*}

 \end{enumerate}
   \end{lemma}

The results in  Lemma  \ref{lem-functional-convergence-derivative-martingale-2} do not include
the case that  $\lambda_{t} \equiv \sqrt{2}$, where  the limiting distribution is no longer Gaussian.
 According to  \cite[Theorem 1.2]{Madaule16}, we know that  for every bounded continuous function $f$,
    \begin{equation*}
   \lim_{t \to \infty}\sqrt{t} \sum_{u \in \mathsf{N}_{t}} f\left( \frac{\sqrt{2} t- \mathsf{X}_{u}(t)}{\sqrt{t}} \right) e^{-\sqrt{2}(\sqrt{2} t-\mathsf{X}_{u}(t))}
   =    \sqrt{\frac{2}{\pi}} \, \mathsf{Z}_{\infty}  \int_{0}^{\infty} f(z)  z e^{-\frac{z^2}{2}} \dif z
      \end{equation*}
in probability.
The following lemma is
a generalization of this central limit theorem.

\begin{lemma}[{\cite[Proposition 2.6]{MR23}}]
\label{lem-functional-convergence-derivative-martingale}   Let $G$ be a non-negative bounded measurable function with compact support. Suppose $F_t(z)=$ $G\left(\frac{z-r_t}{h_t}\right)$
with $r_t$ and $h_t$ satisfying that
$ \log^{3}(t)/\sqrt{t} \leq r_{t} \leq \bar{r}<\infty$, $r_{t}+yh_{t}=\Theta(r_{t})$  uniformly for $y \in \mathrm{supp}(G)$ and $h_{t}$ decreases at most polynomially fast. Define
   \begin{equation*}
    \mathsf{W}^{F_t}_{t}(\sqrt{2})  := \sum_{u \in \mathsf{N}_{t}}  \mathsf{F}_{t} \left( \frac{\sqrt{2}t- \mathsf{X}_{u}(t) }{\sqrt{t}} \right)e^{-\sqrt{2}(\sqrt{2} t-\mathsf{X}_{u}(t))}   .
   \end{equation*}
  Put $\mu_{\mathrm{Mea}} (\dif z)= z e^{-\frac{z^2}{2}} 1_{\{z>0\}}\dif z $. Write $\langle F_t, \mu_{\mathrm{Mea}} \rangle= \int_{0}^{\infty} F_t(z)z e^{-\frac{z^2}{2}} \dif z $. Then we  have
   \begin{equation*}
   \lim_{t \to \infty}\frac{ \sqrt{ t } }{ \langle \mathsf{F}_{t},\mu_{\mathrm{Mea}} \rangle } \, \mathsf{W}^{F_t}_{t}(\sqrt{2}) = \sqrt{\frac{2}{\pi}} \mathsf{Z}_{\infty} \ \text{ in probability.}
   \end{equation*}
    \end{lemma}

\subsection{Many-to-one lemmas}
Recall that  the
type transformation time $T_u$
of some particle $u\in N^{2}_{t}$, is  the time at which the oldest ancestor of type $2$ of $u$ was born.
We write
\begin{equation}\label{eq-def-Born}
\mathcal{B}=\left\{u \in \bigcup_{t \geq 0} N_t^2, T_{u}=b_u \right\}
\end{equation}
for the set of particles of type $2$ that are born from a particle of type $1$.  We write $u^{\prime} \succcurlyeq u$ if $u^{\prime}$ is a descendant of $u$.

\begin{lemma}[Many-to-one lemmas {\cite[Section 4]{BM21}}]\label{many-to-one-1}
 Let $f$ be a   non-negative  measurable function.
 \begin{enumerate}[(i)]
\item  $
  \mathbb{E}^{\beta,\sigma^{2}} \left(\sum_{u \in \mathcal{B}} f\left(X_u(s), s \leq T_{u}\right)\right) = \int_0^{\infty} e^{\beta t}
  \mathbf{E}
  \left(f\left(  \sigma B_s, s \leq t\right)\right) \dif t$.
\item  $  \mathbb{E} ^{\beta,\sigma^{2}}  \left( e^{-\sum_{u \in \mathcal{B}} f\left(X_u(s), s \leq T_{u} \right)} \right) =\mathbb{E} ^{\beta,\sigma^{2}}  \left(\exp \{ -  \int_0^{\infty} \sum_{u \in N_t^1} ( 1-e^{-f\left(X_u(s), s \leq t\right)} ) \dif t \}\right) $.
 \end{enumerate}
\end{lemma}

The many-to-one lemma is a fundamental tool to compute the first moment or the Laplacian transform of the functional of our two type BBM.
In the rest of this paper,
to simplify notation, when there is no ambiguity, we use $\mathbb{P}$ and $\mathbb{E}$ to denote  $\mathbb{P}^{\beta_{t},\sigma^{2}_{t}}$  and $\mathbb{E}^{\beta_{t},\sigma^{2}_{t}}$, respectively.

\begin{corollary}\label{cor-reduce-to-E-R}
 Let $m(t)$ be a function on $\mathbb{R}_+$.
  For each $R>0$, $t > 0$,  take $\Omega_{t}^{R} \subset [0,t] \times \R$.
\begin{enumerate}[(i)]
  \item   For $A>0$, $0\leq r \leq  t$ and $x \in \mathbb{R}$ define
  \begin{equation*}
  	\mathsf{F}_{t}(r, x)=\mathsf{F}_{t}(r, x ;m(\cdot)):=\mathsf{P}\left(x+ \max\limits_{u \in \mathsf{N}_{r}} \mathsf{X}_{u}(r) \geq m(t)-A\right) ,
  \end{equation*}
  and for $K>0$, define
  \begin{equation}\label{eq-assumption-EYtARK}
 I(t,R)=I(t,R; A,K) := \int_{0}^{t} e^{\beta_{t} s}
  \mathbf{E}
  \left[\mathsf{F}_{t}\left(t-s, \sigma_{t} B_{s}\right) 1_{\{\sigma_{t}B_{r} \leq v_{t}t+ \sigma_{t} K, \forall r \leq s \}}  1_{\left\{(s,\sigma_{t} B_{s} ) \notin  \Omega^{R}_{t}  \right\}}    \right]   \dif s .
  \end{equation}
If for each fixed  $A, K$, we have $\lim\limits_{R \to \infty} \limsup\limits_{t \to \infty} I(t,R) = 0$, then for each $A>0$,
  \begin{equation}\label{eq-control-of-tu-Xtu}
  \lim_{R \to \infty} \limsup_{t \to \infty} \mathbb{P}   \left( \exists u \in N^{2}_{t}: X_{u}(t)>m(t)-A,(T_{u},X_{u}(T_{u})) \notin \Omega^{R}_{t} \right) = 0.
\end{equation}
 \item Let $\widehat{\mathcal{E}}_{t}:=\sum_{u \in N_t^2}   \delta_{  X_u(t)-m(t) }  $  and  $\widehat{\mathcal{E}}_t^R
:=\sum_{u \in N_t^2} 1_{ \{  (T_{u},X_{u}(T_{u})) \in \Omega_{t}^{R} \}}  \delta_{  X_u(t)-m(t) } $.
For any  $\rho  \geq \sqrt{2}$ and $\varphi \in \mathcal{T}$,
using   $\Phi_{\rho}$
defined in \eqref{def-Phi-rho}, we can rewrite $ \mathbb{E}\left(e^{-\langle \varphi, \widehat{\mathcal{E}}^{R}_{t} \rangle} \right) $ as
\begin{multline*}
 \mathbb{E}\left[ \exp \left( -\int_{0}^{\infty} \sum_{u \in N^{1}_{s}}   \Phi_{\rho}\big(t-s,  X_{u}(s)+\rho(t-s)-m(t) \big) 1_{\{ (s,X_{u}(s) )\in \Omega^{R}_{t}  \} }  \dif s \right) \right].
\end{multline*}
 Moreover if  \eqref{eq-control-of-tu-Xtu} holds, then
  \begin{equation*}
 \lim_{R \to \infty} \limsup_{t \to \infty}   \left| \mathbb{E} \left(e^{-\left\langle\widehat{\mathcal{E}}_t^R, \varphi\right\rangle}\right)-\mathbb{E}\left(e^{-\left\langle \widehat{\mathcal{E}_{t}}, \varphi\right\rangle}\right)\right|    =0 \,, \text{ for all }\varphi \in \mathcal{T}.
\end{equation*}
\end{enumerate}
\end{corollary}

\begin{proof}
(i).
Fix  $ A, K> 0$.
For $R, t > 0$, define
  \begin{equation*}
  Y_{t}(R)=Y_{t}(R;A,K):=\sum_{u \in \mathcal{B}}  1_{\left\{ X_{u}(r) \leq v_{t}r+
 \sigma_{t}K,
  \forall r \leq T_{u} \right\}} 1_{\left\{(T_{u},X_{u}(T_{u})) \notin  \Omega^{R}_{t}  \right\}}   1_{\left\{M_{t}^{u} \geq m(t)-A\right\}},
  \end{equation*}
  where $M_{t}^{u}$ is the position of the rightmost descendant of the individual $u$ at time $t$.    Then the probability in \eqref{eq-control-of-tu-Xtu} is less than
  \begin{equation*}
    \P(  \exists s \leq t, u \in \mathsf{N}_s: \mathsf{X}_u(s)
  \geq v_t s +K  )
    +\P( Y_{t}(R) \geq 1).
  \end{equation*}
  Applying the Markov inequality, $\P( Y_{t}(R) \geq 1)$ is bounded above by   $\mathbb{E}\left[Y_{t}(R)\right]$. The branching property implies that
  \begin{equation*}
    \mathbb{E}\left[Y_{t}(R)\right]=\mathbb{E} \left[\sum_{u \in \mathcal{B}}   \mathsf{F}_{t}(t-T_{u}, X_{u}(T_{u}) )1_{\left\{ X_{u}(r) \leq v_{t}r+
    \sigma_{t}K,
      \,
    \forall r \leq T_{u} \right\}} \right].
  \end{equation*}
   Applying  Lemma \ref{many-to-one-1} (i),
   we get
  \begin{equation*}
    \mathbb{E}\left[Y_{t}(R)\right]=  \int_{0}^{t} e^{\beta_{t} s}
    \mathbf{E}
    \left[\mathsf{F}_{t}\left(t-s, \sigma_{t} B_{s}\right) 1_{\{\sigma_{t}B_{r} \leq v_{t}t+ \sigma_{t} K,\, \forall r \leq s \}}  1_{\left\{(s,\sigma_{t} B_{s} ) \notin  \Omega^{R}_{t}  \right\}}    \right] \dif s.
    \end{equation*}
Then by the assumption, $\lim_{R \to \infty} \limsup_{t \to \infty}\P( Y_{t}(R) \geq 1)=0$.
Now the desired result \eqref{eq-control-of-tu-Xtu} follows from \eqref{eq-upper-envelope-0}.

(ii).
 Notice that  $ \langle\widehat{\mathcal{E}}_t^R, \varphi \rangle $ can be rewritten as
  \begin{equation*}
    \sum _{\substack{u \in \mathcal{B} \\ (T_{u},X_{u}(T_{u})) \in \Omega^{R}_{t}}}   \sum _{\substack{u^{\prime} \in N_t^2 \\ u^{\prime} \succcurlyeq u}}
    \ \varphi\left( X_{u'}(t)-X_{u}(T_{u})- \rho(t-T_{u})+X_{u}(T_u)+\rho(t-T_{u})-m(t)   \right)
  \end{equation*}
First using the branching property, and then applying Lemma \ref{many-to-one-1} (ii),  we get
\begin{equation*}
\begin{aligned}
&\mathbb{E}\left(e^{-\left\langle\widehat{\mathcal{E}}_t^R, \varphi\right\rangle}\right)
 =  \mathbb{E}\bigg(    \prod_{\substack{u \in \mathcal{B} \\   (T_{u} , X_{u}(T_{u})) \in
 \Omega^{R}_{t}
   }} \left[ 1-\Phi_{\rho}\big(t-T_{u} ,    X_{u}(T_u)+\rho(t-T_{u})-m(t)  \big)  \right]\bigg)
\\
&=\mathbb{E}\left(\exp \left(-\int_{0}^{\infty} \sum_{u \in N^{1}_{s}}   \Phi_{\rho}\big(t-s, X_{u}(s)+\rho(t-s)-m(t) \big) 1_{\{ (s,X_{u}(s) )\in \Omega^{R}_{t}
   \} }  \dif s \right)\right)
\end{aligned}
\end{equation*}
as desired.
Taking $A>0$ such that
$\mathrm{supp}(\varphi)\subset [-A,\infty)$, we have
  \begin{equation*}
   \left| \mathbb{E} \left(e^{-\left\langle\widehat{\mathcal{E}}_t^R, \varphi\right\rangle}\right)-\mathbb{E}\left(e^{-\left\langle \widehat{\mathcal{E}_{t}}, \varphi\right\rangle}\right)\right| \leq \mathbb{P}\left( \exists u \in N^{2}_{t}: X_{u}(t)>m(t)-A, (T_{u},X_{u}(T_{u})) \notin \Omega^{R}_{t}  \right) .
  \end{equation*}
 Then the result  assertion follows.
\end{proof}

\section{Approaching  $\mathscr{B}_{I,III}$ and $(1,1)$ from $\mathscr{C}_{I}$}
\subsection{Approaching $\mathscr{B}_{I,III}$ from $\mathscr{C}_{I}$}

In this subsection, we are going to prove Theorem \ref{thm-13-approximate} for the case $(\beta,\sigma^2) \in \mathscr{B}_{I,III}$ and  $(\beta_{t},\sigma^{2}_{t})_{t>0} \in \mathscr{A}^{h,-}_{(\beta,\sigma^2)}$.
For simplicity,   in this subsection, we set
\begin{equation*}
  h' := \min\{ h,1/2 \}.
\end{equation*}
Then $
m^{1,3}_{h,-}(t)= v_{t} t - \frac{3-4 h' }{2 \theta_{t}} \log t $. Define
\begin{equation}\label{def-delta}
\delta(x;s,t):= x-v_{t}s + (\theta_{t}-v_{t})(t-s)
\end{equation}
 and
\begin{equation}\label{eq-def-Omega-t-R-8}
  \Omega_{t,h}^{R}:= \left\{ (s,x):  t-s \in [ \frac{1}{R} t^{h'}  ,R t^{h'}] , \   | \delta(x;s,t) | \leq R \sqrt{t-s}     \right\}.
\end{equation}

\begin{lemma}\label{lem-contribution-area-6}
For any $A>0$ and  $h \in (0, \infty]$, we have
  \begin{equation*}
  \lim _{R \rightarrow \infty} \limsup _{t \rightarrow \infty} \mathbb{P}\left(\exists u \in N^{2}_{t}: X_{u}(t) \geq m^{1,3}_{h,-}(t)  -A, (T_{u} , X_{u}(T_{u}))  \notin \Omega_{t,h}^{R}  \right)=0 .
  \end{equation*}
  \end{lemma}

\begin{proof}
  Applying  Corollary \ref{cor-reduce-to-E-R} with $m(t)=m^{1,3}_{h,-}(t)$
  and $\Omega^{R}_{t}=\Omega^{R}_{t,h}$ defined in \eqref{eq-def-Omega-t-R-8},
  it suffices to show that  for each $A, K>0$, $I(t,R) = I(t,R;A,K)$  defined in \eqref{eq-assumption-EYtARK}  vanishes
  first as
  $t \to \infty$ and then $R \to \infty$.
  By conditioning on ``$B_{s}=\sqrt{2\beta_{t}}  s - x$"  in \eqref{eq-assumption-EYtARK}, we have
  \begin{equation}\label{eq-def-I-t-R-6}
  \begin{aligned}
   I(t,R)
  &= \int_{0}^{t} \dif s \int_{-K}^{\infty} \mathbf{P} \left(  B_{r} \leq \sqrt{2\beta_{t}} r+  K, \forall r \leq s | B_{s} =  \sqrt{2\beta_{t}}  s - x   \right)     \\
   & \qquad \qquad\qquad   e^{\beta_{t}s}  \mathsf{F}_{t}\left(t-s, v_{t} s -\sigma_{t} x\right) 1_{\left\{(s ,  v_{t}s - \sigma_{t}x ) \notin  \Omega^{R}_{t,h} \right\}} e^{-\frac{ (\sqrt{2\beta_{t}}s-x)^2 }{2s}} \frac{\dif x}{\sqrt{2\pi s}}   \\
   & \lesssim_{K} \int_{0}^{t} \dif s \int_{-K}^{\infty}\frac{K+x}{s^{3/2}}  e^{ \sqrt{2 \beta_{t}} x - \frac{ x^2 }{2s}}    \mathsf{F}_{t}\left(t-s, v_{t} s -\sigma_{t} x\right) 1_{\left\{(s ,  v_{t}s - \sigma_{t}x ) \notin  \Omega^{R}_{t,h} \right\}}   \dif x  .
  \end{aligned}
  \end{equation}
 Above, we used  $ \mathbf{P} \left(  B_{r} \leq \sqrt{2\beta_{t}} t+  K, \forall r \leq s | B_{s} =  \sqrt{2\beta_{t}}  s - x   \right)   \lesssim_{K} \frac{  K+ x}{ s} $,  which holds by Lemma  \ref{lem-bridge-estimate-0} since $(B_{r}-\frac{r}{s}B_{s})_{r \leq s}$ is a Brownian bridge independent of $B_{s}$.
 Put $\mathrm{w}:= \frac{3-4h'}{2\theta_{t}} \log t + A$.
  \begin{itemize}
  \item
 If $v_{t}   (t-s) +\sigma_{t} x   -\mathrm{w}   \geq 1 $, by \eqref{eq-BBM-tail-2}, we have
    \begin{align}
  & \mathsf{F}_{t}\left(t-s,  v_{t} s - \sigma_{t} x \right)  =\mathsf{P}\left( \max\limits_{u \in \mathsf{N}_{t-s}} \mathsf{X}_{u}(t-s) >  v_{t}   (t-s) +\sigma_{t} x   -\mathrm{w}  \right) \notag \\
  &    \lesssim     \frac{\sqrt{t-s}}{ v_{t}  (t-s) + \sigma_{t} x   - \mathrm{w} } \exp \left\{- (\frac{v_{t}^2}{2}-1)(t-s)- v_{t}\sigma_{t}x +  v_{t} \mathrm{w} - \frac{(\sigma_{t}x- \mathrm{w})^2}{2(t-s)} \right\}.  \label{eq-bound-F-6}
  \end{align}
  \item If $v_{t}   (t-s) +\sigma_{t} x   -\mathrm{w}   \leq 1$, we simply upper bound  $\mathsf{F}_{t}\left(t-s,  v_{t} s - \sigma_{t} x \right)$ by $1$.  Note that as long as $t$ is large, we can deduce from
    $v_{t}   (t-s) +\sigma_{t} x   -\mathrm{w}   \leq 1$  that  $s \geq t- (\log t)^2$ and $\sigma_{t} x \leq  \mathrm{w}+1$, and hence
    $ \sqrt{2 \beta_{t} } x \leq \frac{3-4h'}{2} \log t + (A+1) \theta_{t} $. Therefore,
\begin{align}
 & \int_{t-(\log t)^{2}}^{t}  \dif s     \int_{-K}^{ O(\log t)}  \frac{K+x}{s^{3/2}}e^{ \sqrt{2 \beta_{t} } x - \frac{x^2}{2s}} 1_{ \{v_{t}   (t-s) +\sigma_{t} x   -\mathrm{w}   \leq 1 \}  } \dif x   \notag \\
 &\lesssim  \int_{t-(\log t)^{2}}^{t}  \dif s     \int_{-K}^{ O(\log t)}  \frac{K+O(\log t) }{ t^{3/2}}e^{ \sqrt{2 \beta_{t} } x  }   \dif x \lesssim
  (\log t)^{4}   t^{-3/2}  t^{ \frac{3-4h'}{2}} =  o(1).  \label{eq-bound-F-62}
\end{align}
  \end{itemize}
  Combining \eqref{eq-bound-F-6} and \eqref{eq-bound-F-62} and
letting  $J= \frac{K + x}{s^{3/2}} \frac{\sqrt{t-s}}{ v_{t}  (t-s) + \sigma_{t} x   - \mathrm{w} }  e^{v_{t} \mathrm{w}}$, we get
  \begin{equation*}
 I(t,R) \lesssim  \int_{0}^{t} \dif s\int_{-K}^{\infty}   J  \, e^{\sqrt{2\beta_{t}}(1-  \sigma_{t}^2)x - (\frac{v_{t}^2}{2}-1)(t-s) - \frac{(\sigma_{t}x- \mathrm{w})^2 }{2(t-s)} }
   e^{  -\frac{x^2}{2s}}    1_{\left\{(s ,  v_{t}s+\sigma_{t}x ) \notin  \Omega^{R}_{t,h} \right\}}     \dif x  + o(1) .
  \end{equation*}
  Make a change of variable $x = \frac{\theta_{t}-v_{t}}{\sigma_{t}}(t-s)+ y  = \sqrt{2 \beta_{t}}(\sigma_{t}^{-2}-1)(t-s)+y$.  Note that  $ (s ,  v_{t}s+\sigma_{t}x ) \in \Omega_{t,h}^{R}$ if and only if $(t-s)/ t^{h'} \in  \Gamma_{R} :=[R^{-1},R]$   and $ |\sigma_{t} y | \leq R\sqrt{s} $. Moreover,   computing that
 \begin{align*}
  & \sqrt{2\beta_{t}}(1-  \sigma_{t}^2)x - (\frac{v_{t}^2}{2}-1)(t-s) - \frac{(\sigma_{t}x- \mathrm{w})^2 }{2(t-s)} \\
   &= [2 \beta_{t}(1-\sigma^2_{t}) ( \sigma^{-2}_{t}-1)-(v_{t}^2/2-1)- \sigma_{t}^2  \beta_{t}(\sigma_{t}^{-2}-1)^2 ] (t-s) \\
    & \qquad \qquad \qquad  + [\sqrt{2\beta_{t}}(1- \sigma_{t}^2) - (\theta_{t}-v_{t}) \sigma_{t}] y  + (\theta_{t}-v_{t}) \mathrm{w} - \frac{(\sigma_{t}y-\mathrm{w})^2}{2(t-s)} \\
    &=  -  \left[\frac{ 2\beta_{t}-1}{\sigma_{t}^2 t^{h}}  - \frac{\beta_{t}}{\sigma_{t}^2 t^{2h}}  \right](t-s )    + (\theta_{t}-v_{t}) \mathrm{w} - \frac{(\sigma_{t}y-\mathrm{w})^2}{2(t-s)},
  \end{align*}
 we get
  \begin{equation*}
    I(t,R) \lesssim  \int_{0}^{t} \dif s \int_{-K}^{\infty}   e^{(\theta_{t}-v_{t})\mathrm{w}} J
   e^{
   -  [\frac{ 2\beta_{t}-1}{\sigma_{t}^2}+o(1)] \frac{t-s}{t^{h}}  } e^{    - \frac{(\sigma_{t}y-\mathrm{w})^2}{2(t-s)} }
   e^{  -\frac{x^2}{2s}}    1_{\left\{ \substack{ (t-s)/ t^{h'} \notin  \Gamma^{R} \\   \text{ or }
   |\sigma_{t} y |>R \sqrt{t-s} } \right\} }      \dif y + o(1).
  \end{equation*}
Now make change of variables again
$t-s = \xi t^{h'}$, $y= \eta \sqrt{t-s}$.
 Denote  $\tilde{J}:= t^{h'} \sqrt{t-s}\,  e^{(\theta_{t}-v_{t})\mathrm{w}} J  $. We have
  \begin{align*}
    I(t,R) \lesssim \int_{0}^{t^{1-h'}} \dif \xi\int_{\mathbb{R}}  \tilde{J}
      & \ e^{ -  \left[\frac{ 2\beta_{t}-1}{\sigma_{t}^2}+o(1) \right] \xi t^{h'-h}  }e^{   - \frac{1}{2} \left(\sigma_{t}\eta- \frac{o(1)}{ \sqrt{\xi  }}  \right)^2  }
    \\
    &\times  e^{  -   \frac{[(\theta_{t}-v_{t})\xi t^{h'}+  \sigma_{t}\eta \sqrt{\xi t^{h'}} ]^2}{2\sigma_{t}^2(t-\xi t^{h'})}}  1_{\left\{ \substack{ \xi \notin  \Gamma^{R}  , \text{or}\\
    |\sigma_{t} \eta |>R}\right\}}  \dif \eta + o(1) .
   \end{align*}
Notice that for fixed $\xi > 0$ and $\eta >0$,
\begin{align*}
 \lim_{t \to \infty} \tilde{J}  =  \lim_{t \to \infty}  t^{h'}      \frac{  (\theta_{t}-v_{t}) \xi t^{h'}+ o(t^{h'})  }{ (t- \xi t^{h'}) ^{3/2}} \frac{ \xi t^{h'} }{ \theta_{t}  \xi t^{h'} + o(t^{h'}) }  t^{\frac{{3-4h'}}{2}} e^{A \theta_{t}} = \frac{\theta- v}{\theta} \xi e^{A \theta}.
\end{align*}
  Letting $t \to \infty$, applying the dominated convergence theorem twice, we finally get
   \begin{equation*}
    \varlimsup_{t \to \infty}I(t,R) \lesssim  \int_{0}^{\infty} \dif \xi \int_{-K}^{\infty}   \xi
    e^{ - \frac{ 2\beta-1}{\sigma^2} \xi 1_{\{h \leq \frac{1}{2}\}}  }e^{ - \frac{\sigma^2\eta^2}{2}}
    e^{  -\frac{(\theta-v)^2}{2 \sigma^2} \xi^2 1_{\{h \geq \frac{1}{2}\}} }    1_{\left\{ \substack{ \xi \notin  \Gamma^{R}  , \text{or}\\
    |\sigma \eta |>R}\right\}}  \dif \eta \overset{R \to \infty }{\longrightarrow } 0.
   \end{equation*}
  This completes the proof.
  \end{proof}

Now we are ready to show Theorem \ref{thm-13-approximate} for the case that   $(\beta_{t},\sigma^{2}_{t})_{t>0} \in \mathscr{A}^{h,-}_{(\beta,\sigma^2)}$.

\begin{proof}[Proof of Theorem \ref{thm-13-approximate}
for   $(\beta_{t},\sigma^{2}_{t})_{t>0} \in \mathscr{A}^{h,-}_{(\beta,\sigma^2)}$
]
  Take $\varphi \in \mathcal{T}$. Applying
 Corollary \ref{cor-reduce-to-E-R} (ii)
  with $m(t)=m^{1,3}_{h,-}(t)$, $\rho=\theta_{t}$, and $\Omega^{R}_{t,h}$ defined in \eqref{eq-def-Omega-t-R-8}, it suffices to study the asymptotic behavior of $\mathbb{E}\left(e^{-\left\langle\widehat{\mathcal{E}}_t^R, \varphi\right\rangle}\right)  $, which equals
  \begin{equation*}
 \mathbb{E}\bigg(\exp \bigg\{ -  \int_{0}^{t} \sum_{u \in N^{1}_{s}}  \Phi_{\theta_{t}}\bigg(t-s, \delta(X_{u}(s);s,t) + \frac{3-4h'}{2\theta_{t}}\log t \bigg) 1_{\{ (s,X_{u}(s) )\in \Omega^{R}_{t,h}  \} }  \dif s \bigg\} \bigg),
  \end{equation*}
 where we used the fact that
  $X_{u}(s)+\theta_{t}(t-s)-m^{1,3}_{h,-}(t) =\delta(X_{u}(s);s,t)+\frac{3-4h'}{2\theta_{t}}\log t$
  with $\delta(\cdot, \cdot,\cdot)$  defined in \eqref{def-delta}.
  Since $\theta_{t} \to \theta > \sqrt{2}$,  applying part (Ii) of Lemma \ref{thm-Laplace-BBM-order}, we have,
  uniformly for $(s,X_{u}(s))\in \Omega^{R}_{t,h}$,
  \begin{align*}
   & \Phi_{\theta_{t}}\big(t-s, \delta(X_{u}(s);s,t)+\frac{3-4h'}{2\theta_{t}}\log t \big) \\
  &\sim  \frac{\gamma_{\theta}(\varphi)}{\sqrt{t-s}}  e^{-(\frac{\theta_{t}^2}{2}-1)(t-s)} e^{\theta_{t} [X_{u}(s)+\theta_{t}(t-s)-v_{t}t]+  \frac{3-4h'}{2}\log t  }e^{-\frac{1}{2(t-s)} \delta(X_{u}(s);s,t)^2} \\
  & \sim \gamma_{\theta}(\varphi) \frac{t^{3/2}}{t^{2h'}(t-s)^{1/2}} e^{-\beta_{t}\frac{t-s}{t^{h}}} e^{\theta_{t}X_{u}(s)- 2\beta_{t}s} e^{-\frac{1}{2(t-s)} \delta(X_{u}(s);s,t)^2}.
  \end{align*}
 Above, we used   that  $ -(\frac{\theta_{t}^2}{2}-1)(t-s)+\theta_{t}^2(t-s)-\theta_{t} v_{t}t= - \beta_{t}\frac{t-s}{t^{h}}-2\beta_{t}s$. Thus
    substituting this asymptotic equality into the integral, we get
    \begin{align}
   &  \int_{0}^{t} \sum_{u \in N^{1}_{s}}  \Phi_{\theta_{t}}\bigg(t-s, \delta(X_{u}(s);s,t) + \frac{3-4h'}{2\theta_{t}}\log t \bigg) 1_{\{ (s,X_{u}(s) )\in \Omega^{R}_{t,h}  \} }  \dif s  \notag \\
   & \sim  \int_{t-Rt^{h'}}^{t-\frac{1}{R}t^{h'}} \sum_{u \in N^{1}_{s}}  \frac{\gamma_{\theta}(\varphi)t^{3/2}}{t^{2h'}(t-s)^{1/2}} e^{-\beta_{t}\frac{t-s}{t^{h}}} e^{\theta_{t}X_{u}(s)- 2\beta_{t}s} e^{-\frac{1}{2(t-s)} \delta(X_{u}(s);s,t)^2}  1_{\{ (s,X_{u}(s) )\in \Omega^{R}_{t,h}  \} }   \dif s \notag \\
   &  \sim \gamma_{\theta}(\varphi)\int_{\frac{1}{R}}^{R}   \lambda e^{-\beta_{t}\lambda t^{h'-h}} \frac{t^{3/2}}{r^{3/2}}\sum_{u \in N^{1}_{s}}e^{\theta_{t}X_{u}(s)- 2\beta_{t}s} e^{-\frac{\delta(X_{u}(s);s,t)^2}{2r} }  1_{\{ (s,X_{u}(s) )\in \Omega^{R}_{t,h}  \} }   \dif \lambda, \label{eq-integral-Phi-1}
  \end{align}
where in the last equality we made  change of variables $t-s = r$ and   $ r=   \lambda t^{h'} $.
Let  $G(x)=G_{R}(x) =e^{-\frac{x^2}{2}}1_{\{|x|  \leq R \}}$, and define
    \begin{equation*}
W^G(s, r ; t):=\sum_{u \in N_s^1} e^{-\theta_{t}\left[v_{t} s-X_u(s)\right]} G\left(\frac{v_{t} s-X_u(s)-(\theta_{t}-v_{t}) r}{\sqrt{r}}\right)  \text{ for } s>0, r > 0.
    \end{equation*}
   Then  the integral in \eqref{eq-integral-Phi-1} can be  rewritten as $ \int_{\frac{1}{R}}^{R}  \lambda e^{-\beta_{t}\lambda t^{h'-h}} \left(  \frac{t}{r}  \right)^{3/2}W^G(t-r, r; t)   \dif \lambda $. Hence
  \begin{equation}\label{eq-integral-10}
   \mathbb{E}\left(e^{-\left\langle\widehat{\mathcal{E}}_t^R, \varphi\right\rangle}\right)  = \mathbb{E} \left( \exp \left\{ - [1+o(1)] \gamma_{\theta}(\varphi) \int_{\frac{1}{R}}^{R}  \lambda e^{-\beta_{t}\lambda t^{h'-h}} \left(  \frac{t}{r}  \right)^{3/2}W^G(t-r, r; t)   \dif \lambda  \right\} \right).
  \end{equation}

By the scaling property of Brownian motion,
  $\left\{X_u(s): u \in N_s^1\right\} \overset{law}{=} \left\{\frac{\sigma_{t}}{\sqrt{\beta_{t}}} \mathsf{X}_u(s') : u \in \mathsf{N}_{s'}\right\} $, where $s'=\beta_{t}s$. So $W^{G}(s,r;t)$ has the same distribution as
 \begin{equation*}
  \mathsf{W}^{G}(s',r ; t):= \sum_{u \in \mathsf{N}_{s'}} e^{ \sqrt{2} \mathsf{X}_{u}(s')-2 s'}   G\left( \frac{ \frac{\sqrt{2}s'-\mathsf{X}_{u}(s')}{\sqrt{s'}}  -\frac{\sqrt{\beta_{t}}}{\sigma_{t}} (\theta_{t}-v_{t}) \frac{r}{\sqrt{s'}} }{\frac{\sqrt{\beta_{t}}}{\sigma_{t}} \frac{\sqrt{r}}{\sqrt{s'}} }\right).
 \end{equation*}
 Here we remind that $\mathsf{W}^G$ is for single type BBM and $W^G$ is for two-type BBM.
For each fixed $\lambda>0$ and $r= \lambda t^{h'}$,
applying Lemma \ref{lem-functional-convergence-derivative-martingale} with $r_t=\frac{\sqrt{\beta_{t}}}{\sigma_{t}} (\theta_{t}-v_{t}) \frac{r}{\sqrt{\beta_t(t-r)}}$ and $h_t=\frac{\sqrt{\beta_{t}}}{\sigma_{t}} \frac{\sqrt{r}}{\sqrt{\beta_t(t-r)}}$
(noticing that $r_t=\Theta(t^{h'-\frac{1}{2}})$ and $h_t\ll r_t$, the conditions in  Lemma \ref{lem-functional-convergence-derivative-martingale}  are satisfied),
we have
    \begin{align*}
   & \lim _{t \to \infty} \left(\frac{t }{ r}\right)^{3/2}
   {W}^G
   \left(t-r, r ; t\right) = \lim _{t \to \infty} \left(\frac{t }{ r}\right)^{3/2} \mathsf{W}^{G}( \beta_{t}(t-r),r ; t)  \\
    &=  \mathsf{Z}_{\infty} \sqrt{\frac{2}{\pi}}  \lim_{t \to \infty} \left(\frac{t }{ r}\right)^{3/2}  \frac{1}{\sqrt{\beta_{t} t}}   \int_{0}^{\infty} z e^{-\frac{z^2}{2}}  G_{R}\left(  \frac{ z   -\frac{\sqrt{\beta_{t}}}{\sigma_{t}} (\theta_{t}-v_{t})
  \frac{r}{\sqrt{\beta_{t}(t-r)}} }{\frac{\sqrt{\beta_{t}}}{\sigma_{t}}  \frac{\sqrt{r}}{\sqrt{\beta_{t}(t-r)}} }
    \right) \dif z   \\
    &= \mathsf{Z}^{\beta,\sigma^2}_{\infty} \frac{(\theta-v)}{\sigma^3}   e^{-\frac{(\theta-v)^2}{2 \sigma^2} \lambda^2 1_{\{h \geq 1/2\}}} \sqrt{\frac{2}{\pi}} \int G_{R}(y) \dif  y    \quad \text{ in law,}
    \end{align*}
    where in the last equality we used the fact that $   \mathsf{Z}_{\infty}   \overset{law}{=}  \frac{\sqrt{\beta}}{\sigma} \mathsf{Z}_{\infty}^{\beta, \sigma^2}$.
  Letting $t \to \infty$ in \eqref{eq-integral-10} and applying  the dominated convergence theorem,   we finally get
  \begin{equation*}
 \lim_{t \to \infty} \mathbb{E}\left(e^{-\left\langle\widehat{\mathcal{E}}_t^R, \varphi\right\rangle}\right) =
\mathsf{E}
 \left(\exp \left\{-  C_{h,R} \gamma_{\theta}(\varphi)  \mathsf{Z}^{\beta,\sigma^2}_{\infty}   \right\}\right) ,
    \end{equation*}
 where
  \begin{align*}
   & C_{h,R}  = \int_{\frac{1}{R}}^{R}  e^{-\beta \lambda 1_{\{h \leq 1/2\}} }  \frac{(\theta-v) \lambda}{\sigma^3} e^{-\frac{(\theta-v)^2}{2 \sigma^2} \lambda^2 1_{\{h \geq 1/2\}}} \dif  \lambda  \sqrt{\frac{2}{\pi}}  \int G_{R}(y) \dif  y  \\
    & \overset{R \to \infty}{\longrightarrow}
    C_{h}:=
   \frac{{2}(\theta- v)}{\beta^{2} \sigma^{3}} 1_{\{  h < 1/2\}}
     +   \frac{2 1_{\{  h > 1/2\}}}{\sigma(\theta - v)}   + {2} \int_{0}^{\infty}  \frac{(\theta-v) \lambda}{\sigma^3} e^{-\beta \lambda -\frac{(\theta-v)^2}{2 \sigma^2} \lambda^2 } \dif  \lambda1_{\{h = 1/2\}} .
 \end{align*}
    Then by part (ii) of Corollary \ref{cor-reduce-to-E-R},
   letting $R \to \infty$   we  get
    \begin{equation*}
      \lim\limits_{t \to \infty}  \mathbb{E}\left(e^{-\left\langle\widehat{\mathcal{E}}_t, \varphi\right\rangle}\right)=
     \mathsf{E}
       \left(\exp \left\{-
     C_{h}
      \gamma_{\theta}(\varphi)  \mathsf{Z}^{\beta,\sigma^2}_{\infty}   \right\}\right) .
      \end{equation*}
  Recalling the definition of $\gamma_{\theta}(\varphi)$ in Lemma \ref{thm-Laplace-BBM-order}, the right hand side
    is the Laplace functional of
   $\operatorname{DPPP}\left( C_{h, -}   \mathsf{Z}_{\infty}^{\beta,\sigma^2}  \theta e^{-\theta x} \dif x, \mathfrak{D}^{\theta}\right)$ with $C_{h,-}:= \frac{C_{h}C(\theta)}{\theta \sqrt{2 \pi}}$.
  By \cite[Lemma 4.4]{BBCM22}, we complete the proof.
\end{proof}

\subsection{Approaching $(1,1)$ from $\mathcal{C}_{I}$}

In this subsection,
we are going to prove Theorem \ref{thm-11-approximate}
for the case $(\beta_{t},\sigma^{2}_{t})_{t>0} \in \mathscr{A}^{h,1}_{(1,1)}$. We set, in this subsection,
\begin{equation*}
  h' := \min\{ h, 1  \}.
\end{equation*}
Then $ m^{(1,1)}_{h,1}(t)= v_{t} t-\frac{3-2h'}{2\sqrt{2}} \log t$.  Define
\begin{equation}\label{eq-def-Omega-t-R-11}
  \Omega_{t,h}^{R}=
 \begin{cases}
 \{ (s,x):  t-s \in  [\frac{1}{R} t^{h}, R t^{h}], \    v_{t} s- x \in   [\frac{1}{R} \sqrt{t-s}, R\sqrt{t-s}]     \} \quad \text{ for } h \in (0,1) ; \\
 \{ (s,x): t-s \in [\frac{1}{R}t, (1-\frac{1}{R})t] , \ v_{t} s- x \in   [\frac{1}{R} \sqrt{t}, R\sqrt{t}]   \} \quad \text{ for } h \in [1,\infty] .
  \end{cases}
\end{equation}

\begin{lemma}\label{lem-contribution-area-11}
  For each $A>0$,
  \begin{equation*}
  \lim _{R \rightarrow \infty} \limsup _{t \rightarrow \infty} \mathbb{P}\left(\exists u \in N^{2}_{t}: X_{u}(t) \geq m^{(1,1)}_{h,1}(t)  -A, (T_{u} , X_{u}(T_{u}))  \notin \Omega_{t,h}^{R}  \right)=0 .
  \end{equation*}
  \end{lemma}

\begin{proof}
  Applying  Corollary \ref{cor-reduce-to-E-R} with $m(t)=m^{(1,1)}_{h,1}(t)$  and $\Omega^{R}_{t,h}$ defined in \eqref{eq-def-Omega-t-R-11}, it suffices to show that  for each $A, K>0$, $I(t,R) = I(t,R;A,K)$  defined in \eqref{eq-assumption-EYtARK}  vanishes as first $t \to \infty$ and then $R \to \infty$. As  in \eqref{eq-def-I-t-R-6}  we have
  \begin{equation}\label{eq-def-I-t-R-61}
   I(t,R) \lesssim_{K} \int_{0}^{t} \dif s \int_{-K}^{\infty}\frac{K+x}{s^{3/2}}  e^{ \sqrt{2 \beta_{t}} x - \frac{ x^2 }{2s}}    \mathsf{F}_{t}\left(t-s, v_{t} s -\sigma_{t} x\right) 1_{\left\{(s ,  v_{t}s - \sigma_{t}x ) \notin  \Omega^{R}_{t,h} \right\}}   \dif x .
  \end{equation}
Now we need  an finer upper bound for  $\mathsf{F}_{t}$, which are given below.
Let $L_{s,t}:= \sqrt{2}(\sigma_{t}^2-1)(t-s) -  \frac{3}{2\sqrt{2}} \log (\frac{t}{t-s+1})  +  \frac{h'}{\sqrt{2}} \log t -A$.
  \begin{itemize}
  \item
   If $  L_{s,t} + \sigma_{t} x  > 1$,
 noticing that
   $(\beta_{t},\sigma^{2}_{t})_{t>0} \in \mathscr{A}^{h,1}_{(1,1)}$ implies that  $\theta_t=\sqrt{2}$ and $v_t=\sqrt{2}\sigma_t^2$,
   we have
  \begin{align*}
    & \mathsf{F}_{t}(t-s, v_{t} s - \sigma_{t} x ) = \mathsf{P} \left( \max\limits_{u \in \mathsf{N}_{t-s}} \mathsf{X}_{u}(t-s) \geq  v_{t}(t-s) + \sigma_{t} x -  \frac{3-2h'}{2\sqrt{2}} \log t-A \right)\\
     & = \mathsf{P} \left(\max\limits_{u \in \mathsf{N}_{t-s}} \mathsf{X}_{u}(t-s ) \geq  \sqrt{2} (t-s)-  \frac{3}{2\sqrt{2}}\log (t-s+1) +   L_{s,t}  +\sigma_{t} x \right) \\
     & \lesssim_{A} (\sigma_{t}x + L_{s,t}) \frac{t^{3/2 }}{(t-s + 1)^{3/2}t^{h'}} e^{ - \sqrt{2} \sigma_{t} x - 2(\sigma_{t}^2-1)(t-s) } e^{- \frac{(\sigma_{t} x+L_{s,t})^2}{3(t-s)}},
  \end{align*}
where in that last inequality we used \eqref{eq-Max-BBM-tail-2}.

 \item  If  $L_{s,t} + \sigma_{t} x  \leq  1$, we simply upper bound
 $ \mathsf{F}_{t}(t-s, v_{t} s - \sigma_{t} x ) $ by $1$.
 Note that $L_{s,t} + \sigma_{t} x  \leq  1$ implies that when $t$ is large,  $s \geq t/2$
  and  $ \sigma_{t} x \leq 1-L_{s,t} \leq   \frac{3}{2\sqrt{2}} \log (\frac{t}{t-s+1})  -  \frac{h'}{\sqrt{2}} \log t +1 $. Moreover as $\beta_{t}= \sigma_{t}^2$,
 \begin{equation*}
  \int_{t/2}^{t} \dif s \int_{-K}^{O(\log t)} \frac{K+x}{s^{3/2}}  e^{ \sqrt{2 \beta_{t}} x}  1_{ \{ L_{s,t} + \sigma_{t} x  \leq  1 \} }   \dif x  \lesssim \frac{O(\log t)^2 }{t^{h'}}    \int_{0}^{t/2}\frac{1}{ \left(u+1\right)^{3/2} } \dif u = o(1).
 \end{equation*}
  \end{itemize}
   Therefore we have
   \begin{equation}\label{eq-I(t,R)-1}
    I(t,R)  \lesssim  \int_{0}^{t} \int_{-K}^{\infty}  \frac{K+x}{s^{3/2}} |\sigma_{t}x + L_{s,t}|\frac{t^{3/2 }}{(t-s+1)^{3/2}t^{h'}} e^{  - 2(\sigma_{t}^2-1)(t-s) }  e^{-\frac{x^2}{2s}- \frac{(\sigma_{t} x+L_{s,t})^2}{3(t-s)}} \dif x \dif s + o(1).
   \end{equation}

   For the case  $h \in (0,1)$,
   make change of variables $t-s= \xi t^{h}$ and $x= \eta \sqrt{t-s}$ . Note that $(s, v_{t} s + \sigma_{t} x) \in  \Omega^{R}_{t,h}$ if and only if $\xi \in \Gamma_{R} :=[R^{-1},R] $ and $ \sigma_{t} \eta \in \Gamma_{R}$.
  Use $\ell_{\xi,t}$ to denote  $L(t-\xi t^{h},t) $.
   Noting that   $|\ell_{\xi,t}| \leq \sqrt{2}(\sigma_{t}^2-1)t^{h} \xi +  \frac{3}{2\sqrt{2}} \log (\xi t^{h}) + \frac{3-4h'}{2\sqrt{2}} \log t +A  = \Theta(\xi)+O(\log t)$, applying the dominated convergence theorem   we have
\begin{equation*}
  \begin{aligned}
I(t,R) &\lesssim \int_{0}^{t^{1-h}} \dif \xi \int_{-\frac{K}{\sqrt{\xi t^{h}}}}^{\infty}  1_{\left\{ \substack{ \xi \notin \Gamma_{R}  \text{ or}  \\  \sigma_{t} \eta \notin \Gamma_{R}
} \right\} }   \frac{K+\eta \sqrt{\xi t^{h}}}{(t-\xi t^{h}+1)^{3/2}} \\
& \qquad \qquad\qquad\qquad \times   [ \sigma_{t} \eta \sqrt{\xi t^{h}}+\ell_{\xi,t}]   \frac{t^{3/2}}{ \xi t^{h} }
e^{-2(\sigma_{t}^2-1)t^{h}\xi  -  \frac{1}{3}  (\sigma_{t} \eta + \frac{\ell_{\xi,t}}{\sqrt{\xi t^{h}}})^2}   \dif \eta
+ o(1)  \\
 & \overset{t \to \infty }{\longrightarrow }  \int_{0}^{\infty}  \dif \xi  \int_{0}^{\infty}
  1_{\left\{ \substack{ \xi \notin \Gamma_{R}   \\   \text{ or }\eta \notin \Gamma_{R}
 } \right\} }
  \eta^2    e^{-\xi  }
 e^{ - \frac{\eta^2}{3}}
  \dif \eta \overset{R \to \infty }{\longrightarrow }  0.
  \end{aligned}
\end{equation*}

For the case  $h \geq 1$,
make  change of variables $s= \xi t$ and $x= \eta \sqrt{t}$ to the integral in \eqref{eq-I(t,R)-1}. Now $(s,v_{t} s + \sigma_{t} x) \in  \Omega^{R}_{t,h}$ if and only if $\xi \in [R^{-1},1-R^{-1}] $ and $ \sigma_{t} \eta \in \Gamma_{R}$. Similarly letting  $\ell_{\xi,t}:=L(\xi t,t) $ and noting that   $|\ell_{\xi,t}| =\Theta(\xi)+O(\log t)$,  applying the dominated convergence theorem
\begin{equation*}
  \begin{aligned}
I(t,R) &\lesssim \int_{0}^{1}   \dif \xi \int_{-\frac{K}{\sqrt{t}}}^{\infty}  1_{\left\{ \substack{ \xi \notin [R{^{-1}}, 1-R^{-1}]    \\ \text{or } \sigma_{t} \eta \in \Gamma_{R}
} \right\} }   \frac{K+\eta \sqrt{t}}{(\xi t)^{3/2}} [\sigma_{t} \eta \sqrt{t} +\ell_{\xi,t} ]  \\
& \qquad  \qquad \qquad \qquad \qquad  \times \frac{t^{3/2} \sqrt{t}}{(t-\xi t+1)^{3/2}}   e^{  - \frac{\eta^2}{2\xi}- \frac{(\sigma_{t} \eta +\ell_{\xi,t}/\sqrt{t})^2}{3(1-\xi)} }  \dif \eta + o(1)  \\
 & \overset{t \to \infty }{\longrightarrow }  \int_{0}^{1}   \dif \xi \int_{0}^{\infty}  1_{\left\{ \substack{ \xi \notin [R{^{-1}}, 1-R^{-1}]    \\ \text{or }  \eta \in [R^{-1} , R  ]
 } \right\} }
  \eta^2  \frac{1}{\xi^{3/2}}  \frac{1}{(1-\xi)^{3/2}}  e^{-\frac{\eta^2}{2\xi } \frac{\eta^2}{3(1-\xi)} }        \dif \eta \overset{R \to \infty} \to 0,
  \end{aligned}
\end{equation*}
where we used the fact that
$ \int_{0}^{1}    \frac{1}{\xi^{3/2}}  \frac{1}{(1-\xi)^{3/2}}      \dif \xi \int_{0}^{\infty} \eta^2 e^{-\frac{\eta^2}{2\xi } -\frac{\eta^2}{3(1-\xi)} }   \dif \eta    < \infty $. We now complete the proof.
  \end{proof}

\begin{proof}[Proof of Theorem \ref{thm-11-approximate}
for $(\beta_{t},\sigma^{2}_{t})_{t>0} \in \mathscr{A}^{h,1}_{(1,1)}$
]
Take $\varphi \in \mathcal{T}$.
Applying Corollary \ref{cor-reduce-to-E-R} with $m(t)=m^{(1,1)}_{h,1}(t)$, $\rho=\sqrt{2}$, and $\Omega^{R}_{t,h}$ defined in \eqref{eq-def-Omega-t-R-11}, it suffices to study the asymptotic behavior of $\mathbb{E}\left(e^{-\left\langle\widehat{\mathcal{E}}_t^R, \varphi\right\rangle}\right)$, which is equal to
  \begin{equation*}
  \mathbb{E}\bigg(\exp \bigg\{ - \int_{0}^{t} \sum_{u \in N^{1}_{s}}  \Phi_{\sqrt{2}}\big(t-s, X_u(s) + \sqrt{2}(t-s)-  m^{(1,1)}_{h,1}(t)   \big) 1_{\{ (s,X_{u}(s) )\in \Omega^{R}_{t,h}  \} }  \dif s \bigg\} \bigg).
  \end{equation*}

Rewrite   $  \sqrt{2}(t-s)-  m^{(1,1)}_{h,1}(t)  $ as $ - v_{t}s-\mathrm{y}$, where $\mathrm{y}:=(v_{t}-\sqrt{2})(t-s)- \frac{3-2h'}{2 \sqrt{2}}\log t $.
For  $(s,X_{u}(s)) \in \Omega^{R}_{t,h} $, we have $t-s = \Theta(t^{h'})$, $ v_{t} s- X_{u}(s) =   \Theta(\sqrt{t-s})$ and  $\mathrm{y}=O(\log t)$.
 Then part (i) of  Lemma \ref{thm-Laplace-BBM-order} yields that as $t \to \infty$ uniformly in $(s,X_{u}(s))\in \Omega^{R}_{t,h}$,
\begin{align*}
& \Phi_{\sqrt{2}}(t-s, X_u(s) - v_{t}s -\mathrm{y} )   \\
& \sim \gamma_{\sqrt{2}}(\varphi)
\frac{ v_{t} s-X_{u}(s) + \mathrm{y}}{(t-s)^{3/2}}
e^{ \sqrt{2}(X_{u}(s)-v_{t}s)- \sqrt{2}(v_t-\sqrt{2})(t-s)} t^{\frac{3-2h'}{2}}
e^{-\frac{(X_{u}(s)-v_{t}s-\mathrm{y})^2}{2(t-s)}} \\
&\sim \gamma_{\sqrt{2}}(\varphi) \frac{t^{3/2} }{(t-s)^{3/2}t^{h'}}  e^{- 2(\beta_{t}-1)(t-s)} [\sqrt{2} \beta_{t} s-X_{u}(s)] e^{ \sqrt{2}X_{u}(s)-2 \beta_{t}s}e^{-\frac{(X_{u}(s)-\sqrt{2}\beta_{t} s)^2}{2(t-s)}},
\end{align*}
 where we used that  $v_{t}=\sqrt{2}\beta_{t}$ as $\beta_{t}=\sigma_{t}^2$. Thus
 \begin{equation}\label{eq-integral-Phi-2}
    \begin{aligned}
  & \int_{0}^{t} \sum_{u \in N^{1}_{s}}  \Phi_{\sqrt{2}}\big(t-s, X_u(s) + \sqrt{2}(t-s)-  m^{(1,1)}_{h,1}(t)   \big) 1_{\{ (s,X_{u}(s) )\in \Omega^{R}_{t,h}  \} }  \dif s    \\
  & = (1+o(1))  \gamma_{\sqrt{2}}(\varphi)  \int_{0}^{t} \frac{t^{3/2}  e^{- 2(\beta_{t}-1)(t-s)}   }{(t-s)^{3/2} t^{h'}} \sum_{u \in N^{1}_{s}}1_{\{ (s,X_{u}(s) )\in \Omega^{R}_{t,h}  \} }  \\
  & \qquad  \qquad  \qquad  \qquad \qquad\times   [\sqrt{2} \beta_{t} s-X_{u}(s) ] e^{ \sqrt{2}X_{u}(s)-2 \beta_{t}s} e^{-\frac{(X_{u}(s)-\sqrt{2}\beta_{t} s)^2}{2(t- s)}}    \dif s.
 \end{aligned}
 \end{equation}

\textbf{Case 1: $h \in (0,1)$}.
 In this case we need to slightly modify $\Omega^{R}_{t,h}$  to $\widetilde{\Omega}^{R}_{t,h}= \{(s,x): (t-s)/t^{h} \in  \Gamma_{R},  \frac{|X_{u}(s)-\sqrt{2}\beta_{t}s|}{\sqrt{t-\beta_{t} s}} \in \Gamma_{R} \}$. In the argument below,  we abuse the notation since both Lemma \ref{lem-contribution-area-11},  Corollary \ref{cor-reduce-to-E-R} and previous argument in this proof still hold for $\widetilde{\Omega}^{R}_{t,h}$.
By the Brownian scaling $\left( X_u(s): u \in N_s^1\right) \overset{law}{=} \left( \frac{\sigma_{t}}{\sqrt{\beta_{t}}} \mathsf{X}_u(s') : u \in \mathsf{N}_{s'}\right)$ where $s'=\beta_{t}s$.
Let $G(x)= G_{R}(x)=   x e^{-\frac{x^2}{2}} 1_{ \{ x \in \Gamma_{R}\}}$ (Recall that $\Gamma_{R}:=[R^{-1},R]$). Hence
for $(t-s)/t^{h} \in  \Gamma_{R} $ we have
  \begin{align*}
   & \sum_{u \in N^{1}_{s}}  e^{ \sqrt{2}X_{u}(s)-2 \beta_{t}s} \frac{ \sqrt{2} \beta_{t} s-X_{u}(s) }{\sqrt{t-s}} e^{-\frac{(X_{u}(s)-\sqrt{2}\beta_{t} s)^2}{2(t- s)}}    1_{\{ (s,X_{u}(s) )\in \Omega^{R}_{t,h}  \} }\\
    &= [1+o(1)] \sum_{u \in N^{1}_{s}}  e^{ \sqrt{2}X_{u}(s)-2 \beta_{t}s} \frac{\sqrt{2} \beta_{t} s-X_{u}(s) }{\sqrt{t-\beta_{t} s}} e^{-\frac{(X_{u}(s)-\sqrt{2}\beta_{t} s)^2}{2(t-\beta_{t}s)}}   1_{\left\{  \frac{|X_{u}(s)-\sqrt{2} s'|}{\sqrt{t-s'}} \in \Gamma_{R}   \right\} } \\
    &=  \sum_{u \in \mathsf{N}_{s'}} e^{\sqrt{2}\mathsf{X}_{u}(s')-2 s'} G\left( \frac{\sqrt{2}s'- \mathsf{X}_{u}(s')}{\sqrt{t-s'}} \right) =: \mathsf{W}^{G}(s',t).
  \end{align*}
  Making a change of variable $s = t-\lambda t^{h}$, we get
   \begin{equation*}
  \mathbb{E}\left(e^{-\left\langle\widehat{\mathcal{E}}_t^R, \varphi\right\rangle}\right)   =
  \mathsf{E}
  \bigg(\exp \bigg\{ -  [1+o(1)] \gamma_{\sqrt{2}}(\varphi) \int_{\frac{1}{R} }^{R} e^{-2(\beta_{t}-1)t^{h} \lambda} \, \frac{t^{3/2}}{\lambda t^{h}}    \mathsf{W}^{G}(\beta_{t}(t-\lambda  t^{h}), t )   \dif \lambda \bigg\} \bigg).
   \end{equation*}
Let $h_{t}= \frac{\sqrt{t-s' }} {\sqrt{s'}}$ where $s'=\beta_{t}s=\beta_{t}(t-\lambda t^{h})$.  Applying Lemma \ref{lem-functional-convergence-derivative-martingale}  we have
 \begin{align*}
  & \lim_{t \to \infty} \sqrt{t} \, \frac{\mathsf{W}^{G}(t- \lambda \beta_{t}t^{h},t)  }{ \int_{0}^{\infty} G(z/h_{t}) z e^{-\frac{z^2}{2}} \dif z }   = \sqrt{ \frac{2}{\pi}} \mathsf{Z}_{\infty} \\
  &\qquad  \Longrightarrow \lim_{t \to \infty}  \frac{t^{3/2}}{\lambda t^{h}}  \mathsf{W}^{G}(t- \lambda \beta_{t}t^{h},t)   = \sqrt{ \frac{2}{\pi}} \,   \mathsf{Z}_{\infty}  \int_{\frac{1}{R}}^{R} y^{2} e^{-\frac{y^2}{2}} \dif y .
 \end{align*}
 Letting $t \to \infty$ then $R\to \infty$, applying the dominated convergence theorem and part (ii) of  Corollary \ref{cor-reduce-to-E-R}, we have
  \begin{align*}
   &\lim_{t \to \infty}  \mathbb{E}\left(e^{-\left\langle\widehat{\mathcal{E}}_t, \varphi\right\rangle}\right)= \lim_{R \to \infty}  \lim_{t \to \infty}  \mathbb{E}\left(e^{-\left\langle\widehat{\mathcal{E}}_t^R, \varphi\right\rangle}\right)\\
    & = \lim_{R \to \infty}    \mathsf{E}\bigg(\exp \bigg\{ -  \gamma_{\sqrt{2}}(\varphi)\mathsf{Z}_{\infty} \int_{\frac{1}{R} }^{R}    e^{-\lambda}     \dif \lambda  \sqrt{ \frac{2}{\pi}}\int_{\frac{1}{R}}^{R} z^2 e^{-\frac{z^2}{2}} \dif z\bigg\} \bigg)  =    \mathsf{E}  \left[ e^{   - \gamma_{\sqrt{2}}(\varphi) \mathsf{Z}_{\infty}  }  \right] ,
  \end{align*}
  which is the Laplace functional of $\operatorname{DPPP}\left( \sqrt{2}C_{\star}   \mathsf{Z}_{\infty}   e^{-\sqrt{2} x} \dif x, \mathfrak{D}^{\sqrt{2}}\right)$ (by the definition of $\gamma_{\sqrt{2}}(\varphi)$ in Lemma \ref{thm-Laplace-BBM-order}).
  By \cite[Lemma 4.4]{BBCM22}, we complete the proof for the case $h \in (0,1)$.

  \textbf{Case 2: $h \in [1,\infty]$.}
   Making change of variable $s= \lambda t$, the integral in \eqref{eq-integral-Phi-2} equals
  \begin{align*}
  \int_{\frac{1}{R}}^{1-\frac{1}{R}}   \frac{ e^{- 2(\beta_{t}-1)(1-\lambda)t} }{(1-\lambda)^{3/2} } &\sum_{u \in N^{1}_{s}}  1_{\{ v_{t}s-X_{u}(s) \in [\frac{1}{R}\sqrt{t}, R\sqrt{t}]  \} }  \\
      &\qquad  \times  e^{ \sqrt{2}X_{u}(s)-2 \beta_{t}s} [\sqrt{2} \beta_{t} s-X_{u}(s) ] e^{-\frac{(X_{u}(s)-\sqrt{2}\beta_{t} s)^2}{2(t- s)}}   \dif \lambda.
   \end{align*}
   Let $ G_{\lambda}(x)= G_{\lambda,R}(x):= x e^{-\frac{\lambda}{2(1-\lambda)} x^2} 1_{ \{ \sqrt{\lambda }x \in \Gamma_{R}\}}$.
 Similar as the explanation at the beginning of case 1, with abuse of notation,  for
   $\lambda \in [\frac{1}{R} ,(1-\frac{1}{R})  ]$,  $s':= \beta_{t} s=\beta_{t} \lambda t$,  we have
  \begin{align*}
   & \sum_{u \in N^{1}_{s}}  e^{ \sqrt{2}X_{u}(s)-2 \beta_{t}s} \frac{ \sqrt{2} \beta_{t} s-X_{u}(s) }{\sqrt{\beta_{t} s}} e^{-\frac{(X_{u}(s)-\sqrt{2}\beta_{t} s)^2}{2(t- s)}}  1_{\{ v_{t}s-X_{u}(s) \in [\frac{1}{R}\sqrt{t}, R\sqrt{t}] \}} \\
    &\sim   \sum_{u \in N^{1}_{s}}
     e^{ \sqrt{2}X_{u}(s)-2 \beta_{t}s}
     \frac{\sqrt{2} \beta_{t} s-X_{u}(s) }{\sqrt{\beta_{t} s}} e^{-\frac{\lambda}{1-\lambda}\frac{(X_{u}(s)-\sqrt{2}\beta_{t} s)^2}{2 \beta_{t}s }}  1_{\{ \frac{ \sqrt{2}\beta_{t}s-X_{u}(s) }{\sqrt{\beta_{t} s}}\in [\frac{1}{R} \frac{\sqrt{t}}{\sqrt{s}}, R \frac{\sqrt{t}}{\sqrt{s}}  ] \}} \\
    &=  \sum_{u \in \mathsf{N}_{s'}} e^{\sqrt{2}\mathsf{X}_{u}(s')-2 s'} G_{\lambda}\left( \frac{\sqrt{2}s'- \mathsf{X}_{u}(s')}{\sqrt{s'}} \right) =: \mathsf{W}^{G_{\lambda}}(s' ;t) .
  \end{align*}
Therefore we get that  $   \mathbb{E}\left(e^{-\left\langle\widehat{\mathcal{E}}_t^R, \varphi\right\rangle}\right)$ equals
   \begin{equation*}
   \mathsf{E}
    \bigg(\exp \bigg\{ -  [1+o(1)] \gamma_{\sqrt{2}}(\varphi) \int_{\frac{1}{R}}^{1-\frac{1}{R}}   \frac{ e^{- 2(\beta_{t}-1)(1-\lambda)t} }{(1-\lambda)^{3/2} } \sqrt{\lambda \beta_{t} t}  \mathsf{W}^{G_{\lambda}}( \lambda \beta_{t}t, t)  \dif \lambda    \bigg\} \bigg).
   \end{equation*}
 Applying Lemma \ref{lem-functional-convergence-derivative-martingale} we get
 \begin{equation*}
  \lim_{t \to \infty}  \sqrt{ \lambda \beta_{t}t} \, \mathsf{W}^{G_{\lambda}}( \lambda \beta_{t}t, t)   =   \mathsf{Z}_{\infty} \sqrt{ \frac{2}{\pi}} \int_{0}^{\infty} G_{\lambda}(z) z e^{-\frac{z^2}{2}} \dif z  =   \mathsf{Z}_{\infty} \sqrt{ \frac{2}{\pi}} \,
  \int_{\frac{1}{R\sqrt{\lambda}}}^{ \frac{R}{\sqrt{\lambda}} } z^{2} e^{-\frac{z^2}{2(1-\lambda)}} \dif z
 \end{equation*}
 in probability. Let $
 C_{h,1} = 1_{\{ h > 1\}} + (1-e^{-1})1_{\{h=1\}}$.
 Letting $t \to \infty$ then $R\to \infty$ and by  the dominated convergence theorem and  Corollary \ref{cor-reduce-to-E-R}   we have
 \begin{equation*}
  \begin{aligned}
  & \lim _{t \rightarrow \infty} \mathbb{E}\left(e^{-\left\langle\widehat{\mathcal{E}}_t, \varphi\right\rangle}\right)=\lim _{R \rightarrow \infty} \lim _{t \rightarrow \infty} \mathbb{E}\left(e^{-\left\langle\widehat{\mathcal{E}}^{R}_t, \varphi\right\rangle}\right) \\
  & =\lim _{R \rightarrow \infty}
   \mathsf{E}
  \left(\exp \left\{ -\gamma_{\sqrt{2}}(\varphi) \mathsf{Z}_{\infty}
  \int_{\frac{1}{R}}^{1-\frac{1}{R}}
  \frac{e^{-\lambda 1_{\{h=1\}} }} {(1-\lambda)^{3 / 2}} \dif\lambda \sqrt{\frac{2}{\pi}} \int_{\frac{1}{R \sqrt{\lambda}}}^{\frac{R}{\sqrt{\lambda}}} z^2 e^{-\frac{z^2}{2(1-\lambda)}} \mathrm{d} z\right\}\right) \\
  & = \mathsf{E}
  \left(e^{-C_{h,1} \gamma_{\sqrt{2}}(\varphi) Z_{\infty}}\right),
  \end{aligned}
  \end{equation*}
  which is the Laplace functional of $\operatorname{DPPP}\left( C_{h,1} \sqrt{2}C_{\star}   \mathsf{Z}_{\infty}   e^{-\sqrt{2} x} \dif x, \mathfrak{D}^{\sqrt{2}}\right)$.
  By \cite[Lemma 4.4]{BBCM22}, we complete the proof for the case $h \in [1,\infty]$.
\end{proof}

\section{ Approaching $\mathscr{B}_{I,III}, \mathscr{B}_{II,III}$ and $(1,1)$ from $\mathscr{C}_{III}$}

We first introduce
several important constants  introduced in \cite{BM21}.
For $\left( \beta,\sigma^2\right) \in \mathscr{C}_{I I I}$,   we set
\begin{equation} \label{eq-def-star-a-b-p}
  \begin{aligned}
 b^{*}(\beta,\sigma^2) & :=\sqrt{2 \frac{\beta-1}{1-\sigma^2}} , \   a^{*}(\beta,\sigma^2):=\sigma^2 b(\beta,\sigma^2), \      p^{*}(\beta,\sigma^2):=\frac{\sigma^2+\beta-2}{2(\beta-1)\left(1-\sigma^2\right)}  ;  \\
 &  v^{*}(\beta,\sigma^2) :=a^{*}  p^{*} +b^{*} (1- p^{*})= \frac{\beta-\sigma^2}{\sqrt{2(\beta-1)\left(1-\sigma^2\right)}} .
  \end{aligned}
\end{equation}
Moreover we have
\begin{align}
\left(\beta -\frac{ (a^{*}) ^2}{2\sigma
^2 }\right)p^{*}  + \left(1-\frac{(b^{*})^2}{2} \right)(1-p^{*} )& =0;  \label{eq-condition-zero}\\
 b^{*} v^{*} -    \beta  -  \frac{  \sigma^{2}(b^{*}) ^{2}}{2 }  & = 0. \label{eq-condition-zero-2}
\end{align}
 For the sake of simplicity,
we will write
\begin{equation}\label{eq-def-a-b-p}
   b_{t}=b^{*}(\beta_{t},\sigma^2_{t}) ,\quad a_{t}:=a^{*}(\beta_{t},\sigma^2_{t}) ,\quad p_{t}=p^{*}(\beta_{t},\sigma^2_{t})  ,\quad v^{*}_{t} = v^{*}(\beta_{t},\sigma^2_{t}) .
\end{equation}

\begin{lemma}\label{lem-computation}
Let $s=p_{t} t +u \in (0,t)$, $y:= (\sqrt{2}-a_{t})s+(v^{*}_{t}-\sqrt{2})$ and
  \begin{equation}
L(u,t) = (\beta_{t}-\frac{a_t^2}{2\sigma
    ^2_{t}})s - \sqrt{2} y - \frac{y^2}{2(t-s)}  .
  \end{equation}
\begin{enumerate}[(i)]
  \item For $(\beta,\sigma^{2}) \in \mathscr{B}_{II,III}$ and $(\beta_{t},\sigma^{2}_{t})_{t>0} \in \mathscr{A}^{h,+}_{(\beta,\sigma^2)}$. We have
 \begin{equation*}
  L(\xi \sqrt{t},t) = -(1-\sigma^2)^{2}\xi^2 - R(\xi,t),
 \end{equation*}
  where for each fixed $\xi$, $ R(\xi,t) \to 0$  as  $t \to \infty$; and there is some  $c >0$ such that   $ L(\xi \sqrt{t},t)  \leq  - c \xi^{2}$ for all $\xi$ satisfying that  $p_{t}t+\xi \sqrt{t} \in (0,t)$.

 \item  For  $(\beta_{t},\sigma^{2}_{t})_{t>0} \in \mathscr{A}^{h,3}_{(1,1)}$ and fixed  $h \in (0,1)$,
 \begin{equation*}
  L(\xi t^{\frac{1+h}{2}}, t) = -(\sqrt{2}+1)\xi^2-R(\xi,t),
 \end{equation*}
 where for each fixed $\xi$, $ R(\xi,t) \to 0$  as  $t \to \infty$; and there is some  $c >0$ such that   $ L(\xi t^{\frac{1+h}{2}},t)  \leq  - c \xi^{2}$ for all $\xi$ satisfying that  $p_{t}t+\xi \sqrt{t} \in (0,t)$.
  \end{enumerate}
  \end{lemma}

The proof of Lemma \ref{lem-computation}  is postponed to Appendix \ref{App-B}.

\subsection{Approaching $\mathscr{B}_{II,III}$ from $\mathscr{C}_{III}$}

Assume that $(\beta,\sigma^2)\in \mathscr{B}_{II,III}$, and   $(\beta_{t},\sigma^{2}_{t})_{t>0} \in
\mathscr{A}^{h,+}_{(\beta,\sigma^2)}$.
Combining with \eqref{eq-def-star-a-b-p}   we have
\begin{equation}\label{eq-asymptotic-abv}
   p_{t}\sim \frac{1}{2(1-\sigma^2)^2t^{h}} , \  \frac{b_{t}}{\sqrt{2}}-1  \sim \frac{1}{2(1-\sigma^2)t^{h}}  \text{ and }  v^{*}_{t}=\sqrt{2}+\frac{1}{b_{t}}\left(\frac{b_{t}}{\sqrt{2}}-1 \right)^2.
\end{equation}
In fact, as $t\to\infty$,
$p_{t}= \frac{\sigma^2_{t}+\beta_{t}-2}{2(\beta_{t}-1)\left(1-\sigma^2_{t}\right)} = \frac{ 1}{2(\beta_{t}-1)\left(1-\sigma^2_{t}\right)t^{h}}\sim \frac{ 1}{2(\beta -1)\left(1-\sigma^2\right)t^{h}}= \frac{ 1}{2\left(1-\sigma^2\right)^2t^{h}}$, $\frac{b_{t}}{\sqrt{2}}-1= \sqrt{\frac{\beta_{t}-1}{1-\sigma^{2}_{t}}} -1 = \sqrt{\frac{1-\sigma_{t}^{2}+t^{-h}}{1-\sigma^{2}_{t}}} -1 \sim \frac{1}{2(1-\sigma_{t}^{2})t^{h}} \sim \frac{1}{2(1-\sigma^{2})t^{h}}$ and  $v^{*}_{t}=\frac{\beta_{t}-1+1-\sigma^2_{t}}{\sqrt{2(\beta_{t}-1)\left(1-\sigma^2_{t}\right)}}=\frac{1}{\sqrt{2}}(\frac{b_{t}}{\sqrt{2}} + \frac{\sqrt{2}}{b_{t}})=\sqrt{2}+\frac{1}{b_{t}}\left(\frac{b_{t}}{\sqrt{2}}-1 \right)^2$.

Let $ m^{2,3}_{h,+}(t):=   v^{*}_{t}  t - \frac{h' }{ \sqrt{2}} \log t$, where   $h'  = \min\{ h,1/2 \}$,
Define
\begin{equation}\label{eq-def-Omega-t-R-2}
  \Omega_{t,h}^{R}=
  \begin{cases}
    \{ (s,x): | s-p_{t} t| \leq    R \sqrt{t}, \   | x - a_{t} s | \leq R \sqrt{s} \} \text{ if }  h \in (0,\frac{1}{2}) , \\
  \{ (s,x):   s \in [\frac{1}{R} \sqrt{t}, R \sqrt{t}],  | x - a_{t} s | \leq R \sqrt{s} \}  \text{ if }  h \in [\frac{1}{2}, \infty].
\end{cases}
\end{equation}

\begin{lemma}\label{lem-contribution-area-2}
  For all $A>0$,
  \begin{equation*}
  \lim _{R \rightarrow \infty} \limsup _{t \rightarrow \infty} \mathbb{P}\left(\exists u \in N^{2}_{t}: X_{u}(t) \geq m^{2,3}_{h,+}(t) - A ,  (T_{u} , X_{u}(T_{u}) )  \notin \Omega_{t,h}^{R}  \right)=0 .
  \end{equation*}
  \end{lemma}

  \begin{proof}
    Applying  Corollary \ref{cor-reduce-to-E-R} with $m(t)=m^{2,3}_{h,+}(t)$,  and $\Omega^{R}_{t,h}$ defined in \eqref{eq-def-Omega-t-R-2}, it suffices to show that  for each $A, K>0$, $I(t,R) = I(t,R;A,K)$  defined in \eqref{eq-assumption-EYtARK}  vanishes as first $t \to \infty$ and then $R \to \infty$. Conditioned on the Brownian motion $B_{s}$ in \eqref{eq-assumption-EYtARK} equals $ \frac{a_{t}}{\sigma_{t}} s + x $,    we have
\begin{equation} \label{eq-def-I-R-t-2}
  \begin{aligned}
 I(t,R)  &= \int_{0}^{t} \dif s \int_{-\infty}^{\frac{v_{t}-a_{t}}{\sigma_{t}} s + K} \mathbf{P} \left(  B_{r} \leq \sqrt{2\beta_{t}} r+  K, \forall r \leq s | B_{s} =  \frac{a_{t}}{\sigma_{t}} s + x   \right)     \\
 & \qquad \qquad\qquad   e^{\beta_{t}s}  \mathsf{F}_{t}\left(t-s, v_{t} s -\sigma_{t} x\right) 1_{\left\{(s ,  v_{t}s - \sigma_{t}x ) \notin  \Omega^{R}_{t,h} \right\}} e^{-\frac{ (  a_{t} s/\sigma_{t}  + x )^2 }{2s}} \frac{\dif x}{\sqrt{2\pi s}}   \\
  &   \lesssim     \int_{0}^{t}\dif s  \int_{-\infty}^{\frac{v_{t}-a_{t}}{\sigma_{t}} s + K}  \mathsf{F}_{t}\left(t-s, a_{t} s +\sigma_{t} x\right) 1_{\left\{(s ,  a_{t}s+\sigma_{t}x ) \notin  \Omega^{R}_{t,h} \right\}}   e^{(\beta_{t}   -\frac{a_{t}^2}{2\sigma^{2}_{t}} )s -\frac{a_{t}}{\sigma_{t}} x  -\frac{x^2}{2s} }  \frac{\dif x}{\sqrt{ s}}  .
  \end{aligned}
\end{equation}

Now we aim to get an upper bound  for $\mathsf{F}_{t}\left(t-s, a_{t} s +\sigma_{t} x\right)$.
  Let $\mathrm{y}:= (\sqrt{2}-a_{t})s+(v^{*}_{t}-\sqrt{2})t  $ and $\mathrm{w}:=\frac{h'}{\sqrt{2}} \log t -\frac{3}{2\sqrt{2}}\log(t-s+1) +A$.  By Lemma \ref{lem-Max-BBM-tail},  provided that $\mathrm{y}-\sigma_{t}x-\mathrm{w}>1$, we have
 \begin{equation}\label{eq-bound-F-2}
   \begin{aligned}
& \mathsf{F}_{t}\left(t-s,  a_{t} s + \sigma_{t} x \right)
= \mathsf{P} \left( \max\limits_{u \in \mathsf{N}_{t-s}} \mathsf{X}_{u}(t-s) \geq   v^{*}_{t} t - a_{t}s-  \sigma_{t} x -  \frac{h'}{\sqrt{2}} \log t-A \right)\\
&  =\mathsf{P} \left( \max\limits_{u \in \mathsf{N}_{t-s}} \mathsf{X}_{u}(t-s) > \sqrt{2}(t-s) - \frac{3}{2\sqrt{2}} \log (t-s+1) +\mathrm{y}-\sigma_{t}x-\mathrm{w} \right) \\
& \lesssim_{A}  (\mathrm{y}-\sigma_{t}x-\mathrm{w}) \frac{t^{h'}}{(t-s+1)^{3/2}}\exp\left\{ -\sqrt{2}\mathrm{y}  + \sqrt{2}\sigma_{t}x-\frac{1}{2(t-s)}[\mathrm{y}-\sigma_{t}x- \widetilde{\mathrm{w}}  ]^2 \right\},
 \end{aligned}
 \end{equation}
 where $\widetilde{\mathrm{w}}:= \mathrm{w}- \frac{3}{2\sqrt{2}} \log(t-s+1)$.
 We claim that for large $t$ we always have  $\mathrm{y}-\sigma_{t}x-\mathrm{w}>1$ for $s \in (0,t)$ and $\sigma_{t}x \leq (v_{t}-a_{t})s + \sigma_{t}K $. In fact,  $\mathrm{y}-\sigma_{t}x-\mathrm{w}> (\sqrt{2} - v_{t})s - \frac{h'}{\sqrt{2}} \log t + \frac{3}{2\sqrt{2}}\log(t-s+1)- O(1)$. By our assumption $(\beta_{t},\sigma_{t}^2) \to (\beta,\sigma^2) \in \mathscr{B}_{II,III}$,  we have  $   \sqrt{2}-v_{t} > \delta >0$ for large $t$. Then  for each  $\delta s > 2\log t$,
$\mathrm{y}-\sigma_{t}x-\mathrm{w}>2\log t - \frac{h'}{\sqrt{2}} \log t-O(1)  > 1$;
  for each  $\delta s \leq  2\log t$,
 we have $t-s+1 \geq t/2$, and  hence $  \mathrm{y}-\sigma_{t}x-\mathrm{w}> \frac{3}{2\sqrt{2}}\log(\frac{t}{2})-\frac{h'}{\sqrt{2}} \log t- O(1) >1$.

Substituting \eqref{eq-bound-F-2} into \eqref{eq-def-I-R-t-2}, we get
\begin{equation}  \label{eq-bound-I-t-R-1}
   \begin{aligned}
 I(t,R)  \lesssim   & \int_{0}^{t}  \frac{t^{h'}}{(t-s+1)^{3/2} s^{1/2}} \exp\left\{ (\beta_{t}-\frac{a_t^2}{2\sigma
  ^2_{t}}) s - \sqrt{2} \mathrm{y} - \frac{\mathrm{y}^2 }{2(t-s)} \right\} \dif s \\
  &    \times  \int_{\mathbb{R}}   | \mathrm{y}-\sigma_{t}x-\mathrm{w} |
 \exp \left\{(\sqrt{2}\sigma_{t}-\frac{a_{t}}{\sigma_{t}}+\frac{\mathrm{y}\sigma_{t}}{t-s})x \right\}
  e^{\frac{-\widetilde{\mathrm{w}} \mathrm{y}}{t-s}}
  e^{-\frac{[\sigma
  _{t}x+\widetilde{\mathrm{w}}]^2}{2(t-s)} -\frac{x^2}{2s}}
1_{\left\{(s ,  a_{t}s+\sigma_{t}x ) \notin  \Omega^{R}_{t,h} \right\}}  \dif x.
 \end{aligned}
\end{equation}
 Making a change of variable  $s=p_{t} t+ \xi \sqrt{t}$,  by \eqref{eq-def-star-a-b-p}, we have  $\mathrm{y}=(b_{t}-\sqrt{2})(1-p_{t})t+(\sqrt{2}-a_{t})\xi \sqrt{t}$. Thanks to Lemma \ref{lem-computation},
   \begin{equation}\label{eq-estimate-I1-1}
  (\beta_{t}-\frac{a_t^2}{2\sigma
    ^2_{t}}) s - \sqrt{2} \mathrm{y} - \frac{\mathrm{y}^2}{2(t-s)}   =   L(\xi \sqrt{t},t)=  - (1-\sigma^2)^{2} \xi^2 - R(\xi,t).
   \end{equation}
  Moreover,  since  $a_{t}/\sigma_t=\sigma_t b_t$ (see \eqref{eq-def-star-a-b-p}), the coefficient for the term $x$ in \eqref{eq-bound-I-t-R-1} is
  \begin{equation}\label{eq-estimate-I1-2}
  \sqrt{2}\sigma_{t}-\frac{a_{t}}{\sigma_{t}}+\frac{\mathrm{y}\sigma_{t}}{t-s} =
  \frac{\sigma
    _t(b_{t}-a_t)}{t-s} \xi \sqrt{t} .
  \end{equation}
Let $\Gamma^{R}_{h}=[-R,R]$ if $h\in (0,1/2)$ and $\Gamma^{R}_{h}=[  R^{-1},R]$ if $h \in [1/2,\infty]$.
 Note that $(s ,  a_{t}s+\sigma_{t}x ) \in  \Omega^{R}_{t,h}$ if and only if $ \xi +p_{t}\sqrt{t} 1_{\{ h \geq 1/2\}}   \in \Gamma^{R}_{h}$ and $|\sigma_{t} x|  \leq R\sqrt{s}$.
 Combining equalities \eqref{eq-estimate-I1-1} and \eqref{eq-estimate-I1-2}, we have
   \begin{multline*}
    I_{1}(t,R) =\int_{-p_{t} \sqrt{t}}^{(1-p_t)\sqrt{t}}     \frac{t^{h'} t^{1/2}}{(t-s)^{3/2} s^{1/2}}
    e^{  - (1-\sigma^2)^{2} \xi^2 - R(\xi,t) }  \dif \xi \\
      \times \int_{\mathbb{R}}   1_{\left\{
  \substack{  p_{t}\sqrt{t} 1_{\{ h \geq 1/2\}}   +\xi  \notin  \Gamma^{R}_{h}  \\
   \text{ or } |x|>R\sqrt{s} }\right\}}
     | \mathrm{y}-\sigma_{t}x-\mathrm{w}| \, e^{ \frac{\sigma
    _t(b_{t}-a_t)\sqrt{t}}{t-s} \xi  x  }
    \, e^{\frac{-\widetilde{\mathrm{w}} \mathrm{y}}{t-s}} \,
    e^{-\frac{[\sigma
    _{t}x+\widetilde{\mathrm{w}}]^2}{2(t-s)} -\frac{x^2}{2s}}
     \dif x .
   \end{multline*}
   Now  again making a change of variable $x=\eta \sqrt{s}$ (where $s= p_{t}t+\xi \sqrt{t}$),   we get
   \begin{multline*}
    I_{1}(t,R) \lesssim \int_{-p_{t} \sqrt{t}}^{(1-p_t)\sqrt{t}}     \frac{t^{h'+1/2}}{(t-s)^{3/2}} e^{  - (1-\sigma^2)^{2} \xi^2 - R(\xi,t) }  \dif \xi \\
    \times \int_{\mathbb{R}}   1_{\left\{
   \substack{   p_{t}\sqrt{t} 1_{\{ h \geq 1/2\}}   +\xi   \notin  \Gamma^{R}_{h} \\
     \text{ or } |\eta|>R  }\right\}}
     | \mathrm{y}-\sigma_{t}\eta \sqrt{s}-\mathrm{w}| \, e^{ \frac{\sigma
     _t(b_{t}-a_t)\sqrt{ts}}{t-s} \xi \eta  } \,
     e^{\frac{-\widetilde{\mathrm{w}} \mathrm{y}}{t-s}} \,
     e^{-\frac{[\sigma
     _{t}\eta \sqrt{s}+\widetilde{\mathrm{w}}]^2}{2(t-s)} -\frac{\eta^2}{2}}
      \dif \eta .
    \end{multline*}
 For fixed $\xi$, by \eqref{eq-asymptotic-abv}, $p_{t}=\Theta(t^{-h})$  and  $b_{t}-\sqrt{2}=\Theta(t^{-h})$,
 and then
  $s=O(t^{1-h})+O(\sqrt{t})= O(t^{1-h'})$,  $\mathrm{y}=(b_{t}-\sqrt{2})(1-p_{t})t+(\sqrt{2}-a_{t})\xi \sqrt{t} =O(t^{1-h'})$ and  $\mathrm{w}= O(\log t)=\widetilde{\mathrm{w}}$. Let $\xi_{0}(h)=1_{\{ h\geq 1/2\}} \lim_{t\to\infty} p_{t}\sqrt{t} =   \frac{1}{2(1-\sigma^2)^2} 1_{\{ h=1/2\}} $.
 Then the dominated convergence theorem  yields that
  \begin{equation*}
  \limsup_{t \to \infty} I_{1}(t,R) \lesssim \int_{-\lim_{t} p_{t}\sqrt{t} }^{\infty} e^{ -  (1-\sigma^{2})^{2}\xi^2 }   \dif \xi \int_{\mathbb{R}}
     C(h,\xi)
  e^{  -\frac{\eta^2}{2}}  1_{\left\{
    \substack{   \xi \notin  \Gamma^{R}_{h}-\xi_{0}(h)  \\ \text{ or }
     |\eta|>R  }\right\}}
  \dif \eta ,
  \end{equation*}
  where   for fixed $\xi$, by \eqref{eq-asymptotic-abv},  $ C(h,\xi):= \lim_{t}
   \frac{| \mathrm{y}-\sigma_{t}\eta \sqrt{s}-\mathrm{w}| }{t^{1-h'}} =\lim_{t}
   \frac{ \mathrm{y}  }{t^{1-h'}}  = \frac{1_{\{ h \leq 1/2 \}}}{\sqrt{2}(1-\sigma^2)}  +  \sqrt{2}(1-\sigma^2)\xi     1_{ \{h \geq 1/2 \} } $. Finally, letting $R \to \infty $, the desired result follows.
  \end{proof}

We now show Theorem \ref{thm-23-approximate}
for the case that  for $(\beta_{t},\sigma^{2}_{t})_{t>0} \in
{\mathscr{A}}^{h,+}_{(\beta,\sigma^2)}$.

\begin{proof}[Proof of Theorem \ref{thm-23-approximate}
for $(\beta_{t},\sigma^{2}_{t})_{t>0} \in
{\mathscr{A}}^{h,+}_{(\beta,\sigma^2)}$
]
  Take $\varphi \in \mathcal{T}$. Applying
  Corollary \ref{cor-reduce-to-E-R} with $m(t)=m^{2,3}_{h,+}(t)$, $\rho=\sqrt{2}$, and $\Omega^{R}_{t,h}$ defined in \eqref{eq-def-Omega-t-R-2}, it suffices to study the asymptotic behavior of $ \mathbb{E}\left(e^{-\left\langle\widehat{\mathcal{E}}_t^R, \varphi\right\rangle}\right) $ that equals
\begin{equation*}
\mathbb{E}\bigg(\exp \bigg\{ -  \int_{0}^{t} \sum_{u \in N^{1}_{s}}   \Phi_{\sqrt{2}}\big(t-s, X_{u}(s)-a_{t}s-\mathrm{y}+ \frac{h'}{\sqrt{2}}\log t  \big) 1_{\{ (s,X_{u}(s)) \in \Omega_{t,h}^{R} \} }  \dif s \bigg\} \bigg) .
\end{equation*}
where we used the fact that
  $\sqrt{2}(t-s)-m^{2,3}_{h,+}(t)= -a_{t}s-\mathrm{y}+ \frac{h'}{\sqrt{2}}\log t$ and  $\mathrm{y}: = (\sqrt{2}-a_{t})s+(v^{*}_{t}-\sqrt{2})t$.
 Moreover,
by Lemma \ref{thm-Laplace-BBM-order}, uniformly for
 $(s,X_{u}(s)) \in \Omega_{t,h}^{R}$, $s= \Theta(t^{1-h'})$ and  $|X_{u}(s)-a_{t}s| = \Theta( \sqrt{s})$, we have
\begin{equation}\label{eq-Phi-bound-11}
    \begin{aligned}
     & \Phi_{\sqrt{2}}\left(t-s, X_{u}(s)-a_{t}s+ \frac{h'}{\sqrt{2}}\log t- \mathrm{y} \right)\\
      & \sim  \frac{\gamma_{\sqrt{2}} (\varphi)\, \mathrm{y}}{(t -s)^{3/2}} t^{h'}
      \exp \left\{    \sqrt{2}X_{u}(s)-\sqrt{2} a_t s  -\sqrt{2}\mathrm{y} -\frac{[X_{u}(s)-a_{t}s-\mathrm{y}-\Theta(\log t)]^2}{2(t-s)}  \right\}
        \\
      &\sim  \gamma_{\sqrt{2}}(\varphi) \frac{ t^{h'} \mathrm{y}}{t^{3/2}} \exp \left\{ \sqrt{2}(X_{u}(s)-  a_t s)  -\sqrt{2}\mathrm{y}-\frac{\mathrm{y}^2}{2(t-s)} -\frac{\mathrm{y} (X_u(s)-a_{t}s)}{t-s}  \right\}.
    \end{aligned}
\end{equation}
Make a change of variable $s= p_{t} t+ \xi \sqrt{t}$. On the one hand,
 \begin{equation}\label{eq-Phi-bound-12}
  \begin{aligned}
       \frac{\mathrm{y}(X_u(s)-a_{t}s) }{t-s}&=(b_t-\sqrt{2}) (X_u(s)-a_{t}s) + \frac{(b_t-a_t)\xi \sqrt{t}}{t-s} (X_u(s)-a_{t}s) \\
  &=  (b_t-\sqrt{2}) (X_u(s)-a_{t}s) + o(1) .
   \end{aligned}
\end{equation}
  On the other hand, by Lemma \ref{lem-computation}, we have
  \begin{equation}\label{eq-Phi-bound-13}
    -\sqrt{2}\mathrm{y}-\frac{\mathrm{y}^2}{2(t-s)}   = -(\beta_t-\frac{a_t^2}{2 \sigma^2_{t}})s -(1-\sigma^2)^{2}\xi^2+o(1).
  \end{equation}
 So  combining \eqref{eq-Phi-bound-11}, \eqref{eq-Phi-bound-12} and  \eqref{eq-Phi-bound-12},  we have,  uniformly for $(s,X_{u}(s)) \in \Omega_{t,h}^{R}$,
    \begin{align*}
 & \Phi_{\sqrt{2}}\big(t-s, X_{u}(s)-a_{t}s+ \frac{h'}{\sqrt{2}}\log t- \mathrm{y} \big) \\
    &=(1+o(1))  \gamma_{\sqrt{2}}(\varphi) \frac{ t^{h'} \mathrm{y}}{t^{3/2}}  e^{ b_{t} X_{u}(s)- b_{t} a_{t} s  -(\beta_t-\frac{a_t^2}{2 \sigma^2_{t}})s   }  e^{ -  (1-\sigma^2)^{2}\xi^2 } \\
    &=(1+o(1))  \gamma_{\sqrt{2}}(\varphi) \frac{ t^{h'} \mathrm{y}}{t^{3/2}}  e^{ b_{t}X_{u}(s)- (\beta_{t}+\frac{\sigma^2_{t} b^{2}_{t}}{2}) s }  e^{ -  (1-\sigma^2)^{2}\xi^2 } ,
   \end{align*}
 where we used the fact $a_{t}= \sigma^{2} b_{t}$ (see \eqref{eq-def-star-a-b-p}).
 We compute that
\begin{align*}
 &  \int_{0}^{t} \sum_{u \in N^{1}_{s}}   \Phi_{\sqrt{2}}\big(t-s, X_{u}(s)-a_{t}s-\mathrm{y}+ \frac{h'}{\sqrt{2}}\log t  \big) 1_{\{ (s,X_{u}(s)) \in \Omega_{t,h}^{R} \} }  \dif s   \\
  &\sim   \int_{\Gamma^{R}_{h}- p_{t}\sqrt{t} 1_{\{ h \geq 1/2\} }} \frac{ \gamma_{\sqrt{2}}(\varphi) \, \mathrm{y}}{t^{1-h'} }  \sum_{u \in N^{1}_{s}} e^{ b_{t}X_{u}(s)-  (\beta_{t}+\frac{\sigma^2_{t} b^{2}_{t}}{2}) s } 1_{\{  |X_{u}(s)-\sigma_{t}^2 b_{t}s| \leq R\sqrt{s} \} }      e^{ -  (1-\sigma^2)^{2}\xi^2 }  \dif \xi \\
  &\sim
  \int_{  \xi \in  \Gamma^{R}_{h}- p_{t}\sqrt{t} 1_{\{ h \geq 1/2\} }} \frac{ \gamma_{\sqrt{2}}(\varphi) \,  \mathrm{y}}{t^{1-h'} }  W(p_{t} t +\xi \sqrt{t} ;t )
      e^{ -  (1-\sigma^2)^{2}\xi^2 }  \dif \xi,
  \end{align*}
  where $ W(s;t):=  \sum_{u \in N^{1}_{s}} e^{ b_{t}X_{u}(s)-  (\beta_{t}+\frac{\sigma^2_{t} b^{2}_{t}}{2}) s } 1_{\{  |X_{u}(s)-\sigma_{t}^2 b_{t}s| \leq R\sqrt{s} \} }  $.
By the Brownian scaling, $(X_u(s):u \in N_s^1) \overset{law}{=} ( \frac{\sigma_{t}}{\sqrt{\beta_{t}}}\mathsf{X}_u(s'):u \in \mathsf{N}_{s'})$, where $s'=\beta_{t}s$. Let $\lambda_{t}= b_{t}\sigma_{t}/\sqrt{\beta_{t}}$. We have
\begin{equation*}
 W(s,t)   \overset{law}{=}  \sum_{u \in \mathsf{N}_{s'}}  e^{ \lambda_{t} X_{u}(s') - (1+\lambda_{t}^2/2)s'} 1_{\{| \frac{ \mathsf{X}_{u}(s')- \lambda_{t}s'}{\sqrt{s'}}| \leq \frac{ R}{\sigma_{t} } \}}.
\end{equation*}
Since $\lambda_{t} \to  \sqrt{2}\sigma/\sqrt{\beta} < \sqrt{2}$, by part(i) of Lemma \ref{lem-functional-convergence-derivative-martingale-2}, we have
\begin{equation*}
 \lim_{t \to \infty} W(p_{t}t+\xi \sqrt{t},t)  =\mathsf{W}_{\infty} \left( \frac{ \sqrt{2}\sigma}{\sqrt{\beta}} \right) \int_{[-\frac{R}{\sigma},\frac{R}{\sigma}]} e^{-\frac{x^2} {2}}  \frac{   \dif x  }{\sqrt{2 \pi }} =\mathsf{W}^{\beta,\sigma^2}_{\infty}(\sqrt{2}) \int_{[-\frac{R}{\sigma},\frac{R}{\sigma}]} e^{-\frac{x^2} {2}}  \frac{   \dif x  }{\sqrt{2 \pi }}
\end{equation*}
in law.
Also we have  $\lim\limits_{t \to \infty} \frac{\mathrm{y} }{t^{1-h'}}  =  C(h,\xi) =  \frac{1_{\{h \leq 1/2\}}}{\sqrt{2}(1-\sigma^2)} +\sqrt{2}(1-\sigma^2) \xi 1_{\{h \geq 1/2\}} $. Therefore
 \begin{multline*}
  \lim\limits_{t \to \infty} \mathbb{E}\left(e^{-\left\langle\widehat{\mathcal{E}}_t^R, \varphi\right\rangle}\right)
  \\
  =\mathsf{E}
  \bigg(\exp \bigg\{ -   \gamma_{\sqrt{2}}(\varphi)
  \mathsf{W}^{\beta,\sigma^2}_{\infty}(\sqrt{2})
  \int_{  \xi \in  \Gamma^{R}_{h}- \xi_{0}(h) } C(h,\xi)  e^{ - (1- \sigma^2)^{2}\xi^2 }  \dif \xi
  \int_{ |\sigma x| \leq R } e^{-\frac{x^2}{2 } } \frac{\dif x }{\sqrt{2 \pi  }}    \bigg\} \bigg),
 \end{multline*}
 where   $\xi_{0}(h)=1_{\{ h\geq 1/2\}} \lim_{t \to \infty } p_{t}\sqrt{t} =   \frac{1}{2(1-\sigma^2)^2} 1_{\{ h=1/2\}} $.
Applying Corollary \ref{cor-reduce-to-E-R}   we finally  get
 \begin{align*}
   \lim_{t \to \infty}  \mathbb{E}\left(e^{-\left\langle\widehat{\mathcal{E}}_t, \varphi\right\rangle}\right)= \lim_{R \to \infty}  \lim_{t \to \infty}  \mathbb{E}\left(e^{-\left\langle\widehat{\mathcal{E}}_t^R, \varphi\right\rangle}\right)  =
  \mathsf{E}
   \bigg(\exp \big\{ - C_{h,+}\gamma_{\sqrt{2}}(\varphi) \mathsf{Z}_{\infty}    \big\} \bigg),
 \end{align*}
 which is the Laplace functional of $\operatorname{DPPP}\left( C_{h,+}\sqrt{2}C_{\star}   \mathsf{Z}_{\infty}   e^{-\sqrt{2} x} \dif x, \mathfrak{D}^{\sqrt{2}}\right)$, and where
 \begin{align*}
  C_{h,+} := \lim_{R \to \infty} & \int_{\Gamma^{R}_{h}-\xi_{0}(h)}C(h,\xi)  e^{ - (1- \sigma^2)^{2}\xi^2 }  \dif \xi   =  \frac{\sqrt{\pi /2}}{ (1-\sigma^{2})^{2}} 1_{\{h<1/2\} } + \frac{1_{\{ h>1/2 \}}}{\sqrt{2}(1-\sigma^2)} \\
   & +  \int_{-\frac{1}{2(1-\sigma^{2})^{2}}}^{\infty} (  \frac{1}{\sqrt{2}(1-\sigma^2)} +\sqrt{2}(1-\sigma^2) \xi) e^{ - (1- \sigma^2)^{2}\xi^2 }  \dif \xi 1_{\{h=1/2\}}.
 \end{align*}
 We now complete the proof.
\end{proof}

\subsection{Approaching to $(1,1)$ from $\mathcal{C}_{III}$}

Assume  that $(\beta_{t},\sigma^{2}_{t})_{t>0} \in \mathscr{A}^{h,3}_{(1,1)}$. Let $ m^{(1,1)}_{h,3}(t)=v^{*}_{t} t-\frac{h'}{2\sqrt{2}}\log t$ where  $h'=\min\{h,1\}$. By the assumption   $\beta_{t}+\sigma^{2}_{t}=\frac{1}{\beta_{t}}+\frac{1}{\sigma^2_{t}} $, we have
$\beta_{t}\sigma_{t}^2=1$. So   $v_{t}=\sqrt{2}$ and   $\theta_{t}=\sqrt{2}/\sigma_{t}^2$.   Moreover, we have
\begin{equation}\label{eq-asymptotic-abv-3}
  \sigma_{t}=1- \frac{1}{2t^{h/2}} +  \Theta(\frac{1}{t^{h}}), \ p_{t} \equiv \frac{1}{2} \ , \ b_{t}=\frac{\sqrt{2}}{\sigma_{t}} \ , \ \sqrt{2}- a_{t} \sim   \frac{1}{\sqrt{2}t^{h/2}} \ , \ v^{*}_{t}- \sqrt{2} \sim   \frac{1}{4\sqrt{2}t^{h}} \,.
\end{equation}
In fact, \eqref{eq-asymptotic-abv-3} follows from the following computations.
Firstly, $\beta_{t}+ \sigma^{2}_{t}=\sigma^{2}_{t}+ \sigma^{-2}_{t}=2+t^{-h} \Rightarrow (\sigma_{t}^{2}-1)^{2}= \sigma_{t}^{2} t^{-h} \sim t^{-h} \Rightarrow \sigma_{t}^{2} =1- t^{-h/2}+ \Theta(t^{-h})$. Secondly,
$p_{t}=\frac{\beta_{t}+\sigma^{2}_{t}-2}{2(1-\sigma^{2}_{t})(\beta_{t} -1) }= \frac{t^{-h}}{2[ \beta_{t}+\sigma_{t}^{2}-1- \beta_{t}\sigma_{t}^{2} ]}=\frac{1}{2}$. Thirdly, we  have  $(\beta_{t}-1)^{2} = \beta_{t} t^{-h}$. Thus $b_{t}=\sqrt{2} \sqrt{\frac{\beta_{t} -1}{1-\sigma^{2}_{t}}}=\sqrt{2} (\frac{\beta_{t}}{\sigma_{t}^{2}} )^{1/4}=\frac{\sqrt{2}}{\sigma_{t}}$. Besides $\sqrt{2}-a_{t}=\sqrt{2}(1-\sigma_{t}) \sim \frac{1}{\sqrt{2}t^{h/2}}$. Finally $v^{*}_{t}- \sqrt{2} = \frac{1}{\sqrt{2}}(\sigma_{t}+\frac{1}{\sigma_{t}}-2)=\frac{(\sigma_{t}-1)^{2}}{\sqrt{2} \sigma_{t} } \sim \frac{1}{4\sqrt{2}t^{h}}$.

Define
\begin{equation}\label{eq-def-Omega-t-R-3}
  \Omega_{t,h}^{R}=
    \begin{cases}
    \{ (s,x): | s - \frac{t}{2}| \leq R t^{\frac{1+h}{2}}, \   | x -a_{t} s | \leq  R  \sqrt{s}   \} \quad \text{ for } h \in (0,1) ; \\
    \{ (s,x): s \in [\frac{1}{R}t, (1-\frac{1}{R})t] , \     \sqrt{2} s -x  \in [\frac{1}{R} \sqrt{t}, R\sqrt{t}]    \} \quad \text{ for } h \in [1,\infty]
  \end{cases}
\end{equation}

\begin{lemma}\label{lem-contribution-area-3}
  For all $A>0$,  and $h \in (0,\infty]$
  \begin{equation*}
  \lim _{R \rightarrow \infty} \limsup _{t \rightarrow \infty} \mathbb{P}\left(\exists u \in N^{2}_{t}: X_{u}(t) \geq m^{(1,1)}_{h,3}(t) - A ,  (T_{u} , X_{u}(T_{u}) )  \notin \Omega_{t}^{R,h}  \right)=0 .
  \end{equation*}
  \end{lemma}

  \begin{proof}
Applying  Corollary \ref{cor-reduce-to-E-R} with $m(t)=m^{(1,1)}_{h,3}(t)$,  and $\Omega^{R}_{t,h}$ defined in \eqref{eq-def-Omega-t-R-3}, it suffices to show that  for each $A, K>0$, $I(t,R) = I(t,R;A,K)$  defined in \eqref{eq-assumption-EYtARK}  vanishes as first $t \to \infty$ and then $R \to \infty$.

  \textbf{Case 1: $h \in (0,1)$.}  Conditioned on the Brownian motion $B_{s}$ in \eqref{eq-assumption-EYtARK} equals $ \frac{a_{t}}{\sigma_{t}} s + x $,    we have
  \begin{equation}\label{eq-I-t-R-010}
    \begin{aligned}
   I(t,R) \lesssim  &\int_{0}^{t}\dif s  \int_{-\infty}^{\frac{\sqrt{2}-a_{t}}{\sigma_{t}} s + K}   \frac{(\sqrt{2}-a_{t})s+\sigma_{t}K-\sigma_{t}x}{s^{3/2} }    \\
  & \qquad \qquad \times  \mathsf{F}_{t}\left(t-s, a_{t} s +\sigma_{t} x\right)   e^{(\beta_{t}   -\frac{a_{t}^2}{2\sigma^{2}_{t}} )s -\frac{a_{t}}{\sigma_{t}} x  -\frac{x^2}{2s} } 1_{\left\{(s ,  a_{t}s+\sigma_{t}x ) \notin  \Omega^{R}_{t,h} \right\}} \dif x   .
    \end{aligned}
  \end{equation}
 where we use that  $ \mathbf{P}(B_{u} \leq \frac{\sqrt{2}}{\sigma_{t}}u+K, \forall u \leq s| B_{s}= \frac{a_{t}}{\sigma_{t}}+x) \lesssim \frac{(\sqrt{2}-a_{t})s+\sigma_{t}K-\sigma_{t}x}{s } $  by Lemma \ref{lem-bridge-estimate-0}.

We still denote  $\mathrm{y}:= (\sqrt{2}-a_{t})s+(v^{*}_{t}-\sqrt{2})t $. Let  $\mathrm{w}:=\frac{h}{2\sqrt{2}} \log t -\frac{3}{2\sqrt{2}}\log(t-s+1) +A$. As in  \eqref{eq-bound-F-2}, provided that $\mathrm{y}-\sigma_{t}x-\mathrm{w}>1$, we have
  \begin{equation}\label{eq-bound-F-010}
  \begin{aligned}
     & \mathsf{F}_{t}\left(t-s,  a_{t} s + \sigma_{t} x \right) \\
  &  \lesssim_{A}\frac{  (\mathrm{y}-\sigma_{t}x-\mathrm{w}) t^{h/2}}{(t-s+1)^{3/2}}\exp\left\{ -\sqrt{2}\mathrm{y}  + \sqrt{2}\sigma_{t}x-\frac{1}{2(t-s)}[\mathrm{y}-\sigma_{t}x- \widetilde{\mathrm{w}}  ]^2 \right\} .
  \end{aligned}
    \end{equation}
  where $\widetilde{\mathrm{w}}:= \mathrm{w}- \frac{3}{2\sqrt{2}} \log(t-s+1)$. Indeed for large $t$ and for all $a_{t}s + \sigma_{t} x \leq \sqrt{2} s + \sigma_{t} K$ we have $\mathrm{y} - \sigma_{t}x-\mathrm{w} \geq (v^{*}_{t}-\sqrt{2})t - O(\log t) = \Theta(t^{1-h}) > 1$  by \eqref{eq-asymptotic-abv-3}.

Let  $J=J_{x,s,t}:=   \frac{|(\sqrt{2}-a_{t})s+\sigma_{t}K-\sigma_{t}x|}{s^{3/2} } \frac{  |\mathrm{y}-\sigma_{t}x-\mathrm{w}| t^{h/2}}{(t-s+1)^{3/2}} $. Substituting \eqref{eq-bound-F-010} into \eqref{eq-I-t-R-010} we get
  \begin{align*}
  I(t,R)  \lesssim
 &  \int_{0}^{t}    \exp\bigg\{ (\beta_{t}-\frac{a_t^2}{2\sigma
  ^2_{t}}) s - \sqrt{2} \mathrm{y} - \frac{\mathrm{y}^2 }{2(t-s)} \bigg\}  \dif s
 \\ & \times  \int_{\mathbb{R}}
 J  \exp \bigg\{(\sqrt{2}\sigma_{t}-\frac{a_{t}}{\sigma_{t}}+\frac{\mathrm{y}\sigma_{t}}{t-s})x \bigg\}
  e^{\frac{-\widetilde{\mathrm{w}}\mathrm{y}}{t-s}}
  e^{-\frac{[\sigma_{t}x+\widetilde{\mathrm{w}}]^2}{2(t-s)} -\frac{x^2}{2s}}   1_{\left\{(s ,  a_{t}s+\sigma_{t}x ) \notin  \Omega^{R}_{t,h} \right\}} \dif x.
    \end{align*}
Making a change of variable  $s=p_{t} t+ \xi t^{\frac{1+h}{2}}=\frac{1}{2}t+ \xi t^{\frac{1+h}{2}}$, by Lemma \ref{lem-computation},  we have
   \begin{align*}
     (\beta_{t}-\frac{a_t^2}{2\sigma
       ^2_{t}}) s - \sqrt{2} \mathrm{y} - \frac{\mathrm{y}^2}{2(t-s)}  & =   L(\xi t^{\frac{1+h}{2}}, t) =    - (\sqrt{2}+1) \xi^2 - R(\xi,t) , \\
       (\sqrt{2}\sigma_{t}-\frac{a_{t}}{\sigma_{t}}+\frac{\mathrm{y}\sigma_{t}}{t-s}) & =  \frac{\sigma  _t(b_{t}-a_t)}{t-s} \xi t^{\frac{1+h}{2}} .
  \end{align*}
  Note that $(s ,  a_{t}s+\sigma_{t}x ) \in  \Omega^{R}_{t,h}$ if and only if $ |\xi| \leq R $ and $|\sigma_{t} x|  \leq R\sqrt{s}$. Then we have
      \begin{align*}
 I(t,R)  \lesssim
  & \int_{- \frac{1}{2}t^{\frac{1-h}{2}}}^{\frac{1}{2}t^{\frac{1-h}{2}}}
  e^{  - (\sqrt{2}+1)\xi^2  - R(\xi,t) }  \dif \xi
   \\
  & \times   \int_{\mathbb{R}}  t^{\frac{1+h}{2}} J\exp \bigg\{ \frac{\sigma_t(b_{t}-a_t)}{t-s} \xi t^{\frac{1+h}{2}} x -\frac{\widetilde{\mathrm{w}}\mathrm{y}}{t-s}
  -\frac{[\sigma_{t}x+\widetilde{\mathrm{w}}]^2}{2(t-s)} -\frac{x^2}{2s} \bigg\}   1_{\left\{  \substack{ |\xi|>R, \text{or} \\ |\sigma_{t} x|>R\sqrt{s} }\right\}} \dif x .  \end{align*}
Again  making   a change of variable $x=\eta \sqrt{s}$,   we get
 \begin{align*}
  I(t,R)  \lesssim
   & \int_{- \frac{1}{2}t^{\frac{1-h}{2}}}^{\frac{1}{2}t^{\frac{1-h}{2}}}
   e^{  - (\sqrt{2}+1)\xi^2  - R(\xi,t) }  \dif \xi \\
   & \times
     \int_{\mathbb{R}} \sqrt{s}\,  t^{\frac{1+h}{2}}  J\exp \bigg\{
   \frac{\sigma_t(b_{t}-a_t)}{t-s}  \sqrt{s} \, t^{\frac{1+h}{2}} \xi \eta -\frac{\widetilde{\mathrm{w}}\mathrm{y}}{t-s}
   -\frac{[\sigma_{t} \eta \sqrt{s}+\widetilde{\mathrm{w}}]^2}{2(t-s)} -\frac{\eta^2}{2} \bigg\}   1_{\left\{  \substack{ |\xi|>R, \text{or} \\ |\sigma_{t} \eta|>R }\right\}} \dif \eta.
  \end{align*}
  By \eqref{eq-asymptotic-abv-3},  $\sqrt{2}-a_{t} \sim \frac{1}{\sqrt{2}t^{h/2}}$, $\mathrm{y} \sim (\sqrt{2}-a_{t})\frac{t}{2}$ and $\mathrm{w}=O(\log t)$.  For fixed $\xi, \eta$  we have
  \begin{align*}
    \lim_{t \to \infty}   \sqrt{s}\,  t^{\frac{1+h}{2}}    J & = \lim_{t \to \infty}    t^{\frac{1}{2}+h}     \frac{|(\sqrt{2}-a_{t})s+\sigma_{t}K-\sigma_{t}\eta \sqrt{s}|}{s } \frac{  |\mathrm{y}-\sigma_{t}\eta \sqrt{s}-\mathrm{w}| }{(t-s+1)^{3/2}} \\
    & \leq  \lim_{t \to \infty}    t^{\frac{1}{2}+h}   [ (\sqrt{2}-a_{t}) +  \frac{\eta}{\sqrt{t}} ]   \frac{ (\sqrt{2}-a_{t})\frac{t}{2}  }{(t/2)^{3/2}} = 2 .
  \end{align*}
Similarly, $\lim_{t \to \infty} \frac{\sigma_t(b_{t}-a_t)}{t-s}  \sqrt{s} \, t^{\frac{1+h}{2}} = 2$. Then the dominated convergence theorem   yields that
     \begin{equation*}
     \limsup_{t \to \infty}  I(t,R) \lesssim \iint_{\R^2\backslash [-R,R]^2}  \exp\{-(\sqrt{2}+1)\xi^2+2\xi \eta-\eta^2 \}
     \dif \eta \dif \xi \overset{R \to \infty}{\longrightarrow} 0.
     \end{equation*}

 \textbf{Case 2: $h \in [1,\infty]$.} Now  conditioned on the Brownian motion $B_{s}$ in \eqref{eq-assumption-EYtARK} equals $ \sqrt{2 \beta_{t}}s - x $,  we get
     \begin{equation}
     \begin{aligned}
 I(t,R)  & = \int_{0}^{t} \dif s \int_{-K}^{\infty} \mathbf{P}\left(  B_{r}  \leq  \sqrt{2\beta_{t}}r+  K, \forall r \leq s  \mid    B_{s}=\sqrt{2\beta_{t}}s - x  \right) \\
 & \qquad  e^{\beta_{t}s}  \mathsf{F}_{t}\left(t-s,  \sqrt{2}s-\sigma_{t}x \right) 1_{\left\{(s, \sqrt{2}s-\sigma_{t}x) \notin  \Omega^{R}_{t,h} \right\}} e^{-\frac{-(\sqrt{2\beta_{t}}s-x)^2}{2s}}
 \frac{\dif x}{\sqrt{2\pi s}}  \\
 & \lesssim \int_{0}^{t}   \dif s \int_{-K}^{\infty} \frac{K+x}{s^{3/2}}   \mathsf{F}_{t}\left(t-s,  \sqrt{2}s-\sigma_{t}x \right) 1_{\left\{(s, \sqrt{2}s-\sigma_{t}x) \notin  \Omega^{R}_{t,h} \right\}}
 e^{-\sqrt{2 \beta_{t}} x -\frac{x^{2}}{2s} }
 \dif x .
  \end{aligned}
  \end{equation}
  Let $\mathrm{w}:=  \frac{1}{2\sqrt{2}}\log t - \frac{3}{2\sqrt{2}}\log (t-s+1)  + A$.  By Lemma \ref{lem-Max-BBM-tail}, if  $\sigma_{t}x-\mathrm{w}>1$, we have
  \begin{equation}\label{eq-bound-F-3}
    \begin{aligned}
  & \mathsf{F}_{t}\left(t-s,  \sqrt{2} s - \sigma_{t} x \right) \\
  &  =\mathsf{P}( \max\limits_{u \in \mathsf{N}_{t-s}} \mathsf{X}_{u}(t-s) > \sqrt{2}(t-s)- \frac{3}{2\sqrt{2}} \log(t-s+1) +\sigma_{t}x - \mathrm{w} ) \\
 & \lesssim_{A}  (\sigma_{t}x-\mathrm{w}) \frac{t^{1/2}}{(t-s+1)^{3/2}} \exp\bigg\{ \sqrt{2} \sigma_{t} x -\frac{[\sigma_{t} x - \widetilde{\mathrm{w}}]^2}{2 (t-s)} \bigg\},
  \end{aligned}
  \end{equation}
  where $ \widetilde{\mathrm{w}}:= \mathrm{w}- \frac{3}{2\sqrt{2}} \log(t-s+1)$.
Note that $ \sigma_{t}x - \mathrm{w} \geq    \frac{3}{2\sqrt{2}}\log (t-s+1) - \frac{1}{2\sqrt{2}}\log t- O(1)$. So if  $\sigma_{t}x - \mathrm{w} \leq 1$ it must be $s \geq t/2$ and $-\sqrt{2 \beta_{t}} x \leq - \sqrt{2} \beta_{t} \mathrm{w}=-\sqrt{2} \mathrm{w} + o(1)$. We upper bound $  \mathsf{F}_{t}\left(t-s, \sqrt{2}s-\sigma_{t}x \right)$ by $1$. Futhermore, there holds
\begin{equation*}
  \int_{t/2}^{t}   \dif s \int_{-K}^{O(\log t)} \frac{K+x}{s^{3/2}}
  e^{-\sqrt{2 \beta_{t}} x -\frac{x^{2}}{2s} }1_{\left\{   \sigma_{t}x -\mathrm{w} \leq 1 \right\}}
  \dif x \lesssim   \frac{O(\log^{2} t)}{t^{3/2}}  \int_{t/2}^{t} \frac{t^{1/2}\, \dif s }{(t-s+1)^{3/2}} =o(1).
\end{equation*}
In summary, we have
\begin{align*}
  I(t,R) \lesssim \int_{0}^{t}  \dif s & \int_{-K}^{\infty} \frac{|K+x| \, t^{1/2} \, |\sigma_{t}x-\mathrm{w}|}{s^{3/2}(t-s+1)^{3/2}} 1_{\left\{(s, \sqrt{2}s-\sigma_{t}x) \notin  \Omega^{R}_{t,h} \right\}} \\
  &\qquad  \qquad \qquad \times \exp \left\{\sqrt{2}(\sqrt{\beta_{t}}-\sigma_{t})x \right\} \exp\left\{-\frac{x^2}{2s}-\frac{(\sigma_{t}x- \widetilde{\mathrm{w}} )^2}{2(t-s)} \right\}  \dif x .
\end{align*}
Make change of variables $s=\xi t$ and $x=\eta \sqrt{t}$. Then $(s,\sqrt{s}-\sigma_{t} x  ) \in \Omega_{t,h}^{R}$ if and only if $\xi \in [R^{-1},1-R^{-1}]$ and $\sigma_{t} \eta \in [R^{-1}, R] $. Hence
\begin{align*}
  I(t,R) \lesssim \int_{0}^{1} \dif \xi & \int_{-\frac{K}{\sqrt{t}}}^{\infty}\frac{ ( |\eta|+\frac{K}{\sqrt{t}}) (|\eta| + \frac{\mathrm{w}}{\sqrt{t}}) }{\xi^{3/2} (1-\xi)^{3/2}} \ 1_{\left\{\substack{ \xi \notin [R^{-1},1-R^{-1}] \\
  \text{or } \sigma_{t} \eta \notin [R^{-1},R]} \right\}}\\
&  \qquad\times \exp\left\{ \sqrt{2}(\sqrt{\beta_{t}}-\sigma_{t})\sqrt{t}\eta \right\}
\exp\left\{-\frac{\eta^2}{2\xi}-\frac{(\sigma_{t} \eta- \widetilde{\mathrm{w}}/\sqrt{t})^2}{2(1-\xi)} \right\}  \dif \eta .
\end{align*}
By \eqref{eq-asymptotic-abv-3},  we have $\lim_{t \to \infty} (\sqrt{\beta_{t}}-\sigma_{t}) \sqrt{t} =\lim_{t \to \infty} (\sigma_{t}^{-1}-\sigma_{t}) \sqrt{t}=  1_{\{ h=1/2\}}$.  Applying the dominated convergence theorem, we get
\begin{equation*}
 \limsup_{t \to \infty} I(t,R) \lesssim
 \int_{0}^{1} \dif \xi
 \int_{0}^{\infty}
 \frac{ \eta^2}{\xi^{3/2} (1-\xi)^{3/2}}
 e^{\sqrt{2}\eta} e^{-\frac{\eta^2}{2\xi}-\frac{\eta^2}{2(1-\xi)}} 1_{\left\{\substack{ \xi \notin [R^{-1},1-R^{-1}] \\
  \text{or } \eta \notin [R^{-1},R]} \right\}} \dif \eta   \overset{R \to \infty}{\longrightarrow} 0 ,
\end{equation*}
where we use the integratibility of  $ \frac{   \eta^2}{\xi^{3/2} (1-\xi)^{3/2}}  e^{\sqrt{2}\eta-  \frac{\eta^2}{2\xi}-\frac{\eta^2}{2(1-\xi)}}$  on $(0,1)\times [0,\infty]$. We now compelte the proof.
     \end{proof}

Finally we give the proof of Theorem \ref{thm-11-approximate}
for the case that $(\beta_{t},\sigma^{2}_{t})_{t>0} \in \mathscr{A}^{h,3}_{(1,1)}$.

\begin{proof}[Proof of Theorem \ref{thm-11-approximate}
for $(\beta_{t},\sigma^{2}_{t})_{t>0} \in \mathscr{A}^{h,3}_{(1,1)}$
]
Take $\varphi \in \mathcal{T}$. Applying  Corollary \ref{cor-reduce-to-E-R} with $m(t)=m^{(1,1)}_{h,3}(t)$, $\rho=\sqrt{2}$, and $\Omega^{R}_{t,h}$ defined in \eqref{eq-def-Omega-t-R-3}, it suffices to study the asymptotic of $ \mathbb{E}\left(e^{-\left\langle\widehat{\mathcal{E}}_t^R, \varphi\right\rangle}\right)  $ which equals
\begin{equation*}
 \mathbb{E}\bigg(\exp \bigg\{ -  \int_{0}^{t} \sum_{u \in N^{1}_{s}}   \Phi_{\sqrt{2}}\big(t-s, X_{u}(s)+ \sqrt{2}(t-s)- m^{(1,1)}_{h,3}(t) \big) 1_{\{ (s,X_{u}(s)) \in \Omega_{t,h}^{R} \} }  \dif s \bigg\} \bigg) .
      \end{equation*}
\textbf{Case 1: $h \in (0,1)$}. Let $\mathrm{y}: =(v^{*}_{t}-\sqrt{2})t +(\sqrt{2}-a_{t})s$, then  $\sqrt{2}(t-s)-m^{(1,1)}_{h,3}(t)= -a_{t} s-\mathrm{y}+ \frac{h}{2\sqrt{2}}\log t$. Making a change of variable
 $s=\frac{t}{2}+\xi t^{\frac{1+h}{2}}$,
 we get that $  \mathbb{E}\left(e^{-\left\langle\widehat{\mathcal{E}}_t^R, \varphi\right\rangle}\right) $ is equal to
 \begin{align*}
  \mathbb{E}\bigg(\exp \bigg\{ -  t^{\frac{1+h}{2}}  \int_{-R}^{R} \sum_{u \in N^{1}_{s}} \Phi_{\sqrt{2}}\left(t-s, X_{u}(s)-a_{t}s  - \mathrm{y}+\frac{h}{2\sqrt{2}}\log t \right) 1_{\{  |X_{u}(s)-a_{t}s| \leq R\sqrt{s} \} }   \dif \xi
 \bigg\} \bigg)  .
 \end{align*}
By \eqref{eq-asymptotic-abv-3}, uniformly for  $(s,X_{u}(s)) \in \Omega_{t,h}^{R}$, we have $\mathrm{y}= [1+o(1)] \frac{1}{2\sqrt{2}}t^{1-\frac{h}{2}}$. Then  by Lemma \ref{thm-Laplace-BBM-order}, uniformly for $(s,X_{u}(s)) \in \Omega_{t,h}^{R}$,  we have
    \begin{align*}
 & \Phi_{\sqrt{2}}\left(t-s, X_{u}(s)-a_{t}s- \mathrm{y}+ \frac{h }{2\sqrt{2}}\log t \right) \\
 &\sim  \gamma_{\sqrt{2}}(\varphi) \frac{ t^{h/2} \mathrm{y} }{(t-s)^{3/2}}
 \exp\left\{  \sqrt{2}(X_{u}(s)-  a_t s )-\sqrt{2}\mathrm{y}-\frac{[X_{u}(s)-a_{t}s-\mathrm{y}+ O(\log t)]^2   }{2(t-s)}\right\} \\
 &\sim     \frac{  \gamma_{\sqrt{2}}(\varphi)}{t^{1/2} }
 \exp \left\{  (\sqrt{2} + \frac{\mathrm{y}  }{t-s} )  (X_{u}(s)-  a_t s )  -\sqrt{2}\mathrm{y}-\frac{\mathrm{y}^2}{2(t-s)}-\frac{(X_{u}(s)-a_{t}s)^2}{2(t-s)}    \right\}.
  \end{align*}
We now simplify the term inside exponential.
By  \eqref{eq-asymptotic-abv-3}, $\delta:=\delta_{t,\xi,h} =\frac{(b_t-a_t)}{t-s}t^{\frac{1+h}{2}} = [1+o(1)]  \frac{2\sqrt{2}}{\sqrt{t}}$. Now  $\mathrm{y}=(b_t-\sqrt{2})(t-s)+(b_t-a_t) t^{\frac{1+h}{2}} \xi $  can be written as $\sqrt{2}+\frac{\mathrm{y}}{t-s}= b_{t}+ \delta \xi$.
Part (ii) of Lemma \ref{lem-computation} yields that
$-\sqrt{2}\mathrm{y}-\frac{\mathrm{y}^2}{2(t-s)} = - (\beta_t-\frac{a_t^2}{2 \sigma^2_{t}}) s - ( \sqrt{2}+1)\xi^2+o(1)$.
Besides,
$(\beta_{t}-\frac{a_{t}^2}{2 \sigma_{t}^2})s +a_{t}(b_{t}+\delta \xi)s = [\beta_{t}+ \frac{\sigma^2_{t}}{2}(b_{t}+\delta \xi)^2]s -\frac{\sigma^2_{t} }{2}\delta^2 s \xi^2 =  [\beta_{t}+ \frac{\sigma^2_{t}}{2}(b_{t}+\delta \xi)^2]s-2 \xi^{2}+o(1)$
and
$-\frac{(X_{u}(s)-a_{t}s)^2}{2(t-s)}=-\frac{(X_{u}(s)-a_{t}s)^2}{2s}+o(1) $. Thus,
\begin{align*}
 & \Phi_{\sqrt{2}}\left(t-s, X_{u}(s)-a_{t}s- \mathrm{y}+ \frac{h }{2\sqrt{2}}\log t \right) \\
 & \sim  \frac{  \gamma_{\sqrt{2}}(\varphi)}{t^{1/2} }
 \exp\left\{  (b_{t}+ \delta \xi ) X_{u}(s)- [\beta_t+\frac{\sigma^2_{t}}{2}(b_{t}+\delta \xi)^2 ]s  -\frac{(X_{u}(s)-a_{t} s)^2}{2s} -(\sqrt{2}-1)\xi^2  \right\} .
  \end{align*}
As a consequence,  we have
      \begin{equation}\label{eq-asymptotic-Laplace-02}
     \mathbb{E}\left(e^{-\left\langle\widehat{\mathcal{E}}_t^R, \varphi\right\rangle}\right) = \mathbb{E}\bigg(\exp \bigg\{ - [1+o(1)] \gamma_{\sqrt{2}}(\varphi)  \int_{-R}^{R} t^{h/2} W^{G}_{s}(b_{t}+\delta \xi ) e^{-(\sqrt{2}-1)\xi^2} \dif \xi \bigg\} \bigg) ,
    \end{equation}
  where  $G(x)= G_{R}(x)=e^{-\frac{x^2}{2}}1_{\{|x| \leq R\}}$ and
 \begin{equation*}
   W^{G}_{s}(b_{t}+\delta \xi ):=  \sum_{u \in N^{1}_{s}} e^{(b_{t}+ \delta \xi ) X_{u}(s)- [\beta_t+\frac{\sigma^2_{t}}{2}(b_{t}+\delta \xi)^2 ]s  }  G\left(\frac{a_{t} s-X_{u}(s)}{\sqrt{s}} \right) .
  \end{equation*}

 By the Brownian scaling property,
$\left( X_u(s): u \in N_s^1\right) \overset{law}{=} \left(\frac{\sigma_{t}}{\sqrt{\beta_{t}}} \mathsf{X}_u(s') : u \in \mathsf{N}_{s'} \right)$, where $s'=\beta_{t}s $.
Let $\lambda_{t}:= \frac{\sigma_{t}}{\sqrt{\beta_{t}}}(b_{t}+ \delta\xi) = a_{t} + \sigma_{t}^{2} \delta \xi$
(noticing that
$\frac{\sigma_{t}}{\sqrt{\beta_{t}}} =\sigma_{t}^2 $).  So $W^{G}_{s}(b_{t}+\delta\xi)$ has the same distribution as
\begin{equation*}
 \mathsf{W}^{G}_{s'}(\lambda_{t}):= \sum_{u \in \mathsf{N}_{s'}} e^{ \lambda_{t} \mathsf{X}_{u}(s')-(\frac{\lambda^2_{t} }{2}+1) s'}   G\left( \frac{  \lambda_{t} s'-\mathsf{X}_{u}(s')  }{\frac{1}{\sigma_{t}}  \sqrt{s'}} - \sigma_{t}^{2} \delta \sqrt{s} \xi  \right).
\end{equation*}
By \eqref{eq-asymptotic-abv-3},
$\sqrt{2}-\lambda_{t} \sim \sqrt{2}-a_{t}\sim \frac{1}{\sqrt{2}t^{h/2}} $, and $\lim_{t \to \infty} \sigma_{t}^{2} \delta \sqrt{s} = 2 $.
Applying part (ii) of Lemma \ref{lem-functional-convergence-derivative-martingale-2}, we have
  \begin{equation*}
   \lim_{t \to \infty}  t^{h/2}\mathsf{W}^{G}_{s'}(\lambda_{t})= \sqrt{2} \mathsf{Z}_{\infty} \int_{\mathbb{R}} G(z-2 \xi) e^{-\frac{z^2}{2}} \frac{\dif z }{\sqrt{2\pi}} .
\end{equation*}
Letting $t \to \infty$ in \eqref{eq-asymptotic-Laplace-02}, we get that
 \begin{equation*}
 \lim_{t \to \infty} \mathbb{E}\left(e^{-\left\langle\widehat{\mathcal{E}}_t^R, \varphi\right\rangle}\right) =\mathsf{E}
\bigg(\exp \bigg\{ - \gamma_{\sqrt{2}}(\varphi) \mathsf{Z}_{\infty}
\int_{-R}^{R} e^{-(\sqrt{2}-1)\xi^2} \dif \xi \int_{\mathbb{R}} G_{R}(z-2\xi) e^{-\frac{z^{2}}{2}} \frac{\dif z }{\sqrt{ \pi}} \bigg\} \bigg).
    \end{equation*}
Then letting $R \to \infty$, by Corollary \ref{cor-reduce-to-E-R}, we have
\begin{equation*}
  \begin{aligned}
& \lim _{t \rightarrow \infty} \mathbb{E}\left(e^{-\left\langle\widehat{\mathcal{E}}_t, \varphi\right\rangle}\right)= \lim_{R \to \infty}\lim_{t \to \infty} \mathbb{E}\left(e^{-\left\langle\widehat{\mathcal{E}}_t^R, \varphi\right\rangle}\right) \\
& =  \mathsf{E}
  \left(\exp \left\{ - \sqrt{2}\gamma_{\sqrt{2}}(\varphi) \mathsf{Z}_{\infty}
  \iint_{\mathbb{R}^{2}} e^{-\sqrt{2}\xi^{2}} e^{-(z-\xi)^{2}}  \frac{\dif \xi \dif z }{\sqrt{2 \pi}} \right\}\right)=
  \mathsf{E}
   \left( e^{ - \sqrt{\frac{\pi}{\sqrt{2}}} \gamma_{\sqrt{2}}(\varphi) \mathsf{Z}_{\infty}
   }\right),
  \end{aligned}
  \end{equation*}
  which is the Laplace functional of $\operatorname{DPPP}\left(  \sqrt{\frac{\pi}{\sqrt{2}}} \sqrt{2}C_{\star}   \mathsf{Z}_{\infty}   e^{-\sqrt{2} x} \dif x, \mathfrak{D}^{\sqrt{2}}\right)$.

\textbf{Case 2: $h \in [1,\infty]$.} Now $\sqrt{2}(t-s)-m^{(1,1)}_{h,3}(t)=-\sqrt{2}s + \frac{1}{2\sqrt{2}}\log t $.
Making a change of variable $s= \xi t$, we can rewrite $\mathbb{E}\left(e^{-\left\langle\widehat{\mathcal{E}}_t^R ,\varphi\right\rangle}\right)$ as
\begin{align*}
 \mathbb{E}\bigg(\exp \bigg\{ -   t \int_{\frac{1}{R}}^{1-\frac{1}{R}} \sum_{u \in N^{1}_{s}}  \Phi_{\sqrt{2}}\left(t-s, X_{u}(s)- \sqrt{2} s +\frac{1}{2\sqrt{2}}\log t \right)   1_{\{  \sqrt{2}s- X_{u}(s) \in [\frac{1}{R}\sqrt{s}, R\sqrt{s}] \} }  \dif \xi  \bigg\} \bigg)  .
 \end{align*}
By Lemma \ref{thm-Laplace-BBM-order},   uniformly for $(s,X_{u}(s)) \in \Omega_{t,h}^{R}$, we have
    \begin{align*}
    & \Phi_{\sqrt{2}}\left(t-s, X_{u}(s)- \sqrt{2} s +\frac{1}{2\sqrt{2}}\log t \right)\\
    &\sim   \frac{ \gamma_{\sqrt{2}}(\varphi) t^{1/2} s^{1/2}}{(t-s)^{3/2}} \frac{ \sqrt{2} s- X_{u}(s)}{ \sqrt{s}}
    \exp \left\{ \sqrt{2}(X_{u}(s)-\sqrt{2}s) -\frac{1}{2(t-s)}(X_{u}(s)-\sqrt{2}s)^2 \right\}.
    \end{align*}
Then we have
\begin{equation}\label{eq-asymptotic-Laplace-03}
\mathbb{E}\left(e^{-\left\langle\widehat{\mathcal{E}}_t^R, \varphi\right\rangle}\right)=  \mathbb{E}\bigg(\exp \bigg\{ - [1+o(1)] \gamma_{\sqrt{2}}(\varphi) \int_{\frac{1}{R}}^{1-\frac{1}{R}}  \frac{1}{(1 -\xi)^{3/2} } \sqrt{\xi t} \, W^{G_{\xi}}_{\xi t} \dif \xi  \bigg\} \bigg)  .
 \end{equation}
    where $G_{\xi}(x)=G_{\xi,R}(x): = x e^{-\frac{\xi}{2(1-\xi)} x^2 } 1_{\{x \in [R^{-1},R]\}}  $ and
    \begin{equation*}
       W^{G_{\xi}}_{s} :=  \sum_{u \in N^{1}_{s}}  e^{\sqrt{2}(X_{u}(s)-\sqrt{2}s)}  G_{\xi} \left( \frac{\sqrt{2}s-X_{u}(s)}{\sqrt{s}} \right).
    \end{equation*}
Since  $\left( X_u(s): u \in N_s^1\right) \overset{law}{=} \left(\frac{\sigma_{t}}{\sqrt{\beta_{t}}} \mathsf{X}_u(s') : u \in \mathsf{N}_{s'} \right)$(where $s'=\beta_{t}s$) and $\beta_{t} \sigma_{t}^{2}=1$,   $W^{G_{\xi}}_{s}$ has the same distribution as
    \begin{equation*}
      \mathsf{W}^{G_{\xi}}_{s'} :=  \sum_{u \in \mathsf{N}_{s}}  e^{\sqrt{2} [ \mathsf{X}_{u}(s')- \sqrt{2} s' ] } e^{\sqrt{2}(1-\sigma_{t}^2)[\sqrt{2}s'-\mathsf{X}_{u}(s')]}  G_{\xi}\left(  \sigma_{t} \frac{\sqrt{2}s'-\mathsf{X}_{u}(s')}{\sqrt{s'}} \right).
    \end{equation*}
    Applying Lemma \ref{lem-functional-convergence-derivative-martingale}, we have for $s'=\beta_{t}s=\beta_{t} \xi t$, as $1-\sigma_{t}^2 \sim \frac{1}{t^{h/2}}$,
    \begin{equation*}
     \lim_{t \to \infty} \sqrt{s'}  \mathsf{W}^{G_{\xi}}_{s'} =  \mathsf{Z}_{\infty} \sqrt{\frac{2}{\pi}}  \int_{\frac{1}{R}}^{R}  e^{\sqrt{2 \xi} 1_{\{h=1\}} z} e^{-\frac{\xi}{2(1-\xi)}z^2}   z^2  e^{-\frac{z^2}{2}} \dif z.
    \end{equation*}
    Letting $t \to \infty$ and then $R \to \infty$ in \eqref{eq-asymptotic-Laplace-02},  by Corollary \ref{cor-reduce-to-E-R}  we get that $  \lim\limits_{t \to \infty} \mathbb{E}\left(e^{-\left\langle\widehat{\mathcal{E}}_t \varphi\right\rangle}\right) = $
    \begin{equation*}
   \mathsf{E}
   \bigg(\exp \bigg\{ - \gamma_{\sqrt{2}}(\varphi) \mathsf{Z}_{\infty} \int_{0}^{1} \frac{1}{(1-\xi)^{3/2}} \dif \xi  \int_{0}^{\infty} e^{\sqrt{2 \xi} 1_{\{h=1\}} z} e^{-\frac{\xi}{2(1-\xi)}z^2}   \mu_{\mathrm{Bes}}(\dif z)    \bigg\} \bigg)  ,
    \end{equation*}
 which is the Laplace functional of $\operatorname{DPPP}\left(   C_{h,3}\sqrt{2}C_{\star}   \mathsf{Z}_{\infty}   e^{-\sqrt{2} x} \dif x, \mathfrak{D}^{\sqrt{2}}\right)$,     where
    \begin{align*}
   C_{h,3} & =  \int_{0}^{1} \frac{1}{(1-\xi)^{3/2}} \dif \xi  \sqrt{\frac{ 2}{\pi} } \int_{0}^{\infty} z^{2} e^{\sqrt{2 \xi} 1_{\{h=1\}} z} e^{-\frac{1}{2(1-\xi)}z^2}    \dif z \\
   & = \int_{0}^{1} \dif \xi  \sqrt{\frac{ 2}{\pi} } \int_{0}^{\infty} z^{2} e^{\sqrt{2 \xi(1-\xi)} 1_{\{h=1\}} z} e^{-\frac{1}{2}z^2}    \dif z \\
   & = 1_{\{h>1\}} +  \int_{0}^{1} \dif \xi  \sqrt{\frac{ 2}{\pi} } \int_{0}^{\infty} z^{2} e^{\sqrt{2 \xi(1-\xi)}  z} e^{-\frac{1}{2}z^2}    \dif z 1_{\{h=1\}}.
    \end{align*}
    We now complete the proof.
\end{proof}

  \subsection{Approaching  $\mathscr{B}_{I,III}$ from $\mathscr{C}_{III}$}

    Assume that $(\beta,\sigma^2)\in \mathscr{B}_{I,III}$, and   $(\beta_{t},\sigma^{2}_{t})_{t>0} \in \mathscr{A}^{h,+}_{(\beta,\sigma^2)}$. In this case
    \begin{equation}\label{eq-asymptotic-abv-2}
      1- p_{t} \sim \frac{1}{2 \left(\sigma^{-2}-1\right)^{2} t^{h}}  , \  \theta_{t}-b_{t} \sim  \frac{v}{2(1-\sigma^{2}) t^{h}} , \    v_{t} -a_{t}\sim \frac{\sigma^2  }{2(1-\sigma^2)}  \frac{v}{t^{h}} ,  \ v^{*}_{t}-v_{t} = \Theta(t^{-2h}).
    \end{equation}
    In fact, $1-p_{t}= 1-\frac{\sigma^2_{t}+\beta_{t}-2}{2(\beta_{t}-1)\left(1-\sigma^2_{t}\right)} =  \frac{\beta_{t}^{-1}+\sigma_{t}^{-2}-2}{2(1-\beta_{t}^{-1})\left(\sigma^{-2}_{t}-1\right)} \sim \frac{1}{2 \left(\sigma^{-2}-1\right)^{2} t^{h}}  $;
   $  b_{t}= \theta_{t}   ( \frac{1-\beta_{t}^{-1}}{\sigma_{t}^{-2}-1})^{1/2}= \theta_{t}   ( 1-\frac{ 1}{(\sigma_{t}^{-2}-1)t^{h} })^{1/2}  = \theta_{t} -  \frac{v}{2(1-\sigma^{2}) t^{h}}  +o(\frac{1}{t^{h}})  $;
   $v^{*}_{t}= v_{t} \, \frac{\sigma_{t}^{-2}-\beta_{t}^{-1}}{2\sqrt{(1-\beta_{t}^{-1})( \sigma_{t}^{-2}-1)}} = v_{t}[ 1+ \frac{\theta_{t}}{2b_{t}}(\frac{b_{t}}{\theta_{t}}-1)^{2}] $.

  Recall that $h'  = \min\{ h,1/2 \}$, $ m^{1,3}_{h,+}(t):=   v^{*}_{t}  t - \frac{h' }{ \theta_{t}} \log t$.
   Define
    \begin{equation}\label{eq-def-deltaxst-2}
      \delta(x;s,t) := x- a_{t} s + (b_{t}- a_{t})(p_{t} t-s),
    \end{equation}
    and
    \begin{equation}\label{eq-def-Omega-t-R-4}
      \Omega_{t,h}^{R}=
      \begin{cases}
        \{ (s,x): | s-p_{t} t| \leq    R \sqrt{t}, \   | \delta(x;s,t) | \leq R \sqrt{t-s} \} \text{ for }  h \in (0,\frac{1}{2}) , \\
      \{ (s,x):    t - s \in [\frac{1}{R} \sqrt{t}, R \sqrt{t}],  | \delta(x;s,t) | \leq R \sqrt{t-s} \}  \text{ for }  h \in [\frac{1}{2}, \infty].
    \end{cases}
    \end{equation}

  \begin{lemma}\label{lem-contribution-area-4}
  For all $A>0$,  and $h \in (0,\infty]$
  \begin{equation*}
   \lim _{R \rightarrow \infty} \limsup _{t \rightarrow \infty} \mathbb{P}\left(\exists u \in N^{2}_{t}: X_{u}(t) \geq m^{1,3}_{h,+}(t) - A ,  (T_{u} , X_{u}(T_{u}) )  \notin \Omega_{t}^{R,h}  \right)=0 .
   \end{equation*}
   \end{lemma}

 \begin{proof}
    Applying  Corollary \ref{cor-reduce-to-E-R} with $m(t)=m^{1,3}_{h,+}(t)$ and
    $\Omega^{R}_{t}=\Omega^{R}_{t,h}$
    defined in \eqref{eq-def-Omega-t-R-4}, it suffices to show that  for each $A, K>0$, $I(t,R) = I(t,R;A,K)$  defined in \eqref{eq-assumption-EYtARK}  vanishes as first $t \to \infty$ and then $R \to \infty$. Conditioned on the Brownian motion $B_{s}$ in \eqref{eq-assumption-EYtARK} equals $ \frac{a_{t}}{\sigma_{t}} s + x $,    we have
    \begin{align*}
    I(t,R)
   & = \int_{0}^{t} \dif s \int_{\mathbb{R}} \mathbf{P} \left( \sigma_{t}B_{r} \leq v_{t} r+\sigma_{t} K, \forall r \leq s \mid B_{s} =  \frac{a_{t}}{\sigma_{t}}s + x   \right)     \\
   & \qquad \qquad  \times
   e^{\beta_{t}s}  \mathsf{F}_{t}\left(t-s, a_{t} s +\sigma_{t} x\right) 1_{\left\{(s ,  a_{t}s+\sigma_{t}x ) \notin  \Omega^{R}_{t,h} \right\}}   \exp \left\{-\frac{a_{t}^2}{2\sigma^{2}_{t}}s -\frac{a_{t}}{\sigma_{t}}x  -\frac{x^2}{2s} \right\}   \frac{\dif x}{\sqrt{2\pi s}}    \\
   & \lesssim    \int_{0}^{t} \dif s \int_{-\infty}^{\frac{(v_{t}-a_{t})}{\sigma_{t}}s+ K}   \frac{(v_{t}-a_{t})s+ K-\sigma_{t}x}{\sigma_{t} s^{3/2}}
   \\&   \qquad \qquad  \qquad    \times\mathsf{F}_{t}\left(t-s, a_{t} s +\sigma_{t} x\right) 1_{\left\{(s ,  a_{t}s+\sigma_{t}x ) \notin  \Omega^{R}_{t,h} \right\}} \,
   e^{   (\beta_{t} -\frac{a_{t}^2}{2\sigma^{2}_{t}} )s -\frac{a_{t}}{\sigma_{t}}x  -\frac{x^2}{2s} }  \dif x,
  \end{align*}
 where we used  that
  $\mathbf{P} \left( \sigma_{t}B_{r} \leq v_{t} s+\sigma_{t} K, \forall r \leq s | B_{s} =  \frac{a_{t}}{\sigma_{t}}s + x   \right) \lesssim_{K} \frac{(v_{t}-a_{t})s+ K-\sigma_{t}x}{\sigma_{t} s} $,
  which holds by Lemma \ref{lem-bridge-estimate-0}.

  \begin{itemize}
    \item  If $  v^{*}_{t} t -a_{t}s-\sigma_{t} x   - \mathrm{w} > 1$,  where $\mathrm{w}:= \frac{h'}{\theta_{t}} \log t +A $.
    By the  Markov inequality and Gaussian tail inequality,
    we have
    \begin{equation}\label{eq-bound-F-4}
      \begin{aligned}
   & \mathsf{F}_{t}\left(t-s,  a_{t} s + \sigma_{t} x \right)  =\mathsf{P}\left( \max\limits_{u \in \mathsf{N}_{t-s}} \mathsf{X}_{u}(t-s) >  v^{*}_{t} t -a_{t}s-\sigma_{t} x   - \mathrm{w}\right) \\
   & \lesssim_{A}      \frac{\sqrt{t-s}}{ v^{*}_{t} t -a_{t}s-\sigma_{t} x   - \mathrm{w} } \exp\left\{(t-s)-\frac{1}{2(t-s)}[ v^{*}_{t} t -a_{t}s-\sigma_{t} x   - \mathrm{w}]^2 \right\} .
    \end{aligned}
    \end{equation}
  \item  If $v^{*}_{t} t -a_{t}s-\sigma_{t} x - \mathrm{w}   \leq 1$, we simply upper bound  $\mathsf{F}_{t}\left(t-s,  v_{t} s - \sigma_{t} x \right)$ by $1$. Note that, as $ \sigma_{t} x + a_{t}s \leq v_{t} s + \sigma_{t} K $ and $v_{t}^{*}> v_{t}$,
       we have
       $1 \geq v^{*}_{t} t -a_{t}s-\sigma_{t} x   - \mathrm{w} \geq  v_{t}(t-s) - \mathrm{w}-\sigma_{t}K$. So provided $t$ is large, it  must be $s \geq t- (\log t)^2$ and  $  -\sigma_{t} x \leq  - v_{t}^{*}t + a_{t} s + \mathrm{w} +1$. Thanks to \eqref{eq-condition-zero-2}, we have
  \begin{equation}
    \begin{aligned}
  & \int_{t-(\log t)^{2}}^{t}  \dif s     \int_{-\infty}^{ O(\log t)}   \frac{(v_{t}-a_{t})s+ K-\sigma_{t}x}{\sigma_{t} s^{3/2}} e^{   (\beta_{t} -\frac{a_{t}^2}{2\sigma^{2}_{t}} )s -\frac{a_{t}}{\sigma_{t}}x  -\frac{x^2}{2s} }  1_{ \{v^{*}_{t} t -a_{t}s-\sigma_{t} x - \mathrm{w}    \leq 1 \}  } \dif x \\
   & \lesssim   \int_{t-(\log t)^{2}}^{t}  \dif s   \int_{-\infty}^{ O(\log t)}   \frac{ O(\log t) }{ t^{3/2}} e^{   (\beta_{t} + \frac{\sigma_{t}^{2} b_{t}^2}{2} )s - b_{t} v_{t}^{*}t  + b_{t} \mathrm{w} }   \dif x
   \lesssim  O(\log t)^{4} t^{\frac{b_{t}}{\theta_{t}} h' - 3/2}= o(1).
    \end{aligned}
  \end{equation}
  \end{itemize}

  Let $J=J_{s,x,t}:=  \frac{|(v_{t}-a_{t})s+ K-\sigma_{t}x|}{ s^{3/2}}  \frac{\sqrt{t-s}}{ | v^{*}_{t} t -a_{t}s-\sigma_{t} x   - \mathrm{w} | +1 }  $.
   Then  we have
  \begin{align*}
    I(t,R) \lesssim  \int_{0}^{t} \dif s  &  \int_{-\infty}^{\frac{(v_{t}-a_{t})}{\sigma_{t}}s+ K}    J \, 1_{\left\{(s ,  a_{t}s+\sigma_{t}x ) \notin  \Omega^{R}_{t,h} \right\}} \\
   & \times \exp \left\{ (\beta_{t}-\frac{a_{t}^2}{2\sigma^{2}_{t}})s + (t-s)-\frac{[ v^{*}_{t} t -a_{t}s-\sigma_{t} x   - \mathrm{w}]^2}{2(t-s)}  -\frac{a_{t}}{\sigma_{t}}x \right\} e^{-\frac{x^2}{2s} }  \dif x + o(1).
  \end{align*}
  Making a   change of variable $s= p_{t} t+u$, we have
  \begin{equation*}
   \frac{(v^{*}_{t} t -a_{t}s-\sigma_{t} x   - \mathrm{w})^2}{2(t-s)}=\frac{b_{t}^2}{2}(t-s)+b_{t}(b_{t}-a_{t})u-\sigma_{t}b_{t}x-b_{t} \mathrm{w}+ \frac{[(b_{t}-a_{t})u-\sigma_{t}x-\mathrm{w}]^2}{2(t-s)},
  \end{equation*}
where we used
  $ v^{*}_{t} t -a_{t} s= b_{t}(1-p_{t})t-a_{t}u=b_{t}(t-s)+(b_{t}-a_{t})u  $. Applying   \eqref{eq-condition-zero}
  and the identity $a_{t}=\sigma_{t}^2 b_{t}$, we have
  \begin{align*}
  & (\beta_{t}-\frac{a_{t}^2}{2\sigma^{2}_{t}})s + (t-s)-\frac{[ v^{*}_{t} t -a_{t}s-\sigma_{t} x   - \mathrm{w}]^2}{2(t-s)}  -\frac{a_{t}}{\sigma_{t}}x  \\
  &= (\beta_{t}-\frac{a_{t}^2}{2 \sigma_{t}^2})s +  (1-\frac{b_{t}^2}{2})(t-s) - b_{t}^{2}(1-\sigma_{t}^{2})u  + b_{t} \mathrm{w} - \frac{[(b_{t}-a_{t})u-\sigma_{t}x-\mathrm{w}]^2}{2(t-s)}  \\
  &=  (\beta_{t} -\frac{a_{t}^2}{2 \sigma_{t}^2})p_{t} t +  (1-\frac{b_{t}^2}{2} )(1-p_{t}) t  + \left(\beta_{t}  -1+\frac{\sigma_{t}^2 -1}{2} b_{t}^2 \right)u + b_{t} \mathrm{w} - \frac{[(b_{t}-a_{t})u-\sigma_{t}x-\mathrm{w}]^2}{2(t-s)}  \\
  &=   b_{t} \mathrm{w} - \frac{[(b_{t}-a_{t})u-\sigma_{t}x-\mathrm{w}]^2}{2(t-s)}.
  \end{align*}
   Therefore,
  \begin{equation*}
    I(t,R) \lesssim   \int_{-p_{t}t}^{(1-p_{t})t} \dif u  \int_{\mathbb{R}}
     e^{b_{t} \mathrm{w} }  J  e^{ -  \frac{[(b_{t}-a_{t})u-\sigma_{t}x-\mathrm{w}]^2}{2(t-s)}  }
      e^{-\frac{x^2}{2s} } 1_{\left\{(s ,  a_{t}s+\sigma_{t}x ) \notin  \Omega^{R}_{t,h} \right\}} \dif x + o(1) .
  \end{equation*}
By making a change of variable  $x= \frac{(b_{t}-a_{t})}{\sigma_{t}}u+ z$, we get
    $(s,a_{t}s+\sigma_{t}x) \in \Omega^{R}_{t,h}$ if and only if $ ( [1-p_{t}]\sqrt{t} 1_{\{ h \geq 1/2 \}} - \frac{u}{\sqrt{t}}, \frac{\sigma_{t}z }{\sqrt{t-s}}) \in \Gamma^{R}_{h} \times [-R,R]$,  where  $\Gamma^{R}_{h}=[-R,R]$ if $h\in (0,1/2)$ and $\Gamma^{R}_{h}=[R^{-1},R]$ if $h \in [1/2,\infty]$.
   Hence,
  \begin{equation*}
    I(t,R) \lesssim \int_{-p_{t}t}^{(1-p_{t})t} \dif u
     \int_{\mathbb{R}} Je^{b_{t}    \mathrm{w}} e^{ -  \frac{(\sigma_{t}z+ \mathrm{w})^2}{2(t-s)}-\frac{1}{2s}\big[  \frac{(b_{t}-a_{t})}{\sigma_{t}}u+ z\big]^2  }   1_{\left\{
      \substack{ (1-p_{t})\sqrt{t} 1_{\{ h \geq 1/2 \}} -\frac{u}{\sqrt{t}}\notin  \Gamma^{R}_{h} \\
     \text{ or }  | \frac{\sigma_{t} z}{\sqrt{t-s}} |>R}\right\}}  \dif z   +o(1).
  \end{equation*}
   Making  change of variables   $u= - \xi \sqrt{t}$ and  $z= \eta \sqrt{t-s}$ again,  we have
  \begin{align*}
  I(t,R) \lesssim & \int_{-(1-p_{t})\sqrt{t}}^{p_{t}\sqrt{t}} \dif \xi
  \int_{\mathbb{R}} J \sqrt{t(t-s)} \, e^{b_{t} \mathrm{w}}   1_{\left\{ \substack{ (1-p_{t})\sqrt{t} 1_{\{ h \geq 1/2 \}} + \xi \notin  \Gamma^{R}_{h} \\
  \text{ or } |\sigma_{t} \eta |>R}\right\}}
  \\ & \times   \exp\left\{ -\frac{t}{2s}\big[  \frac{(b_{t}-a_{t})}{\sigma_{t}} \xi  - \eta\sqrt{\frac{t-s}{t}}\big]^2   -  \frac{1}{2 }(\sigma_{t}\eta + \frac{\mathrm{w}}{\sqrt{t-s}} )^2 \right\}
  \dif \eta    +o(1).
  \end{align*}
  As   $v^{*}_{t}t-a_{t}s-\sigma_{t}x=b_{t}(t-s)-\sigma_{t}\eta \sqrt{t-s}$ and  $s= p_{t}t - \xi \sqrt{t}$, for each fixed $\xi$ and $\eta$, by \eqref{eq-asymptotic-abv-2}, we have
  \begin{align*}
  &\lim_{t \to \infty}
  J\sqrt{t(t-s)}  e^{b_{t} \mathrm{w}}
   = \lim_{t \to \infty}  \frac{t-s}{s^{3/2}} \
   \frac{|(v_{t}-a_{t})s+(b_{t}-a_{t})\xi \sqrt{t}-\sigma_{t}\eta \sqrt{t-s} |+K} { | b_t(t-s)-\sigma_{t}\eta \sqrt{t-s}|+1}  \  t^{\frac{b_{t}}{\theta} h ' + 1/2}  \\
  & \leq  \lim_{t \to \infty}  \frac{1}{b_{t}} \
   \frac{ t^{h'} (v_{t}-a_{t}) t + t^{h'} (b_{t}-a_{t}) |\xi |\sqrt{t} } { t }   =  \frac{\sigma^4}{2(1-\sigma^2)}  1_{\{h \leq \frac{1}{2}\}}+  (1-\frac{v}{\theta}) |\xi| 1_{\{h \geq 1/2\}}=:C(\xi;h) .
  \end{align*}
Define $\xi_0(h):= 1_{ \{ h \geq 1/2\}}\lim_{t \to \infty} (1-p_{t})\sqrt{t} = \frac{1_{ \{ h = 1/2\}}}{2(\sigma^{-2}-1)^{2}}$.  Applying the dominated convergence theorem, we get
  \begin{equation*}
   \limsup_{t \to \infty} I(t,R) \lesssim
  \int_{-\lim_{t} (1-p_{t}) \sqrt{t}}^{\infty}  \dif \xi
    \int   C(\xi;h) e^{ -\frac{(\theta-v)^2}{2\sigma^2}\xi^2   -  \frac{\sigma^2 \eta^2 }{2 } } 1_{\left\{ \substack{  \xi \notin  \Gamma^{R}_{h}- \xi_0(h) , \text{or}\\
   |\sigma  \eta |>R}\right\}}    \dif \eta\dif \xi.
    \end{equation*}
    Letting $R \to \infty$,  we get $\lim\limits_{R \to \infty} \lim\limits_{t \to \infty} I(t,R) = 0$ as desired.
  \end{proof}

  \begin{proof}[Proof of Theorem \ref{thm-13-approximate}
 for $(\beta_{t},\sigma^{2}_{t})_{t>0} \in \mathscr{A}^{h,+}_{(\beta,\sigma^2)}$
   ]
    Take $\varphi \in \mathcal{T}$. Applying
     Corollary \ref{cor-reduce-to-E-R} with $m(t)=m^{1,3}_{h,+}(t)$, $\rho=b_{t}$ and
     $\Omega^{R}_{t}=\Omega^{R}_{t,h}$
     defined in \eqref{eq-def-Omega-t-R-4}, it suffices to study the asymptotic of   $\mathbb{E}\left(e^{-\left\langle\widehat{\mathcal{E}}_t^R, \varphi\right\rangle}\right)
     $, which is equal to
     \begin{equation*}
      \mathbb{E}\bigg(\exp \bigg\{ - \int_{0}^{t} \sum_{u \in N^{1}_{s}}  \Phi_{b_{t}}\big(t-s, \delta(X_{u}(s);s,t)+\frac{h'}{\theta_{t}}\log t \big) 1_{\{ (s,X_{u}(s) )\in \Omega^{R}_{t,h}  \} }  \dif s \bigg\} \bigg).
     \end{equation*}
   where we used the fact that $X_{u}(s)+b_{t}(t-s)-m^{1,3}_{h,+}(t) =\delta(X_{u}(s);s,t)+\frac{h'}{\theta_{t}}\log t$, which holds  by \eqref{eq-def-deltaxst-2}.
 By \eqref{eq-asymptotic-abv-2}, $b_{t} \to \theta > \sqrt{2}$. Applying Lemma \ref{thm-Laplace-BBM-order}, we have, uniformly for $(s,X_{u}(s))\in \Omega^{R}_{t,h}$,
    \begin{align*}
     & \Phi_{b_{t}}\big(t-s, \delta(X_{u}(s);s,t)+\frac{h'}{\theta_{t}}\log t \big) \\
    &= (1+o(1)) \frac{\gamma_{\theta}(\varphi)}{\sqrt{t-s}}  e^{-(\frac{b_{t}^2}{2}-1)(t-s)} e^{b_{t} \delta(X_{u}(s);s,t)+ b_{t}\frac{h'}{\theta_{t}}\log t  }e^{-\frac{1}{2(t-s)} \delta(X_{u}(s);s,t)^2} \\
    & = (1+o(1)) \gamma_{\theta}(\varphi) \frac{t^{ h'}}{\sqrt{t-s}}  e^{b_{t}X_{u}(s)- (\frac{\sigma_{t}^2 b_{t}^2}{2}+\beta_{t})s} e^{-\frac{1}{2(t-s)} \delta(X_{u}(s);s,t)^2},
    \end{align*}
where we used the computation that
  $  b_{t} \delta(X_{u}(s);s,t) -(\frac{b_{t}^2}{2}-1)(t-s)   = b_{t} X_{u}(s) + [b_{t}(b_{t}-a_{t})p_{t} -\frac{b_{t}^2}{2}+1]t -(\frac{b_{t}^2}{2}+1)s  = b_{t} X_{u}(s) - (\frac{b_{t}^2}{2}+1)   s  = b_{t} X_{u}(s) - \left(\frac{\sigma_{t}^2 b_{t}^2}{2}+\beta_{t} \right) s $.
 Thus $  \mathbb{E}\left(e^{-\left\langle\widehat{\mathcal{E}}_t^R, \varphi\right\rangle}\right)$ equals
    \begin{equation*}
  \mathbb{E}\bigg(  \exp \bigg\{ - \int_{0}^{t} \sum_{u \in N^{1}_{s}} \frac{[\gamma_{\theta}(\varphi) +o(1)] t^{  h'}}{\sqrt{t-s}}  e^{b_{t}X_{u}(s)- (\frac{\sigma_{t}^2 b_{t}^2}{2}+\beta_{t})s} e^{-\frac{\delta(X_{u}(s);s,t)^2}{2(t-s)} }   1_{\{ (s,X_{u}(s) )\in \Omega^{R}_{t,h}  \} }   \dif s \bigg\}  \bigg).
    \end{equation*}

   \textbf{Case 1: $h \in (0,\frac{1}{2})$.}
   Making a change of variable $s=p_t t- \xi \sqrt{t}$ and letting  $r=r_{\xi,t}=\xi \sqrt{t}$, we  have
   \begin{equation}\label{eq-asymptotic-Laplace-functional-01}
     \mathbb{E}\left(e^{-\left\langle\widehat{\mathcal{E}}_t^R, \varphi\right\rangle}\right)= \mathbb{E}\bigg(\exp \big\{ - (1+o(1)) \gamma_{\theta} (\varphi)  \int_{-R}^{R}  t^{h}	\frac{\sqrt{t}}{\sqrt{t-s}} W^{G}(p_{t}t-\xi \sqrt{t},\xi \sqrt{t} ; t)  \dif \xi \big\} \bigg)\,,
   \end{equation} where  $G(x)=G_{R}(x) :=e^{-\frac{x^2}{2}}1_{\{|x| \leq R\}}$ and
    \begin{equation*}
      W^{G}(s,r; t):=\sum_{u \in N^{1}_{s}} e^{ b_{t} X_{u}(s)-(\frac{\sigma^2_{t}b^2_{t}}{2}+\beta_{t}) s}  G\left( \frac{X_{u}(s)-a_{t} s + (b_{t}-a_{t}) r}{\sqrt{t-s}} \right) \,.
    \end{equation*}
   By the Brownian scaling property, $(X_u(s):u \in N_s^1) \overset{law}{=} ( \frac{\sigma_{t}}{\sqrt{\beta_{t}}}\mathsf{X}_u(s'):u \in \mathsf{N}_{s'})$, where $s'=\beta_{t}s$.
   Then, letting $\lambda_{t}=b_{t} \frac{\sigma_{t}}{\sqrt{\beta_{t}}}$,
    $W^{G}(s,r;t)$ has the same distribution as
   \begin{equation*}
    \mathsf{W}^{G}(s',r; t):= \sum_{u \in \mathsf{N}_{s'}} e^{ \lambda_{t} \mathsf{X}_{u}(s')-(\frac{\lambda^2_{t} }{2}+1) s'}  G\left( \frac{ \frac{\lambda_{t} s'-\mathsf{X}_{u}(s')}{\sqrt{s'}} -\frac{\sqrt{\beta_{t}}}{\sigma_{t}} (b_{t}-a_{t}) \frac{r}{\sqrt{s'}}}{\frac{1}{\sigma_{t}}  \sqrt{\frac{\beta_{t}t-s'}{s'}}} \right).
   \end{equation*}
  Note that $\sqrt{2}-\lambda_{t} \sim \frac{1}{\sqrt{2}(\sigma^{-2}-1)t^{h}}$,   $\frac{\sqrt{\beta_{t}}}{\sigma_{t}} (b_{t}-a_{t}) \frac{r}{\sqrt{s'}}\sim \frac{\theta-v}{\sigma} \xi $ and $\frac{1}{\sigma_{t}}  \sqrt{\frac{\beta_{t}t-s'}{s'}} \sim \frac{1}{\sigma} \sqrt{\frac{t-s}{s}} $.
  Applying part (ii) of Lemma \ref{lem-functional-convergence-derivative-martingale-2},
    we have
   \begin{align*}
    & \mathsf{W}^{G}(p_{t}t-\xi \sqrt{t},\xi \sqrt{t} ; t) \\
    & = [1+o_{\mathsf{P}}(1)]2 \mathsf{Z}_{\infty} (\sqrt{2}-\lambda_{t})\int_{\mathbb{R}}  G\left( \frac{ z -\frac{\sqrt{\beta_{t}}}{\sigma_{t}} (b_{t}-a_{t}) \frac{r}{\sqrt{s'}}}{\frac{1}{\sigma_{t}}  \sqrt{\frac{\beta_{t}t-s'}{s'}}} \right) e^{-\frac{z^2}{2}} \frac{\dif z}{\sqrt{2 \pi}} \\
    &=  [1+o_{\mathsf{P}}(1)] 2 \mathsf{Z}_{\infty} (\sqrt{2}-\lambda_{t}) \left(\int_{\mathbb{R}}  G\left( y \right) \frac{\dif y}{\sqrt{2 \pi}} \right)\,
    e^{-   \frac{(\theta-v)^2}{2\sigma^2} \xi^2 } \frac{1}{\sigma} \sqrt{\frac{t-s}{s}} .
   \end{align*}
  As a consequence, using $\mathsf{Z}_{\infty} \overset{d}{=} \frac{\sqrt{\beta}}{\sigma}  \mathsf{Z}^{\beta,\sigma^2}_{\infty}$,  we have
   \begin{equation*}
     \lim_{t \to \infty}  t^{h}	\frac{\sqrt{t}}{\sqrt{t-s}} W^{G}(p_{t}t-\xi \sqrt{t},\xi \sqrt{t} ; t) = \mathsf{Z}^{\beta,\sigma^{2}}_{\infty}  \frac{\sqrt{2 \beta}}{ 1-\sigma^{2} }  \left(\int_{\mathbb{R}}  G\left( y \right) \frac{\dif y}{\sqrt{2 \pi}} \right)\,
     e^{-   \frac{(\theta-v)^2}{2\sigma^2} \xi^2 }\quad\mbox{ in law}.
   \end{equation*}
   Letting $t \to \infty$ in \eqref{eq-asymptotic-Laplace-functional-01}, by the dominated convergence theorem we get
   \begin{equation*}
     \lim_{t \to \infty}\mathbb{E}\left(e^{-\left\langle\widehat{\mathcal{E}}_t^R, \varphi\right\rangle}\right)=
    \mathsf{E}
     \bigg(\exp \bigg\{ -   \gamma_{\theta} (\varphi) \mathsf{Z}^{\beta,\sigma^2}_{\infty}  \frac{\sqrt{2\beta}}{1-\sigma^{2}}  \int_{-R}^{R}
     e^{-   \frac{(\theta-v)^2}{2\sigma^2} \xi^2 }     \dif \xi  \int_{-R}^{R} e^{-\frac{y^2}{2}} \frac{\dif y}{\sqrt{2 \pi}}  \,  \bigg\} \bigg) .
   \end{equation*}
   Applying Corollary \ref{cor-reduce-to-E-R},  we finally  get
   \begin{align*}
     \lim_{t \to \infty}  \mathbb{E}\left(e^{-\left\langle\widehat{\mathcal{E}}_t, \varphi\right\rangle}\right)= \lim_{R \to \infty}  \lim_{t \to \infty}  \mathbb{E}\left(e^{-\left\langle\widehat{\mathcal{E}}_t^R, \varphi\right\rangle}\right)  =
   \mathsf{E}
     \bigg(\exp \big\{ - C_{h,+}\gamma_{\theta}(\varphi) \mathsf{Z}_{\infty}^{\beta,\sigma^{2}}    \big\} \bigg),
   \end{align*}
   which is the Laplace functional of $\operatorname{DPPP}\left( C_{h,+} \frac{C(\theta)}{\theta \sqrt{2 \pi}}   \mathsf{Z}_{\infty}^{\beta,\sigma^2}  \theta e^{-\theta x} \dif x, \mathfrak{D}^{\theta}\right)$,
  and where
   \begin{align*}
    C_{h,+} &:= \lim_{R \to \infty} \frac{\sqrt{2\beta}}{1-\sigma^{2}}  \int_{-R}^{R}
    e^{-   \frac{(\theta-v)^2}{2\sigma^2} \xi^2 }     \dif \xi  = \frac{ \sqrt{2 \pi} }{1-\sigma^{2}} \frac{v}{\theta- v}.
   \end{align*}

   \textbf{Case 2: $h \in [\frac{1}{2},\infty)$.} Make a change of variable $s= t- \xi \sqrt{t}$ and let  $r=r_{\xi,t}=\xi \sqrt{t}$. As before we have
   \begin{equation}\label{eq-asymptotic-Laplace-functional-012}
     \mathbb{E}\left(e^{-\left\langle\widehat{\mathcal{E}}_t^R, \varphi\right\rangle}\right)= \mathbb{E}\bigg(\exp \big\{ - (1+o(1)) \gamma_{\theta} (\varphi)  \int_{R^{-1}}^{R}   	\frac{t}{\sqrt{t-s}} W^{G}(t-\xi \sqrt{t},\xi \sqrt{t} ; t)  \dif \xi \big\} \bigg) .
   \end{equation}
  Letting
  $\lambda_{t}=b_{t}\frac{\sigma_{t}}{\sqrt{\beta_{t}}}=\sqrt{2}-\frac{1}{\alpha_{t}}$, we have $\alpha_{t} \sim \sqrt{2}(\sigma^{-2}-1)t^{h}$ by \eqref{eq-asymptotic-abv-2}.    $W^{G}(s,r;t)$ has the same distribution as
   \begin{align*}
    & \mathsf{W}^{G}(s',r; t):= \sum_{u \in \mathsf{N}_{s'}} e^{ \lambda_{t} \mathsf{X}_{u}(s')-(\frac{\lambda^2_{t} }{2}+1) s'}  G\left( \frac{  \lambda_{t} s'-\mathsf{X}_{u}(s')  -\frac{\sqrt{\beta_{t}}}{\sigma_{t}} (b_{t}-a_{t})  r }{\frac{1}{\sigma_{t}}  \sqrt{ \beta_{t}t-s'}} \right) \\
    &= e^{-\frac{(s')^2}{2 \alpha_{t}^2}}\sum_{u \in \mathsf{N}_{s'}} e^{ \sqrt{2} \mathsf{X}_{u}(s')-2 s'}  e^{\frac{\sqrt{2}s'-\mathsf{X}_{u}(s')}{\alpha_{t}}} G\left( \frac{  \sqrt{2} s'-\mathsf{X}_{u}(s') -\frac{s'}{\alpha_{t}}- \frac{\sqrt{\beta_{t}}}{\sigma_{t}} (b_{t}-a_{t}) r}{\frac{1}{\sigma_{t}}  \sqrt{\beta_{t} t-s'}} \right) \\
     & = [1+o(1)]e^{\frac{(s')^2}{2 \alpha_{t}^2}+ \frac{\sqrt{\beta_{t}}}{\sigma_{t}} (b_{t}-a_{t})  \frac{r}{\alpha_{t}}} \sum_{u \in \mathsf{N}_{s'}} e^{ \sqrt{2} \mathsf{X}_{u}(s')-2 s'}    G\left( \frac{  \sqrt{2} s'-\mathsf{X}_{u}(s') -\frac{s'}{\alpha_{t}}- \frac{\sqrt{\beta_{t}}}{\sigma_{t}} (b_{t}-a_{t}) r}{\frac{1}{\sigma_{t}}  \sqrt{\beta_{t} t-s'}} \right).
    \end{align*}
   Note that  $r_{s'}=  \frac{\sqrt{s'}}{\alpha_{t}}+ \frac{\sqrt{\beta_{t}}}{\sigma_{t}} (b_{t}-a_{t}) \frac{r}{\sqrt{s'}} \to \frac{\sqrt{\beta}}{\sqrt{2}(\sigma^{-2}-1)}1_{\{h=1/2\}}+ \frac{\theta-v}{\sigma} \xi $ and $h_{s'}= \frac{1}{\sigma_{t}} \sqrt{\frac{\beta_{t}t-s'}{s'}} \sim \frac{1}{\sigma} \sqrt{\frac{t-s}{t}}$.
    Applying Lemma \ref{lem-functional-convergence-derivative-martingale}
    and using the fact
    $\mathsf{Z}^{\infty} \overset{d}{=} \frac{\sqrt{\beta}}{\sigma}  \mathsf{Z}^{\beta,\sigma^2}_{\infty}$, we have
   \begin{align*}
    & \mathsf{W}^{G}(s',r;t)=\frac{[1+o_{\mathsf{P}}(1)] }{\sqrt{s'}}e^{\frac{(s')^2}{2 \alpha_{t}^2}+ \frac{\sqrt{\beta_{t}}}{\sigma_{t}} (b_{t}-a_{t})  \frac{r}{\alpha_{t}}}
     \int_{0}^{\infty}G\left( \frac{ z -r_{s'}}{h_{s'}} \right) z e^{-\frac{z^2}{2}} \dif z \,  \sqrt{\frac{2}{\pi} } \mathsf{Z}_{\infty} \\
     &= \frac{[1+o_{\mathsf{P}}(1)]}{\sqrt{\beta t}} e^{\frac{(s')^2}{2 \alpha_{t}^2}+ \frac{\sqrt{\beta_{t}}}{\sigma_{t}} (b_{t}-a_{t})  \frac{r}{\alpha_{t}}}
     \int_{\mathbb{R}}G\left(y \right)   \dif y \, (r_{s'}  e^{-\frac{r_{s'}^2}{2}})  h_{s'} \, \sqrt{\frac{2}{\pi} }\mathsf{Z}_{\infty} \\
     & = \frac{[1+o_{\mathsf{P}}(1)]} {\sqrt{\beta t}} \, h_{s'} r_{s'}   e^{- \frac{ \beta_{t}}{2 \sigma_{t}^{2}} (b_{t}-a_{t})^{2} \frac{r^{2}}{s'}  } \,
     \int_{\mathbb{R}}G\left(y \right)   \dif y \,  \sqrt{\frac{2}{\pi} }\mathsf{Z}_{\infty} \\
     &=[1+o_{\mathsf{P}}(1)]
      \frac{\sqrt{t-s}}{t}
      \left[ \frac{\sqrt{2\beta}}{ (1-\sigma^2)}1_{\{h=1/2\}}+ \frac{2(\theta-v)}{\sigma^3} \xi \right]
      e^{-\frac{(\theta-v)^2}{2 \sigma^2}\xi^2}
     \int_{\mathbb{R}}G\left(y \right) \frac{\dif y}{\sqrt{2 \pi}}   \mathsf{Z}_{\infty}^{\beta,\sigma^{2}}.
   \end{align*}
 Letting $t \to \infty$ in \eqref{eq-asymptotic-Laplace-functional-01}, we get that $ \lim\limits_{t \to \infty} \mathbb{E}\left(e^{-\left\langle\widehat{\mathcal{E}}_t^R, \varphi\right\rangle}\right)$ equals
    \begin{align*}
   \mathsf{E}
    \bigg(\exp \bigg\{ - \gamma_{\theta} (\varphi)    \mathsf{Z}^{\beta,\sigma^2}_{\infty} \int_{\frac{1}{R}}^{R} \, \left[ \frac{\sqrt{2\beta}\,1_{\{h=1/2\}}}{ (1-\sigma^2)}+ \frac{2(\theta-v)}{\sigma^3} \xi \right]   e^{-\frac{(\theta-v)^2}{2 \sigma^2}\xi^2}   \dif \xi  \int_{-R}^{R} e^{-\frac{y^2}{2}} \frac{ \dif y}{\sqrt{2 \pi}}  \bigg\} \bigg).
     \end{align*}
   Applying Corollary \ref{cor-reduce-to-E-R},   we finally  get
   \begin{align*}
     \lim_{t \to \infty}  \mathbb{E}\left(e^{-\left\langle\widehat{\mathcal{E}}_t, \varphi\right\rangle}\right)= \lim_{R \to \infty}  \lim_{t \to \infty}  \mathbb{E}\left(e^{-\left\langle\widehat{\mathcal{E}}_t^R, \varphi\right\rangle}\right)  =
     \mathsf{E}
      \bigg(\exp \big\{ - C_{h,+}\gamma_{\theta}(\varphi) \mathsf{Z}_{\infty}^{\beta,\sigma^{2}}    \big\} \bigg) ,
   \end{align*}
   which is the Laplace functional of $\operatorname{DPPP}\left( C_{h,+} \frac{C(\theta)}{\theta \sqrt{2 \pi}}   \mathsf{Z}_{\infty}^{\beta,\sigma^2}  \theta e^{-\theta x} \dif x, \mathfrak{D}^{\theta}\right)$,
  and where
   \begin{align*}
    C_{h,+} := &\lim_{R \to \infty} \int_{\frac{1}{R}}^{R} \, \left[ \frac{\sqrt{2\beta}}{ (1-\sigma^2)}1_{\{h=1/2\}} + \frac{2(\theta-v)}{\sigma^3} \xi \right]   e^{-\frac{(\theta-v)^2}{2 \sigma^2}\xi^2}   \dif \xi  \\
    = &\sqrt{\frac{\pi}{2} } \frac{\sigma^{2}}{(1-\sigma^{2})^{2}}1_{\{h=1/2\}}+ \frac{2}{\sigma(\theta- v)}.
   \end{align*}
   We now complete the proof.
    \end{proof}

 \section{Approaching $\mathscr{B}_{II,III}$ and $(1,1)$ from $\mathscr{C}_{II}$}

\subsection{Approaching $\mathscr{B}_{II,III}$ from $\mathscr{C}_{II}$}

Assume that $(\beta,\sigma^2) \in \mathscr{B}_{II,III}$ and $(\beta_{t},\sigma^2_{t}) \in
\mathscr{A}^{h,-}_{(\beta,\sigma^2)}$.
Recall that $
  m^{2,3}_{h,-}(t)= \sqrt{2} t - \frac{3-4 h' }{2 \sqrt{2}} \log t $,  $h' = \min\{ h,1/2 \}$.
 Define
\begin{equation}\label{eq-def-Omega-t-R-1}
  \Omega_{t,h}^{R}:= \left\{ (s,x):  s \in [ \frac{1}{R} t^{h'}  ,R t^{h'}] , \   | x - \sqrt{2} \sigma^2_{t} s | \leq R \sqrt{s}     \right\}.
\end{equation}

\begin{lemma}\label{lem-contribution-area-1}
For any $A>0$,
  \begin{equation*}
  \lim _{R \rightarrow \infty} \limsup _{t \rightarrow \infty} \mathbb{P}\left(\exists u \in N^{2}_{t}: X_{u}(t) \geq m^{2,3}_{h,-}(t)  -A, (T_{u} , X_{u}(T_{u}))  \notin \Omega_{t,h}^{R}  \right)=0 .
  \end{equation*}
  \end{lemma}

\begin{proof}
  Applying part (i) of  Corollary \ref{cor-reduce-to-E-R}  with $m(t)=m^{2,3}_{h,-}(t)$  and
  $\Omega^{R}_{t}=\Omega^{R}_{t,h}$ defined in \eqref{eq-def-Omega-t-R-1}, it suffices to show that  for each $A, K>0$, $I(t,R) = I(t,R;A,K)$  defined in \eqref{eq-assumption-EYtARK}  vanishes as first $t \to \infty$ and then $R \to \infty$. Conditioned on the Brownian motion $B_{s}$ in \eqref{eq-assumption-EYtARK} equals $ \sqrt{2} \sigma_{t}  s + x $, we have
\begin{equation}\label{eq-def-I-t-R-1}
\begin{aligned}
 I(t,R) &=  \int_{0}^{t} \dif s  \int_{-\infty}^{\sqrt{2} (\sqrt{\beta_{t}}-\sigma_{t})s+ K}  \mathsf{F}_{t}\left(t-s, \sqrt{2} \sigma_{t}^2 s + \sigma_{t} x \right) 1_{\left\{(s,\sqrt{2} \sigma_{t}^2 s + \sigma_{t} x) \notin  \Omega^{R}_{t,h}  \right\}} \\
& \qquad   \quad      \times   \mathbf{P}(B_{r} \leq \sqrt{2 \beta_{t}} r + K, \forall r \leq s \mid B_{s}= \sqrt{2} \sigma_{t}  s + x ) \frac{1}{\sqrt{2 \pi s}} e^{\beta_{t}s} e^{-\frac{(\sqrt{2} \sigma_{t} s + x)^2}{2s}} \dif x \\
& \lesssim_{K}  \int_{0}^{t} \dif s  \int_{-\infty}^{\sqrt{2} (\sqrt{\beta_{t}}-\sigma_{t})s+ K}  \mathsf{F}_{t}\left(t-s, \sqrt{2} \sigma_{t}^2 s + \sigma_{t} x \right) 1_{\left\{(s,\sqrt{2} \sigma_{t}^2 s + \sigma_{t} x) \notin  \Omega^{R}_{t,h}  \right\}} \\
&   \qquad   \qquad  \qquad      \qquad  \qquad   \left[ \sqrt{2}(\sqrt{\beta_{t}}-\sigma_{t}) + \frac{K-x}{s} \right] \frac{1}{\sqrt{s}} e^{(\beta_{t}- \sigma^2_{t})s }e^{-\sqrt{2}\sigma_{t} x - \frac{x^2}{2s}}   \dif x,
\end{aligned}
\end{equation}
where in the inequality we used  $ \mathbf{P}(B_{r} \leq \sqrt{2 \beta_{t}} r + K, \forall r \leq s \mid B_{s}= \sqrt{2} \sigma_{t}   s + x ) \lesssim_{K}   \sqrt{2}(\sqrt{\beta_{t}}-\sqrt{\sigma^{2}_{t}}) + \frac{K-x}{s}$, which holds by Lemma \ref{lem-bridge-estimate-0}.
 By the definition of $\mathsf{F}_{t}(r,x)$ in  Corollary \ref{cor-reduce-to-E-R} with $m(t)=m^{2,3}_{h,-}(t)$, we have
 \begin{align*}
   & \mathsf{F}_{t}(t-s, \sqrt{2} \sigma^2_{t} s + \sigma_{t} x ) = \mathsf{P} \left(\max\limits_{u \in \mathsf{N}_{t-s}} \mathsf{X}_{u}(t-s) \geq  \sqrt{2} t- \frac{3-4h'}{2\sqrt{2}}\log t   - A- \sqrt{2} \sigma^2_{t} s-  \sigma_{t} x  \right)\\
    & = \mathsf{P} \left(\max\limits_{u \in \mathsf{N}_{t-s}} \mathsf{X}_{u}(t-s) \geq  \sqrt{2} (t-s)-  \frac{3}{2\sqrt{2}}\log (t-s+1)  + L_{s,t} - \sigma_{t} x   \right)    ,
 \end{align*}
 where $L_{s,t}:=\sqrt{2}(1-\sigma^2_{t})s - \frac{3}{2\sqrt{2}}\log (\frac{t}{t-s+1})+  \sqrt{2}h' \log t   - A$.  Applying Lemma \ref{lem-Max-BBM-tail}, provided that  $L_{s,t}-\sigma_{t}x >1$  we have
 \begin{equation}\label{eq-bound-F-01}
   \mathsf{F}_{t}(t-s, \sqrt{2} \sigma^2_{t} s + \sigma_{t} x ) \lesssim_{A} (L_{s,t}-\sigma_{t} x ) e^{-2(1-\sigma^2_{t})s} (\frac{t}{t-s+1})^{3/2} \frac{1}{t^{2 h'}} e^{\sqrt{2} \sigma_{t} x  } e^{-\frac{(L_{s,t}-\sigma_{t}x)^2}{3(t-s)}} .
 \end{equation}
 In fact, for large $t$, $L_{s,t}-\sigma_{t}x >1$ holds for all $s \in [0,t]$,  $x \leq \sqrt{2} (\sqrt{\beta_{t}}-\sigma_{t})s+ K$.  To see this, note that
$L_{s,t}-\sigma_{t}x \geq (\sqrt{2}-v_{t})s - \frac{3}{2\sqrt{2}}\log (\frac{t}{t-s+1})+  \sqrt{2}h' \log t  - A-\sigma_{t}K$. By our assumption $(\beta_{t},\sigma_{t}^2) \to (\beta,\sigma^2) \in \mathscr{B}_{II,III}$, for large $t$ we have  $  \sqrt{2}-v_{t} > \delta >0$. Then  for each  $\delta s > 2\log t$,   $ L_{s,t}-\sigma_{t}x > \sqrt{2}h' \log t - A-\sigma_{t} K > 1$;   for each  $\delta s \leq  2\log t$, $t-s+1 \geq t/2$ hence $  L_{s,t}-\sigma_{t}x  \geq \sqrt{2} h' \log t - 4\log (2)  -A -\sigma_{t} K>1$.

 Substituting  the inequality  \eqref{eq-bound-F-01} into \eqref{eq-def-I-t-R-1}, and  by the hypothesis $\beta_{t}+ \sigma^2_{t} =2 - 1/t^{h}$, we get
 \begin{align*}
   I(t,R) & \lesssim  \int_{0}^{t}   \dif s \int_{\mathbb{R}}  1_{\left\{(s,\sqrt{2} \sigma_{t}^2 s + \sigma_{t} x) \notin  \Omega^{R}_{t,h}  \right\}}   \frac{|L_{s,t}-\sigma_{t} x |}{\sqrt{ s } t^{2 h'}}   \frac{t^{3/2} }{(t-s+1)^{3/2}}  e^{-\frac{s}{t^{h}}   - \frac{x^2}{2s} -\frac{(L_{s,t}-\sigma_{t}x)^2}{3(t-s)}} \dif x  .
 \end{align*}
We now make change of variables $s= \xi t^{h'}$ and $x= \eta \sqrt{s}$ in the integral above.  Then  $(s,\sqrt{2} \sigma_{t}^2 s + \sigma_{t} x) $ belongs to $ \Omega^{R}_{t,h}  $ if and only if $  \xi \in [R{^{-1}}, R]$ and  $ |\sigma_{t} \eta | \leq R$.  Applying the dominated convergence theorem twice we get
 \begin{equation*}
   \begin{aligned}
 & I(t,R)  \lesssim
 \int_{0}^{t^{1-h'}} \int_{ \mathbb{R} }
 1_{\left\{ \substack{ \xi \notin [R{^{-1}}, R]    \\ \text{or } |\sigma_{t} \eta | > R } \right\} }
  \frac{ |L(\xi t^{h'},t)-\sigma_{t} x|}{t^{ h'}}
    \\
   &  \qquad \qquad \qquad \qquad \qquad \times \left( \frac{t}{t-\xi t^{h'}+1} \right) ^{3/2}
   e^{-\xi t^{h'-h}  - \frac{\eta^2}{2}}
     e^{-\frac{\left(L(\xi t^{h'},t)-\sigma_{t} \eta \sqrt{\xi t^{h'}} \right)^2}{3(t-\xi t^{h'})}}  \dif \xi \dif \eta   \\
  & \overset{t \to \infty }{\longrightarrow }  \int_{0}^{\infty} \int_{\mathbb{R}}  1_{\left\{ \substack{ \xi \notin [R{^{-1}}, R]    \\ \text{or } | \sigma^2 \eta | > R
  } \right\} }     \sqrt{2}(1-\sigma^2) \xi e^{-\xi 1_{ \{h \leq 1/2\}}}e^{ - \frac{\eta^2}{2}}    e^{- (1-\sigma^2)^2 \xi^2 1_{\{h \geq 1/2\}}  }  \dif \xi \dif \eta  \overset{R \to \infty }{\longrightarrow } 0.
   \end{aligned}
 \end{equation*}
This completes the proof.
 \end{proof}

Next we will prove Theorem \ref{thm-23-approximate} for $(\beta_{t},\sigma^2_{t}) \in
{\mathscr{A}}^{h,-}_{(\beta,\sigma^2)}$.

\begin{proof}[Proof of Theorem \ref{thm-23-approximate}
for $(\beta_{t},\sigma^2_{t}) \in
 {\mathscr{A}}^{h,-}_{(\beta,\sigma^2)}$.
]
Take $\varphi \in \mathcal{T}$. Applying part (ii) of
Corollary \ref{cor-reduce-to-E-R} with $m(t)=m^{2,3}_{h,-}(t)$, $\rho=\sqrt{2}$, and $\Omega^{R}_{t}=\Omega^{R}_{t,h}$ defined in \eqref{eq-def-Omega-t-R-1}, it suffices to study the asymptotic of $ \mathbb{E}\left(e^{-\left\langle\widehat{\mathcal{E}}_t^R, \varphi\right\rangle}\right)  $, can be rewritten as
\begin{equation*}
  \mathbb{E}\bigg(\exp \bigg\{ - \int_{\frac{1}{R}t^{h'}}^{Rt^{h'}} \sum_{u \in N^{1}_{s}}  \Phi_{\sqrt{2}}\big(t-s, X_u(s)-\sqrt{2}s + \frac{3-4h'}{2\sqrt{2}}\log t \big) 1_{\{ (s,X_{u}(s) )\in \Omega^{R}_{t,h}  \} }  \dif s \bigg\} \bigg).
\end{equation*}
Observe that,  uniformly for $(s,X_{u}(s)) \in \Omega^{R}_{t,h} $, we have   $ \sqrt{2}s- X_{u}(s) = \sqrt{2}(1-\sigma^{2}_{t})s+\Theta(\sqrt{s})$.   Lemma \ref{thm-Laplace-BBM-order} yields that,
uniformly for $(s,X_{u}(s))\in \Omega^{R}_{t,h}$, as $t \to \infty$,
 \begin{align}
  & \Phi_{\sqrt{2}}\left(t-s, X_{u}(s)-\sqrt{2}s + \frac{3-4h'}{2\sqrt{2}}\log t \right)  \notag  \\
   & =[(1+o(1))]\gamma_{\sqrt{2}}(\varphi) \frac{(\sqrt{2}s-X_{u}(s)) }{t^{3/2}}e^{\sqrt{2}X_{u}(s)-2 s} t^{\frac{3-4h'}{2}} e^{- \frac{(X_{u}(s)-\sqrt{2}s)^2}{2t}  }  \notag\\
  & =[(1+o(1))]\gamma_{\sqrt{2}}(\varphi) \frac{ \sqrt{2}(1-\sigma^{2}_{t})s }{t^{2h'}} e^{\sqrt{2}X_{u}(s)-(\beta_{t}+\sigma^2_{t}) s-\frac{s}{t^{h}}} e^{- \frac{(1-\sigma^2_{t}) ^2 s^2}{t}  },  \label{eq-asymptotic-F-1}
 \end{align}
where in last equality we replaced  $2$ by $\beta_{t}+\sigma^{2}_{t}+1/t^{h}$.
Let
\begin{equation*}
 W(s;t):= \sum_{u \in N^{1}_{s}}  e^{\sqrt{2}X_{u}(s) - (\beta_{t}+\sigma^2_{t}) s} 1_{\{|X_{u}(s)-\sqrt{2}\sigma^{2}_{t}s| \leq R\sqrt{s} \}} .
\end{equation*}
By the asymptotic equality \eqref{eq-asymptotic-F-1}, we have
\begin{align*}
 &  \mathbb{E}\left(e^{-\left\langle\widehat{\mathcal{E}}_t^R, \varphi\right\rangle}\right)
=  \exp \left\{   [1+o(1)]\gamma_{\sqrt{2}}(\varphi) \int_{\frac{1}{R} t^{h'}}^{Rt^{h'}}   \frac{\sqrt{2}(1-\sigma^{2}_{t})s}{t^{2h'}} e^{-\frac{s}{t^{h}} - \frac{(1-\sigma^2_{t}) ^2 s^2}{t}  }  W(s,t)   \dif s \right\}
\\
& =  \exp \left\{     [1+o(1)] \gamma_{\sqrt{2}}(\varphi)\sqrt{2}(1-\sigma^2_{t})  \int_{\frac{1}{R}}^{R} \xi e^{-\xi t^{h'-h} - (1-\sigma^2_{t})^2 \xi^2 t^{2h'-1}}  W(\xi t^{h'},t)    \dif \xi \right\},
\end{align*}
where in the last equality we made a change variable $ s=\xi t^{h'}$.
By the Brownian scaling property, $(X_u(s):u \in N_s^1)$ has the same distribution as $( \frac{\sigma_{t}}{\sqrt{\beta_{t}}}\mathsf{X}_u(s'):u \in \mathsf{N}_{s'})$, where $s'=\beta_{t}s$. Put $\lambda_{t}= \sqrt{2}\sigma_{t}/\sqrt{\beta_{t}}$. We have
\begin{equation*}
W(s;t)  \overset{law}{=}  \sum_{u \in \mathsf{N}_{s'}}  e^{ \lambda_{t} X_{u}(s') - (1+\lambda_{t}^2/2)s'} 1_{\{| \mathsf{X}_{u}(s')- \lambda_{t}s'| \leq R\sqrt{s'}/ \sigma_{t} \}}.
\end{equation*}
Since $\lambda_{t} \to  \sqrt{2}\sigma/\sqrt{\beta} < \sqrt{2}$, by part(i) of Lemma \ref{lem-functional-convergence-derivative-martingale-2}, we have
\begin{equation*}
 \lim_{t \to \infty} W(\xi t^{h'},t)  =\mathsf{W}_{\infty}( \frac{\sqrt{2}\sigma}{\sqrt{\beta}} ) \int_{[-\frac{R}{\sigma},\frac{R}{\sigma}]} e^{-\frac{x^2} {2}}  \frac{   \dif x  }{\sqrt{2 \pi }} =\mathsf{W}^{\beta,\sigma^2}_{\infty}(\sqrt{2}) \int_{[-\frac{R}{\sigma},\frac{R}{\sigma}]} e^{-\frac{x^2} {2}}  \frac{   \dif x  }{\sqrt{2 \pi }}
 \ \text{ in law.}
\end{equation*}
Applying the dominated convergence theorem, we get
\begin{equation*}
 \lim_{t \to \infty}  \mathbb{E}\left(e^{-\left\langle\widehat{\mathcal{E}}_t^R, \varphi\right\rangle}\right) =
 \mathsf{E}
 \bigg(\exp \big\{ -   \gamma_{\sqrt{2}}(\varphi) \mathsf{W}_{\infty}^{\beta,\sigma^2}(\sqrt{2}) C_{R,h} \big\} \bigg) ,
\end{equation*}
where
\begin{align*}
 &  C_{R,h}:= \sqrt{2}\int_{\frac{1}{R}}^{R}  (1-\sigma^2) \xi e^{-\xi 1_{\{h \leq 1/2\}}- (1-\sigma^2)^2\xi^21_{\{h \geq 1/2\}}}  \dif \xi  \int_{[-\frac{R}{\sigma},\frac{R}{\sigma}]} e^{-\frac{x^2} {2}}  \frac{   \dif x  }{\sqrt{2 \pi }} \\
  & \overset{R \to \infty}{\longrightarrow}  C_{h,-}:=\sqrt{2}(1-  \sigma^{2} ) 1_{\{h<1/2\}}  +  \frac{ 1_{\{h> 1/2\}}}{\sqrt{2}(1-  \sigma^{2} )}  \\
  &\qquad \qquad \qquad \qquad \qquad  + \sqrt{2}\int_{0}^{\infty}  (1-\sigma^2) \xi e^{-\xi - (1-\sigma^2)^2\xi^2 }  \dif \xi 1_{\{ h=1/2\}} .
\end{align*}
Letting $R\to \infty$ and applying Corollary \ref{cor-reduce-to-E-R}, we get
\begin{equation*}
  \lim_{t \to \infty}  \mathbb{E}\left(e^{-\left\langle\widehat{\mathcal{E}}_t, \varphi\right\rangle}\right)= \lim_{R \to \infty}  \lim_{t \to \infty}  \mathbb{E}\left(e^{-\left\langle\widehat{\mathcal{E}}_t^R, \varphi\right\rangle}\right) =
  \mathsf{E}
  \bigg(\exp \big\{ - C_{h,-}  \mathsf{W}_{\infty}^{\beta,\sigma^2}(\sqrt{2})  \gamma_{\sqrt{2}}(\varphi)  \big\} \bigg) ,
\end{equation*}
which is the Laplace functional of $\operatorname{DPPP}\left( C_{h,-} \sqrt{2}C_{\star}   \mathsf{W}_{\infty}^{\beta,\sigma^2}(\sqrt{2})  e^{-\sqrt{2} x} \dif x, \mathfrak{D}^{\sqrt{2}}\right)$. Using \cite[Lemma 4.4]{BBCM22}, we complete the proof.
\end{proof}

\subsection{Approaching  $(1,1)$ from $\mathcal{C}_{II}$}

Assume that now  $(\beta_{t},\sigma^{2}_{t})_{t>0} \in \mathscr{A}^{h,2}_{(1,1)}$, i.e., $\beta_{t} = \sigma_{t}^2= 1- \frac{1}{2 t^{h}}$. Let $ m^{(1,1)}_{h,2}(t)= \sqrt{2} t-\frac{3-2h'}{2\sqrt{2}}\log t$, where $h'=\min\{h,1\}$.   Define
\begin{equation}\label{eq-def-Omega-t-R-10}
  \Omega_{t,h}^{R}=
    \begin{cases}
    \{ (s,x): s \in  [\frac{1}{R} t^{h}, R t^{h}], \    \sqrt{2} \sigma_{t}^2 s- x \in   [\frac{1}{R} \sqrt{s}, R\sqrt{s}]     \} \quad \text{ for } h \in (0,1) ; \\
    \{ (s,x): s \in [\frac{1}{R}t, (1-\frac{1}{R})t] , \       \sqrt{2} \sigma_{t}^2 s- x \in   [\frac{1}{R} \sqrt{s}, R\sqrt{s}]   \} \quad \text{ for } h \in [1,\infty].
  \end{cases}
\end{equation}

\begin{lemma}\label{lem-contribution-area-10}
  For any $A>0$,
  \begin{equation*}
  \lim _{R \rightarrow \infty} \limsup _{t \rightarrow \infty} \mathbb{P}\left(\exists u \in N^{2}_{t}: X_{u}(t) \geq  m^{(1,1)}_{h,2}(t)  -A, (T_{u} , X_{u}(T_{u}))  \notin \Omega_{t,h}^{R}  \right)=0 .
  \end{equation*}
  \end{lemma}

\begin{proof}
 As in the proof of Lemma \ref{lem-contribution-area-1}, applying  Corollary \ref{cor-reduce-to-E-R} with $m(t)=m^{(1,1)}_{h,2}(t)$,   and $\Omega^{R}_{t}=\Omega^{R}_{t,h}$ defined in \eqref{eq-def-Omega-t-R-10}, it suffices to show that  for each $A, K>0$, $I(t,R) = I(t,R;A,K)$,  defined in \eqref{eq-assumption-EYtARK},  vanishes as first $t \to \infty$ and then $R \to \infty$. As  in \eqref{eq-def-I-t-R-1}, noting that $\beta_{t}=\sigma_{t}^2$,  it now suffices to show that
\begin{equation}\label{eq-def-I-t-R-10}
\begin{aligned}
I(t,R) \lesssim  \int_{0}^{t} \dif s  \int_{-\infty}^{  K}  \mathsf{F}_{t} (t-s, \sqrt{2} \sigma_{t}^2 s + \sigma_{t} x  ) 1_{\left\{(s,\sqrt{2} \sigma_{t}^2 s + \sigma_{t} x) \notin  \Omega^{R}_{t,h}  \right\}}   \frac{K-x}{s^{3/2}}  e^{-\sqrt{2}\sigma_{t} x - \frac{x^2}{2s}}   \dif x
\end{aligned}
\end{equation}
vanishes as first $t \to \infty$ and then $R \to \infty$.
 Let $L_{s,t}:=\sqrt{2}(1-\sigma^2_{t})s +  \frac{h'}{\sqrt{2}}  \log t - \frac{3}{2\sqrt{2}}\log (\frac{t}{t-s+1})  - A$.
\begin{itemize}
  \item
   If $L_{s,t}-\sigma_{t}x>1 $,
   by Lemma \ref{lem-Max-BBM-tail},
\begin{align*}
  & \mathsf{F}_{t}(t-s, \sqrt{2} \sigma^2_{t} s + \sigma_{t} x ) \\
    & = \mathsf{P} \left(\max\limits_{u \in \mathsf{N}_{t-s}} \mathsf{X}_{u}(t-s) \geq  \sqrt{2} (t-s)-  \frac{3}{2\sqrt{2}}\log (t-s+ 1)  + L_{s,t} - \sigma_{t} x   \right)    \\
  & \lesssim_{A} (L_{s,t}-\sigma_{t} x ) e^{-2(1-\sigma^2_{t})s} (\frac{t}{t-s+1})^{3/2} \frac{1}{t^{h'}} e^{\sqrt{2} \sigma_{t} x  } e^{-\frac{(L_{s,t}-\sigma_{t}x)^2}{3(t-s)}}.
\end{align*}
\item If $ L_{s,t}-\sigma_{t}x \leq 1$, we simply upper bound  $ \mathsf{F}_{t}(t-s, \sqrt{2} \sigma^2_{t} s + \sigma_{t} x ) $ by $1$. Moreover, $ L_{s,t}-\sigma_{t}x \leq 1$ holds only if $h \geq 1$, $s \geq t/2$ and $ -\sigma_{t}x \leq  1- L_{s,t} \leq - \frac{h'}{\sqrt{2}}  \log t + \frac{3}{2\sqrt{2}}\log (\frac{t}{t-s+1})+ A+1  $.  Thus
\begin{equation*}
  \int_{t/2}^{t} \dif s  \int_{ -O(\log t)}^{K}   \frac{K-x}{s^{3/2}}  e^{-\sqrt{2}\sigma_{t} x - \frac{x^2}{2s}}   1_{ \{  L_{s,t}-\sigma_{t}x \leq 1 \}}  \dif x  \leq \frac{O(\log^2 t)}{t^{h'}}  \int_{0}^{t/2}\frac{\dif u}{(u+1)^{3/2}}  \to 0.
\end{equation*}
\end{itemize}
Substituting these inequalities into \eqref{eq-def-I-t-R-10}, and using the assumption $\beta_{t}+ \sigma^2_{t}-2= - t^{-h}$, we get
\begin{align*}
 I(t,R) \lesssim
   \int_{0}^{t}  \dif s & \int_{-\infty}^{K}      1_{\left\{(s,\sqrt{2} \sigma_{t}^2 s + \sigma_{t} x) \notin  \Omega^{R}_{t,h}  \right\}}    \frac{K-x}{s^{3/2}}    e^{-\frac{s}{t^{h}} }e^{ - \frac{x^2}{2s}}   \\
   &\qquad \times  \frac{ | L_{s,t}-\sigma_{t} x | }{t^{  h'}}   (\frac{t}{t-s+1})^{3/2}  e^{-\frac{(L_{s,t}-\sigma_{t}x)^2}{3(t-s)}}    \dif x  +o(1).
\end{align*}

In the case  $h \in (0,1)$, make change of variables $s= \xi t^{h}$ and $-x= \eta \sqrt{s}$. Noting that $(s,\sqrt{2} \sigma_{t}^2 s + \sigma_{t} x) \in  \Omega^{R}_{t,h}$ if and only if $\xi \in [R^{-1},R] $ and $ \sigma_{t} \eta \in [R^{-1} , R  ]$, applying the dominated convergence theorem twice we get
\begin{equation*}
  \begin{aligned}
    I_{1}(t,R,K) &\lesssim \int_{0}^{t^{1-h}} \dif \xi \int_{-K/\sqrt{\xi t^{h}}}^{\infty}  1_{\left\{ \substack{ \xi \notin [R{^{-1}}, R]    \\ \text{or } \sigma_{t} \eta \in [R^{-1} , R  ]} \right\} }   \frac{K+\eta \sqrt{\xi t^{h}}}{\xi t^{h}}   e^{-\xi }e^{ - \frac{\eta^2}{2}} \\
    &\qquad \qquad  \qquad \qquad \qquad \times  [ L(\xi t^{h},t)+\sigma_{t} \eta \sqrt{\xi t^{h}} ]  (\frac{t}{t-\xi t^{h}+1})^{3/2}   \dif \eta    \\
 & \overset{t \to \infty }{\longrightarrow }  \int_{0}^{\infty}  \int_{0}^{\infty}  1_{\left\{ \substack{ \xi \notin [R{^{-1}}, R]    \\ \text{or }  \eta \in [R^{-1} , R  ]
 } \right\} }  e^{-\xi  }e^{ - \frac{\eta^2}{2}}   \eta^2   \dif \xi \dif \eta \overset{R \to \infty} \to 0.
  \end{aligned}
\end{equation*}
If $h \geq 1$, make  change of variables $s= \xi t$ and $-x= \eta \sqrt{s}$. Noting that $(s,\sqrt{2} \sigma_{t}^2 s + \sigma_{t} x) \in  \Omega^{R}_{t,h}$ if and only if $\xi \in [R^{-1},1-R^{-1}] $ and $ \sigma_{t} \eta \in [R^{-1} , R  ]$, applying the dominated convergence theorem twice we get
\begin{equation*}
  \begin{aligned}
    I_{1}(t,R,K) &\lesssim \int_{0}^{1} \int_{-\frac{K}{\sqrt{\xi t}}}^{\infty}  1_{\left\{ \substack{ \xi \notin [R{^{-1}}, 1-R^{-1}]    \\ \text{or } \sigma_{t} \eta \in [R^{-1} , R  ]
} \right\} }   e^{ - \frac{\eta^2}{2}}   \frac{K+\eta \sqrt{\xi t}}{\xi t} \\
&\qquad \qquad\quad\quad \times   [ L(\xi t,t)+\sigma_{t} \eta \sqrt{\xi t} ]
(1-\xi)^{-3/2}
 e^{-\frac{(L_{\xi t,t}-\sigma_{t}x)^2}{3(t-\xi t)}}    \dif \xi \dif \eta    \\
 & \overset{t \to \infty }{\longrightarrow }  \int_{0}^{1} \int_{0}^{\infty}  1_{\left\{ \substack{ \xi \notin [R{^{-1}}, R]    \\ \text{or }  \eta \in [R^{-1} , R  ]
 } \right\} }   e^{ - \frac{\eta^2}{2}}   \eta^2  (\frac{1}{1-\xi})^{3/2}  e^{-\frac{\eta^2}{2} \frac{ \xi}{1-\xi} }      \dif \xi \dif \eta \overset{R \to \infty }{\longrightarrow }    0,
  \end{aligned}
\end{equation*}
where we used the fact $ \int_{0}^{1} (\frac{1}{1-\xi})^{3/2}   \dif \xi \int_{0}^{\infty}      \eta^2   e^{-\frac{\eta^2}{2} \frac{ 1}{1-\xi} }    \dif \eta =  \int_{0}^{1}    \dif \xi \int_{0}^{\infty}      \lambda^2   e^{-\frac{\lambda^2}{2}   }    \dif \lambda   < \infty $.
\end{proof}

\begin{proof}[Proof of Theorem \ref{thm-11-approximate}
for $(\beta_{t},\sigma^{2}_{t})_{t>0} \in \mathscr{A}^{h,2}_{(1,1)}$
]
  Take $\varphi \in \mathcal{T}$. Applying
  Corollary \ref{cor-reduce-to-E-R} with $m(t)=m^{(1,1)}_{h,2}(t)$, $\rho=\sqrt{2}$, and $\Omega^{R}_{t}=\Omega^{R}_{t,h}$ defined in \eqref{eq-def-Omega-t-R-1}, it suffices to study the asymptotic of   $\mathbb{E}\left(e^{-\left\langle\widehat{\mathcal{E}}_t^R, \varphi\right\rangle}\right) $,  which equals
  \begin{equation*}
   \mathbb{E}\bigg(\exp \bigg\{ - \int_{0}^{t} \sum_{u \in N^{1}_{s}}  \Phi_{\sqrt{2}}\big(t-s, X_u(s)-\sqrt{2}s + \frac{3-4h'}{2\sqrt{2}}\log t \big) 1_{\{ (s,X_{u}(s) )\in \Omega^{R}_{t,h}  \} }  \dif s \bigg\} \bigg).
  \end{equation*}
  Note that  uniformly for $(s,X_{u}(s)) \in \Omega^{R}_{t,h} $, we have   $ \sqrt{2}s- X_{u}(s) =  \Theta(\sqrt{s})=\Theta(t^{\frac{h'}{2}})$.  Then Lemma \ref{thm-Laplace-BBM-order} yields that
   \begin{align*}
  & \Phi_{\sqrt{2}}(t-s, X_{u}(s)-\sqrt{2}s + \frac{3-2h'}{2\sqrt{2}}\log t )   \\
  &= [1+o(1)]
 \gamma_{\sqrt{2}}(\varphi) \frac{(\sqrt{2}s-X_{u}(s)) }{(t-s)^{3/2}}e^{\sqrt{2}X_{u}(s)-2 s} t^{\frac{3-2h'}{2}} e^{- \frac{(X_{u}(s)-\sqrt{2}s)^2}{2(t-s)}  }
   \end{align*}
  as $t \to \infty$ uniformly for $(s,X_{u}(s))\in \Omega^{R}_{t,h}$.

  \textbf{Case 1: $h \in (0,1)$}. Since now $s=\Theta(t^{h}) \ll t$, in fact we have
   \begin{align*}
     \Phi_{\sqrt{2}}(t-s, X_{u}(s)-\sqrt{2}s + \frac{3-2h'}{2\sqrt{2}}\log t )
    \sim \gamma_{\sqrt{2}}(\varphi)
    \frac{\sqrt{2}\beta_{t}s- X_{u}(s)}{t^{h}} e^{\sqrt{2}X_{u}(s)-2 \beta_{t} s-\frac{s}{t^{h}}}
   \end{align*}
  as $t \to \infty$ uniformly in $(s,X_{u}(s))\in \Omega^{R}_{t,h}$, where we used that  $\beta_{t}=1-\frac{1}{2 t^{h}}$. Using this we can rewrite   $ \mathbb{E}\left(e^{-\left\langle\widehat{\mathcal{E}}_t^R, \varphi\right\rangle}\right) $ as
  \begin{align*}
      \mathbb{E}\bigg(\exp \bigg\{ - [1+o(1)] \gamma_{\sqrt{2}}(\varphi)  \int_{\frac{1}{R} t^{h}}^{Rt^{h}} \sum_{u \in N^{1}_{s}}   \frac{\sqrt{2}\beta_{t}s- X_{u}(s)}{t^{h}} e^{\sqrt{2}X_{u}(s)-2 \beta_{t} s-\frac{s}{t^{h}}} 1_{\{ (s,X_{u}(s)) \in \Omega_{t,h}^{R}\}}  \dif s \bigg\} \bigg) .
  \end{align*}
  Making a change of variable $ s=\lambda t^{h}$,
  and noticing that
   $\beta_{t}=\sigma^{2}_{t}$,   $\{X_{u}(s):u \in N^{1}_{s}\}$ has the same law of $\{\mathsf{X}_{u}(\beta_{t}s): u\in \mathsf{N}_{\beta_{t}s}\}$, we have
  \begin{equation*}
  \mathbb{E}\left(e^{-\left\langle\widehat{\mathcal{E}}_t^R, \varphi\right\rangle}\right)   =
  \mathsf{E}
   \bigg(\exp \bigg\{ -  [1+o(1)] \gamma_{\sqrt{2}}(\varphi) \int_{\frac{1}{R} }^{R}  \mathsf{Z}^{R}(\lambda \beta_{t} t^{h} )  e^{-\lambda}     \dif \lambda \bigg\} \bigg) ,
  \end{equation*}
where
\begin{equation*}
  \mathsf{Z}^{R}(t) := \sum_{u \in \mathsf{N}_t}[\sqrt{2}t- \mathsf{X}_{u}(t)]  e^{\sqrt{2}\mathsf{X}_{u}(t)-2 t}   1_{\{ \sqrt{2}t-\mathsf{X}_{u}(t) \in [\frac{1}{R}\sqrt{t}, R\sqrt{t}]   \}}.
\end{equation*}
 By Lemma \ref{lem-functional-convergence-derivative-martingale}, for each $\lambda > 0$,  $   \mathsf{Z}^{R}(t) \to \mathsf{Z}_{\infty}  \sqrt{ \frac{2}{\pi}}\int_{1/R}^{R} z^2 e^{-\frac{z^2}{2}} \dif z$ in probability.
 Letting $t \to \infty$ then $R\to \infty$ and applying the dominated convergence theorem and  Corollary \ref{cor-reduce-to-E-R},   we have
  \begin{align*}
    &\lim_{t \to \infty}  \mathbb{E}\left(e^{-\left\langle\widehat{\mathcal{E}}_t, \varphi\right\rangle}\right)= \lim_{R \to \infty}  \lim_{t \to \infty}  \mathbb{E}\left(e^{-\left\langle\widehat{\mathcal{E}}_t^R, \varphi\right\rangle}\right)\\
    & = \lim_{R \to \infty}
    \mathsf{E}
     \bigg(\exp \bigg\{ -  \gamma_{\sqrt{2}}(\varphi)\mathsf{Z}_{\infty} \int_{\frac{1}{R} }^{R}    e^{-\lambda}     \dif \lambda  \sqrt{ \frac{2}{\pi}}\int_{\frac{1}{R}}^{R} z^2 e^{-\frac{z^2}{2}} \dif z\bigg\} \bigg)  =
      \mathsf{E}
     \bigg( e^{ - \gamma_{\sqrt{2}}(\varphi) \mathsf{Z}_{\infty}   } \bigg),
  \end{align*}
  which is the Laplace functional of $\operatorname{DPPP}\left( \sqrt{2}C_{\star}   \mathsf{Z}_{\infty}   e^{-\sqrt{2} x} \dif x, \mathfrak{D}^{\sqrt{2}}\right)$.

  \textbf{Case 2: $h \in [1,\infty]$}.  Notice that for  $(s,X_{u}(s)) \in \Omega^{R}_{t,h} $,
  we have
  $s = \Theta(t)$ and  $ \sqrt{2}s- X_{u}(s) =   \Theta(\sqrt{t})$. Now Lemma \ref{thm-Laplace-BBM-order} yields that
   \begin{align*}
    & \Phi_{\sqrt{2}}(t-s, X_{u}(s)-\sqrt{2}s + \frac{3-2h'}{2\sqrt{2}}\log t )   \\
     & =[1+o(1)]\gamma_{\sqrt{2}}(\varphi) \frac{t^{3/2}}{(t-s)^{3/2}t}[\sqrt{2} \beta_{t} s-X_{u}(s)]e^{\sqrt{2}X_{u}(s)-2\beta_{t} s-\frac{s}{t^{h}}}  e^{-\frac{ s}{2(t-s)} \frac{(X_{u}(s)-\sqrt{2}\beta_{t} s)^2}{\beta_{t} s }  }
   \end{align*}
   as $t \to \infty$ uniformly for $(s,X_{u}(s))\in \Omega^{R}_{t,h}$, where we used the fact that  $\beta_{t}=1-\frac{1}{2 t^{h}}$. Then $ \mathbb{E}\left(e^{-\left\langle\widehat{\mathcal{E}}_t^R, \varphi\right\rangle}\right) $ equals
   \begin{multline*}
        \mathbb{E}\bigg(\exp \bigg\{ - [1+o(1)] \gamma_{\sqrt{2}}(\varphi) \int_{\frac{1}{R} t}^{(1-\frac{1}{R})t} \sum_{u \in N^{1}_{s}} \frac{t^{3/2} [\sqrt{2} \beta_{t} s-X_{u}(s)]}{t(t-s)^{3/2}} \,  \exp \{-\frac{s}{t^{h}} \}
       \\    \times \exp\left\{\sqrt{2}X_{u}(s)-2\beta_{t} s\right\}  \exp\left\{- \frac{(X_{u}(s)-\sqrt{2}\beta_{t} s)^2}{2\beta_{t}(t-s)}  \right\}
       1_{\{ (s,X_{u}(s)) \in \Omega_{t,h}^{R}\}}   \dif s \bigg\} \bigg)
   \end{multline*}
    Making a change of variable  $ s=\lambda t$, we get that $  \mathbb{E}\left(e^{-\left\langle\widehat{\mathcal{E}}_t^R, \varphi\right\rangle}\right) $ is equal to
   \begin{equation*}
  \mathsf{E}
   \bigg(\exp \bigg\{ -  [1+o(1)] \gamma_{\sqrt{2}}(\varphi)  \int_{\frac{1}{R} }^{1-\frac{1}{R}} \frac{ e^{-\lambda 1_{\{h=1\}}}}{(1-\lambda)^{3/2}}  \sqrt{\lambda \beta_{t} t} \mathsf{W}^{G_{\lambda}}(\lambda \beta_{t} t; \sqrt{2})  \dif \lambda \bigg\} \bigg) ,
     \end{equation*}
where $G_{\lambda}(x)=G_{\lambda,R}(x)= x e^{-\frac{\lambda}{2(1-\lambda)} x^2 } 1_{ \{x \in [R^{-1},R]\}}$, and
     \begin{equation*}
    \mathsf{W}^{G_{\lambda}}(t;\sqrt{2}) = \sum_{u \in \mathsf{N}_t}   e^{\sqrt{2}\mathsf{X}_{u}(t)-2 t}   G_{\lambda}\left( \frac{\sqrt{2}t- \mathsf{X}_{u}(t) }{\sqrt{t}}  \right) .
     \end{equation*}
  It follows from Lemma \ref{lem-functional-convergence-derivative-martingale} that  for every $\lambda > 0$,  $ \lim\limits_{t \to \infty} \sqrt{t} \mathsf{W}^{G_{\lambda}}(t;\sqrt{2}) = \mathsf{Z}_{\infty}  \sqrt{ \frac{2}{\pi}}\int_{1/R}^{R} G_{\lambda}(z) z e^{-\frac{z^2}{2}} \dif z = \mathsf{Z}_{\infty}  \sqrt{ \frac{2}{\pi}}\int_{1/R}^{R}   z^2 e^{-\frac{z^2}{2(1-\lambda)}} \dif z $.
    Letting $t \to \infty$ first and  then $R\to \infty$,
    and applying the dominated convergence theorem and  Corollary \ref{cor-reduce-to-E-R},   we have
    \begin{align*}
      &\lim_{t \to \infty}  \mathbb{E}\left(e^{-\left\langle\widehat{\mathcal{E}}_t, \varphi\right\rangle}\right)= \lim_{R \to \infty}  \lim_{t \to \infty}  \mathbb{E}\left(e^{-\left\langle\widehat{\mathcal{E}}_t^R, \varphi\right\rangle}\right)\\
      & = \lim_{R \to \infty}
     \mathsf{E}
       \bigg(\exp \bigg\{ -  \gamma_{\sqrt{2}}(\varphi)\mathsf{Z}_{\infty} \int_{\frac{1}{R} }^{1-\frac{1}{R}}  \frac{  e^{-\lambda 1_{\{h=1\}}}}{(1-\lambda)^{3/2}}    \sqrt{ \frac{2}{\pi}}\int_{\frac{1}{R}}^{R} z^2 e^{-\frac{z^2}{2(1-\lambda)}} \dif z  \dif \lambda \bigg\} \bigg)
      \\
     & =    \mathsf{E}
      \bigg(\exp \big\{ - C_{h,2}\gamma_{\sqrt{2}}(\varphi) \mathsf{Z}_{\infty}    \big\} \bigg),
    \end{align*}
    which is the Laplace functional of $\operatorname{DPPP}\left( C_{h,2}\sqrt{2}C_{\star}   \mathsf{Z}_{\infty}   e^{-\sqrt{2} x} \dif x, \mathfrak{D}^{\sqrt{2}}\right)$.
   Note that $ C_{h,2}=  1$ if $h > 1$ and $ C_{h,2}=  ( 1- e^{-1}) $ if $ h=1 $, and we complete the proof.
  \end{proof}

\begin{appendix}

\section{Proof of Lemma \ref{lem-functional-convergence-derivative-martingale-2}}\label{App-A}

\begin{proof}[Proof of Lemma \ref{lem-functional-convergence-derivative-martingale-2}]
To prove (ii), we only need to prove
\begin{equation*}
  \lim_{t \to \infty} \frac{ \alpha_{t} }{ \langle F_{t},\mu_{\mathrm{Gau}} \rangle }  \mathsf{W}^{F_t}_{t}(\lambda_{t})
     = 2\sqrt{2} \mathsf{Z}_{\infty} \ \text{ in probability.}
\end{equation*}
 The case $F_{t}\equiv 1$ was proved in \cite[Theorem 1.1 (iv)]{Pain18}.
   That is,
  $    \alpha_{t} \mathsf{W}_{t}(\lambda_{t})  $ converges to $  2 \sqrt{2} \mathsf{Z}_{\infty} $ in probability. So it suffices to show that
  \begin{equation}\label{suffice}
 \zeta_{t}:= \alpha_{t} \left(  \frac{ \mathsf{W}^{F_t}_{t}(\lambda_{t})}{ \langle F_{t},\mu_{\mathrm{Gau}} \rangle }  - \mathsf{W}_{t}(\lambda_{t}) \right) \to 0  \text{ in probability.}
  \end{equation}
  Note that the above is also sufficient for (i). We are now  left to prove \eqref{suffice}.

Take   $k_{t}, t> 0, $ such that $k_{t} \leq  t/2 $,  $k_{t}/ \alpha_{t}^{2} \to \infty$ and  $k_{t}/t \to 0$ as $t \to \infty$.
Sometimes we also write $k(t)$.
Such $(k_{t})_{t>0}$ exists by the hypothesis $\alpha_{t}=o(\sqrt{t})$.
 For each $v \in \mathsf{N}_{k(t)}$, let $X^{v}_{u}(s):= X_{u}(k_{t}+s) - X_{v}(k_{t})$, $u \in \mathsf{N}^{v}_{s}$, where $\mathsf{N}^{v}_{s}$ are descendants of $v$ at time $k_{t}+s$.  The branching property yields that conditioned on $\mathcal{F}_{k(t)}$, $\{ X^{v} : v \in \mathsf{N}_{k(t)}\}$ are independent standard BBMs.
We rewrite $\mathsf{W}^{F_t}_{t}(\lambda_{t})$ as
$$\mathsf{W}^{F_t}_{t}(\lambda_{t}) = \sum_{v \in \mathsf{N}_{k(t)}} e^{\lambda_{t}\mathsf{X}_{v}(k_{t})-(\frac{\lambda_{t}^2}{2}+1)k_{t}}  \mathsf{W}^{v,F_t}_{t-k_{t}}(\lambda_{t}), $$
where $\mathsf{W}^{v,F_t}_{t-k_{t}}(\lambda_{t})$ is defined to be
  \begin{equation*}
  \sum_{u \in \mathsf{N}^{v}_{t-k(t)}} e^{\lambda_{t}\mathsf{X}^{v}_{u}(t-k_{t})-(\frac{\lambda_{t}^2}{2}+1)(t-k_{t})} F_{t} \left( \frac{\lambda_{t} k_{t}- \mathsf{X}_{v}(k_{t})}{\sqrt{t}}+ \frac{\lambda_{t}(t-k_{t})- \mathsf{X}_{u}^{v}(t-k_{t}) }{\sqrt{t}} \right) .
  \end{equation*}
By the many-to-one formula,
  \begin{align*}
    \mathsf{E} \left[   \mathsf{W}^{v,F_t}_{t-k_{t}}(\lambda_{t})   | \mathcal{F}_{k(t)} \right]  & =   \mathsf{E} \left[   \mathsf{W}^{F_t(y + \cdot )}_{t-k_{t}}(\lambda_{t})   \right] |_{y=\frac{\lambda_{t} k_{t}- \mathsf{X}_{v}(k_{t})}{\sqrt{t}} } \\
    & =\frac{\sqrt{t}}{\sqrt{2 \pi t-k_{t}}} \int F_{t}(y+  z) e^{-\frac{t}{2(t-k_{t})}z^2} \dif z |_{y=\frac{\lambda_{t} k_{t}- \mathsf{X}_{v}(k_{t})}{\sqrt{t}} }.
  \end{align*}
 For $y \in \mathbb{R}$,  define
 $ \delta_{t}(y) :=  \frac{\sqrt{t}}{\sqrt{ t-k_{t}}}     (\int F_{t}(y+  z)  e^{-\frac{t}{2(t-k_{t})}z^2} \dif z ) / (\int F_{t}(z) e^{-\frac{z^2}{2}} \dif z   ) $. Then
 \begin{equation*}
  \mathsf{E} \left[  \frac{ \mathsf{W}^{v,F_t}_{t-k_{t}}(\lambda_{t})}{  \langle F_{t}, \mu_{\mathrm{Gau}}\rangle} | \mathcal{F}_{k(t)} \right] =  \delta_{t} \left(\frac{\lambda_{t} k_{t}- \mathsf{X}_{v}(k_{t})}{\sqrt{t}} \right).
\end{equation*}
\textbf{Step 1}.  Let
   \begin{equation*}
   \widetilde{\mathsf{W}}_{k_{t}}(\lambda_{t}):= \sum_{v \in \mathsf{N}_{k(t)}} e^{\lambda_{t}\mathsf{X}_{v}(k_{t})-(\frac{\lambda_{t}^2}{2}+1)k_{t}}  \, \delta_{t} \left(\frac{\lambda_{t} k_{t}- \mathsf{X}_{v}(k_{t})}{\sqrt{t}} \right),
   \end{equation*}
  and $p_{t}= 1+ \frac{1}{2 \alpha_{t}}$.  Then, using
 the von Bahr-Esseen inequality  which claims that
  for any sequence $\left(X_i\right)_{i \in \mathbb{N}}$ of independent centered random variables and for any $\gamma \in[1,2]$, $\mathbb{E}\left[\left|\sum X_i\right|^\gamma\right] \leq$ $2 \sum \mathbb{E}\left[\left|X_i\right|^\gamma\right]$, we get
   \begin{align*}
 & \mathsf{E} \left[  \left|  \frac{ \mathsf{W}^{F_t}_{t}(\lambda_{t})}{ \langle F_{t},\mu_{\mathrm{G}} \rangle }  -  \widetilde{\mathsf{W}}_{k_{t}}(\lambda_{t})  \right|^{p_{t}} \big|  \mathcal{F}_{k(t)} \right] \\
 & \leq  2   \sum_{v \in \mathsf{N}_{k(t)}} e^{\lambda_{t}p_{t} \mathsf{X}_{v}(k_{t})    -(\frac{\lambda_{t}^2}{2}+1)p_{t}k_{t}}    \mathsf{E}\left|   \frac{ \mathsf{W}^{v,F_t}_{t-k_{t}}(\lambda_{t})}{  \langle F_{t}, \mu_{\mathrm{G}}\rangle}     - \delta_{t} \left(\frac{\lambda_{t} k_{t}- \mathsf{X}_{v}(k_{t})}{\sqrt{t}} \right) \right| ^{p_{t}}  \\
 & \leq 2^{p_{t}+1}\sum_{v \in \mathsf{N}_{k(t)}} e^{\lambda_{t}p_{t} \mathsf{X}_{v}(k_{t})    -(\frac{\lambda_{t}^2}{2}+1)p_{t}k_{t}}   \left[  \frac{\mathsf{E} |  \mathsf{W}^{v,F_t}_{t-k_{t}}(\lambda_{t}) |^{p_{t}}}{  \langle F_{t}, \mu_{\mathrm{G} }\rangle^{p_{t}}}  +    \delta_{t} \left(\frac{\lambda_{t} k_{t}- \mathsf{X}_{v}(k_{t})}{\sqrt{t}} \right)   ^{p_{t}} \right],
   \end{align*}
   where in the last inequality we used $
    |x+y|^p \leq   2^p\left(|x|^p+|y|^p\right)$.
    \begin{itemize}
      \item Firstly, since $F_{t}$ is bounded,  $\mathsf{E} |  \mathsf{W}^{v,F_t}_{t-k_{t}}(\lambda_{t}) |^{p_{t}} \leq \|G\|_{\infty}^{2} \mathsf{E} |  \mathsf{W}_{t-k_{t}}(\lambda_{t}) |^{p_{t}}  \leq \|G\|_{\infty}^{2} \mathsf{E} | \mathsf{W}_{\infty}(\lambda_{t}) |^{p_{t}} \leq \|G\|_{\infty}^{2} c_{8}$, where $c_{8}>0$ is the constant in  \cite[(4.1)]{Pain18}.
      \item  Secondly, suppose $\mathrm{supp}(G)\subset [-A,A]$. If $F_{t}(z) >0$, then  $ |z-r_{t}| \leq A h_{t}$, and hence  $|z| \leq \bar{r}+ A \bar{h}$.  Then for large $t$
  \begin{align*}
   \delta_{t}(y)  &=  \frac{\sqrt{t}}{\sqrt{  t-k_{t}}}  \frac{\int F_{t}(z)  e^{-\frac{t}{2(t-k_{t})}(z-y)^2}   \dif z  } {\int F_{t}(z) e^{-\frac{z^2}{2}} \dif z} \leq  2  \frac{\int F_{t}(z)  e^{- \frac{(z-y)^2}{2}}   \dif z  } {\int F_{t}(z) e^{-\frac{z^2}{2}} \dif z}  \\
  & \leq  2  \frac{\int F_{t}(z)  e^{- \frac{z^2}{2} +   |z| |y| }   \dif z  } {\int F_{t}(z) e^{-\frac{z^2}{2}} \dif z} e^{- \frac{y^{2}}{2} }
  \leq  2 e^{  (\bar{r}+ A \bar{h}) |y|- \frac{y^2}{2}}  \leq \max_{y } 2 e^{  (\bar{r}+ A \bar{h}) |y|- \frac{y^2}{2}}  := c_{0}<\infty.
  \end{align*}
    \end{itemize}
  Therefore, we have
  \begin{align*}
  &  \mathsf{E} \left(  \alpha_{t}^{p_{t}}\left|  \frac{ \mathsf{W}^{F_t}_{t}(\lambda_{t})}{ \langle F_{t},\mu_{\mathrm{G}} \rangle }  -  \widetilde{\mathsf{W}}_{k_{t}}(\lambda_{t})  \right|^{p_{t}}  \right)
   \\
   &  \leq
    \alpha_{t}^{p_{t}} 2^{p_{t}+1}   \left[ \frac{\|G\|_{\infty}^{2} c_{8} }{  \langle F_{t}, \mu_{\mathrm{G} }\rangle^{p_{t}}} + (c_0)^{p_{t}}  \right]  \mathsf{E} \left[ \sum_{v \in \mathsf{N}_{k(t)}} e^{\lambda_{t}p_{t} \mathsf{X}_{v}(k_{t})    -(\frac{\lambda_{t}^2}{2}+1)p_{t}k_{t}} \right]  \\
    & =  \alpha_{t}^{p_{t}} 2^{p_{t}+1}   \left[ \frac{\|G\|_{\infty}^{2} c_{8} }{  \langle F_{t}, \mu_{\mathrm{G} }\rangle^{p_{t}}} + (c_0)^{p_{t}}  \right]  \exp\left\{ (\frac{\lambda_{t}^2p_{t}^2}{2}+1)k_{t} - (\frac{\lambda_{t}^2}{2}+1)p_{t}k_{t}  \right\}   \overset{t \to \infty}{\longrightarrow} 0,
  \end{align*}
where for the limit above we used the fact that
  $(\frac{\lambda_{t}^2p_{t}^2}{2}+1)k_{t} - (\frac{\lambda_{t}^2}{2}+1)p_{t}k_{t} = (\frac{\lambda_{t}^2p_{t}}{2}-1)(p_{t}-1)k_{t}=-\Theta(\frac{k_{t}}{\alpha_{t}^2}) \to -\infty$. As a consequence,
 \begin{equation*}
  \alpha_{t} \left|  \frac{ \mathsf{W}^{F_t}_{t}(\lambda_{t})}{ \langle F_{t},\mu_{\mathrm{G}} \rangle }  -  \widetilde{\mathsf{W}}_{k_{t}}(\lambda_{t})  \right| \to 0 \ \text{ in probability}.
 \end{equation*}

   \textbf{Step 2}. By  \cite[Theorem 1.1 (iv)]{Pain18} again, as $\sqrt{k_{t}}/\alpha_{t} \to \infty$, we have $\alpha_{t} \mathsf{W}_{k_{t}}(\lambda_{t}) \to 2\sqrt{2}  \mathsf{Z}_{\infty}$ in probability.
  To prove \eqref{suffice}, it suffices to show that
   \begin{equation*}
   \alpha_{t} [\widetilde{\mathsf{W}}_{k_{t}}(\lambda_{t})- \mathsf{W}_{k_{t}}(\lambda_{t}) ] =  \alpha_{t}\mathsf{W}_{k_{t}}(\lambda_{t}) \sum_{v \in \mathsf{N}_{k(t)}} \frac{e^{\lambda_{t}\mathsf{X}_{v}(k_{t})-(\frac{\lambda_{t}^2}{2}+1)k_{t}} }{\mathsf{W}_{k_{t}}(\lambda_{t})} \left[ \delta_{t} \left(\frac{\lambda_{t} k_{t}- \mathsf{X}_{v}(k_{t})}{\sqrt{t}} \right) - 1 \right]  \to 0
   \end{equation*}
in probability as $t\to\infty$.
By \cite[Corollary 1.3]{Pain18}, for any   $\epsilon > 0$,  there exists $K > 0$
such that for $t$ large enough, with probability at least $1-\epsilon/2$ we have
\begin{equation*}
  \sum_{v \in \mathsf{N}_{k(t)}} \frac{e^{\lambda_{t}\mathsf{X}_{v}(k_{t})-(\frac{\lambda_{t}^2}{2}+1)k_{t}} }{\mathsf{W}_{k_{t}}(\lambda_{t})}    1_{ \{ |\lambda_{t} k_{t}- \mathsf{X}_{v}(k_{t})| > K \sqrt{k_{t}}  \} }  <  \epsilon.
\end{equation*}
Again, by $\lim_{t\to\infty}\alpha_{t} \mathsf{W}_{k_{t}}(\lambda_{t})=2\sqrt{2}  \mathsf{Z}_{\infty}$,
there exists $K'>0$ such that  with probability at least $1-\epsilon/2$,  we have  $ | \alpha_{t} \mathsf{W}_{k_{t}}(\lambda_{t})| \leq K'$. Then with probability $1-\epsilon$,
\begin{align*}
  &  \alpha_{t} |\widetilde{\mathsf{W}}_{k_{t}}(\lambda_{t})- \mathsf{W}_{k_{t}}(\lambda_{t}) | \\
  & \leq
  K' \sum_{v \in \mathsf{N}_{k(t)}} \frac{e^{\lambda_{t}\mathsf{X}_{v}(k_{t})-(\frac{\lambda_{t}^2}{2}+1)k_{t}} }{\mathsf{W}_{k_{t}}(\lambda_{t})} \left| \delta_{t} \left(\frac{\lambda_{t} k_{t}- \mathsf{X}_{v}(k_{t})}{\sqrt{t}} \right) - 1 \right|  1_{ \{ |\lambda_{t} k_{t}- \mathsf{X}_{v}(k_{t})| \leq  K \sqrt{k_{t}}  \} } + (c_0+1)\epsilon  \\
 & \leq K'  \left(  \sup\{ |\delta_{t}(y)-1|: |y| \leq K \sqrt{k_{t}/t} \} +  (c_0 + 1) \epsilon \right) .
\end{align*}
 Noticing that $ \delta_{t}(y) \leq 2 e^{  (\bar{r}+ A \bar{h}) |y|- \frac{y^2}{2}} $ and
\begin{align*}
  \delta_{t}(y)  &=  \frac{\sqrt{t}}{\sqrt{  t-k_{t}}}  \frac{\int F_{t}(z)  e^{-\frac{t}{2(t-k_{t})}(z-y)^2}   \dif z  } {\int F_{t}(z) e^{-\frac{z^2}{2}} \dif z}
   \geq     \frac{\int F_{t}(z)  e^{- \frac{(z-y)^2}{2} - \frac{k_{t}}{2(t-k_{t})} (z-y)^2  }   \dif z  } {\int F_{t}(z) e^{-\frac{z^2}{2}} \dif z}  \\
 & \geq    \frac{\int F_{t}(z)  e^{- \frac{z^2}{2} -   |z| |y| - \frac{k_{t}}{t} (|z|+|y|)^2  }   \dif z  } {\int F_{t}(z) e^{-\frac{z^2}{2}} \dif z} e^{- \frac{y^{2}}{2} }
 \geq    e^{  (\bar{r}+ A \bar{h}) |y|- \frac{y^2}{2}}  e^{- \frac{k_{t}}{t} (|y|+\bar{r}+ A \bar{h} )^2  },
 \end{align*}
we have  $\sup\{ |\delta_{t}(y)-1|: |y| \leq K \sqrt{k_{t}/t} \} \to 0$ as $t \to \infty $.
Then $\alpha_{t} |\widetilde{\mathsf{W}}_{k_{t}}(\lambda_{t})- \mathsf{W}_{k_{t}}(\lambda_{t}) |\to 0$ in probability as $t\to\infty$ by the arbitrarily of $\epsilon$, and the desired result follows.
\end{proof}

\section{Proof of Lemma \ref{lem-computation}}\label{App-B}
\begin{proof}[Proof of Lemma \ref{lem-computation}]
By definition $v^{*}_{t}=a_{t}p_{t}+b_{t}(1-p_{t})$, so we have  $y=(b_t-\sqrt{2})(1-p_t)t+(\sqrt{2}-a_{t})u $.
      Consider first
    \begin{equation*}
      \begin{aligned}
   & L_{1}:= (\beta_{t}-\frac{a_t^2}{2\sigma
    ^2_{t}})s - \sqrt{2} y - \frac{y^2}{2(1-p_{t})t}\\
    &= \left[ (\beta_{t}-\frac{a_t^2}{2\sigma
    ^2_{t}})p_{t} - (\sqrt{2}b_{t}-2)(1-p_{t})-\frac{(b_{t}-\sqrt{2})^2 (1-p_{t}) }{2} \right] t \\
     & \qquad  +  \left[ (\beta_{t}-\frac{a_t^2}{2\sigma
     ^2_{t}}) - \sqrt{2}(\sqrt{2}-a_{t})-(b_t-\sqrt{2})(\sqrt{2}-a_t)   \right] u     - \frac{(\sqrt{2}-a_{t})^2}{2(1-p_{t})}\frac{u^2}{t} .
      \end{aligned}
    \end{equation*}
    We write $\Delta_t:=\frac{b_t}{\sqrt{2}}-1  $.
    Recalling our definition  \eqref{eq-def-a-b-p} and \eqref{eq-def-star-a-b-p},
     a little computation yields the following claims.
    \begin{itemize}
      \item The coefficient for term $t$ equals  $(\beta_{t}-\frac{a_t^2}{2\sigma
    ^2_{t}})p_{t} + (1-\frac{b_t^2}{2})(1-p_{t})=0 $  by \eqref{eq-condition-zero}.
    \item The coefficient for term $u$ equals  $\beta_{t}+\frac{\sigma^2_t b^2_t}{2}-\sqrt{2}b_{t}=\beta_t+\sigma^2_t \frac{\beta
    _t-1}{1-\sigma^2_t} - \sqrt{2}b_t=\frac{\beta_t-\sigma^2_t}{1-\sigma^2_t}-\sqrt{2}b_t=\frac{b_t^2}{2}+1-\sqrt{2}b_t =\Delta_{t}^2$.
    \end{itemize}
    Thus $L_1= \Delta_{t}^2  u  - \frac{(\sqrt{2}-a_{t})^2}{2(1-p_{t})}\frac{u^2}{t}$.
   Secondly, letting $A_t= \frac{(\sqrt{2}-a_{t})}{(1-p_t)t}$,  we have
    \begin{equation*}
      \begin{aligned}
     L_2 &= \frac{y^2}{2(1-p_{t})t}- \frac{y^2}{2(1-p_{t})t-2u} =  \frac{-u[(b_t-\sqrt{2})(1-p_t)t+(\sqrt{2}-a_{t})u]^2}{2(1-p_t)t[(1-p_t)t-u]} \\
     &= -u \left[ (\frac{b_t}{\sqrt{2}}-1)  + \frac{(\sqrt{2}-a_{t})u}{(1-p_t)t} \right]^2 (1- \frac{u}{(1-p_t)t} )^{-1}\\
     &=    -u \left( \Delta_{t} +  A_{t} u \right)^{2} (1+  \frac{u}{(1-p_t)t-u} ) = - \Delta_{t}^{2} u -r(u,t),
   \end{aligned}
    \end{equation*}
   where $r(u,t):=  (2\Delta_{t} + A_{t} u) A_{t} u^{2}    + \frac{ \left( \Delta_{t} +  A_{t} u \right)^{2}}{(1-p_t)t-u} u^{2}  $. Therefore
   \begin{equation*}
    L(u,t)= L_{1}(u,t)+ L_{2}(u,t) =  -\frac{(\sqrt{2}-a_{t})^2}{2(1-p_{t})}\frac{u^2}{t} - r(u,t).
   \end{equation*}

  \textbf{Case (i).} By \eqref{eq-asymptotic-abv}   we have  $p_{t}\sim \frac{1}{2(1-\sigma^2)^2t^{h}}$,  $\Delta_{t}\sim \frac{1}{2(1-\sigma^2)t^{h}}$, $\sqrt{2}-a_{t}\sim \sqrt{2}(1-\sigma^2)$ and   $A_{t} \sim \frac{\sqrt{2}(1-\sigma^2)}{t}$. Then, for fixed  $ \xi$,
    \begin{equation*}
    L(\xi \sqrt{t},t) =  -\frac{(\sqrt{2}-a_{t})^2}{2(1-p_{t})}\xi ^{2} - r(\xi \sqrt{t},t)  \to  -(1-\sigma^2)^{2} \xi^2   \text{ as } t \to \infty.
    \end{equation*}
Moreover, we claim that for large $t$,  $r(u; t) \geq 0$ for all $u>-p_{t} t$. In fact,
the  claim is true if we have
 $2\Delta_{t} + A_{t} u \geq 2\Delta_{t} - A_{t} p_{t} t \geq 0 $.
 Note that
  $2\Delta_{t} - A_{t} p_{t} t = 2 \Delta_{t}-  \frac{(\sqrt{2}-a_{t})}{(1-p_t)} p_{t}= [2 \Delta_{t} t^{h} -  \frac{(\sqrt{2}-a_{t})}{(1-p_t)} p_{t} t^{h} ] t^{-h}$. By  \eqref{eq-asymptotic-abv}, we have  $ [2 \Delta_{t} t^{h} -  \frac{(\sqrt{2}-a_{t})}{(1-p_t)} p_{t} t^{h} ]  \to \frac{1}{1-\sigma^{2}}-\frac{1}{\sqrt{2}(1-\sigma^2)}>0 $.  This proves our claim, and thus $L(\xi \sqrt{t},t) \leq    -\frac{(\sqrt{2}-a_{t})^2}{2(1-p_{t})}\xi ^{2}  \leq - c \xi^{2} $ for some constant $c$.

  \textbf{Case (ii)}. By \eqref{eq-asymptotic-abv-3}, we have $p_{t} = 1/2$,
  $\Delta_{t}\sim \frac{1}{2 t^{h/2}}$, $\sqrt{2}-a_{t}\sim \frac{1}{\sqrt{2} t^{h/2}}$ and   $A_{t} \sim \frac{\sqrt{2}}{t^{1+\frac{h}{2}}}$.  Then  $\lim_{t\to\infty} r(\xi \sqrt{t},t)=  2\Delta_{t} A_{t} t^{1+h} \xi^{2}    +  2\Delta_{t}^{2} t^{h} \xi^{2} =  (\sqrt{2}+\frac{1}{2})\xi^{2}$, and
  \begin{equation*}
L(\xi t^{\frac{1+h}{2}},t) =    -(\sqrt{2}-a_{t})^2 t^{h} \xi ^{2}
-  r(\xi \sqrt{t},t)  \to   -(\sqrt{2}+1)\xi^2 .
  \end{equation*}
Finally we prove that for large $t$,  $r(u; t) \geq 0$ for all $u>-1/2 t$ by  showing that   $2\Delta_{t} + A_{t} u \geq 2\Delta_{t} - \frac{1}{2} A_{t} t \geq 0 $. This follows from the fact that  $\lim_{t\to\infty} t^{h/2}[2\Delta_{t} - \frac{1}{2} A_{t}t ]  \to 1-\frac{\sqrt{2}}{2} >0$. Hence $L(\xi \sqrt{t},t) \leq  - (\sqrt{2}-a_{t})^2t^{h} \xi ^{2} \leq - c \xi^{2} $ for some constant $c$.
    \end{proof}
\end{appendix}



\bibliographystyle{amsplain}

\providecommand{\bysame}{\leavevmode\hbox to3em{\hrulefill}\thinspace}
\providecommand{\MR}{\relax\ifhmode\unskip\space\fi MR }
\providecommand{\MRhref}[2]{%
  \href{http://www.ams.org/mathscinet-getitem?mr=#1}{#2}
}
\providecommand{\href}[2]{#2}
\begin{thebibliography}{}

\end{thebibliography}


\begin{thebibliography}{99}
  \bibitem{ABBS13}
  E. A{\"\i}d{\'e}kon, J.~Berestycki, {\'E}.~Brunet, and Z.~Shi.   Branching {B}rownian motion seen from its tip.
    {\em Probab. Theory Relat. Fields}, 157(1):405--451, 2013.

  \bibitem{Arguin16}
  L.-P. Arguin.
    {Extrema of log-correlated random variables principles and examples.}
    {\em Advances in disordered systems, random processes and some applications},  pages 166--204, Cambridge Univ. Press, Cambridge, 2017.

  \bibitem{ABK11}
  L.-P. Arguin, A.~Bovier, and N.~Kistler.
    Genealogy of extremal particles of branching {B}rownian motion.
    {\em Comm. Pure Appl. Math.},
    64(12):1647--1676, 2011.

  \bibitem{ABK12}
  L.-P. Arguin, A.~Bovier, and N.~Kistler.
    Poissonian statistics in the extremal process of branching {Brownian}
    motion.
    {\em Ann. Appl. Probab.}, 22(4):1693--1711, 2012.

  \bibitem{ABK13}
  L.-P. Arguin, A.~Bovier, and N.~Kistler.
    The extremal process of branching {B}rownian motion.
    {\em Probab. Theory Relat. Fields}, 157(3):535--574, 2013.

  \bibitem{BK22}
  E.~C. Bailey and J.~P. Keating.
    Maxima of log-correlated fields: some recent developments.
    {\em J. Phys. A, Math. Theor.}, 55(5):76, 2022.
    Id/No 053001.

  \bibitem{Belloum22}
  M.~A. Belloum The extremal process of a cascading family of branching {B}rownian
    motion.
    {\em arXiv:2202.01584}, 2022.

  \bibitem{BM21}
  M.~A. Belloum and B.~Mallein.
    {Anomalous spreading in reducible multitype branching {B}rownian
    motion}.
    {\em Electron. J. Probab.}, 26, no. 39, 2021.

  \bibitem{BBCM22}
  J. Berestycki, {\'E}.~Brunet, A.~Cortines, and B.~Mallein.
    {A simple backward construction of branching {B}rownian motion with
    large displacement and applications}.
    {\em  Ann. Inst. Henri Poincar{\'e}, Probab. Stat.}, 58(4):2094--2113, 2022.


  \bibitem{BGKM23}
  J.~Berestycki, C.~Graham, Y.~H. Kim, and B.~Mallein.
    KPP traveling waves in the half-space.
    {\em arXiv:2305.17057}, 2023.

  \bibitem{BBGMRR22}
  J.~Berestycki,
  E. Brunet,
  C.~Graham, L.~Mytnik, J.-M. Roquejoffre and
    L.~Ryzhik.
    The distance between the two BBM leaders.
    {\em Nonlinearity}, 35(4):1558, feb 2022.


  \bibitem{Biggins10}
  J. D. Biggins.
    Branching out.
    In {\em Probability and mathematical genetics. Papers in honour of
    Sir John Kingman}, pages 113--134. Cambridge: Cambridge Univ. Press,
    2010.

  \bibitem{Biggins12}
  J. D. Biggins.
    {Spreading speeds in reducible multitype branching random walk}.
    {\em Ann. Appl. Probab.}, 22(5):1778--1821, 2012

  \bibitem{BH17}
  A.~Bovier and L.~Hartung.
    Extended convergence of the extremal process of branching {Brownian}
    motion.
    {\em Ann. Appl. Probab.}, 27(3):1756--1777, 2017.

  \bibitem{Bh20}
  A.~Bovier and L.~Hartung.
    From 1 to 6 : a finer analysis of perturbed branching {Brownian}
    motion.
    {\em Commun. Pure Appl. Math.}, 73(7):1490--1525, 2020.

  \bibitem{BH14}
  A.~Bovier and L.~B. Hartung.
    The extremal process of two-speed branching {Brownian} motion.
    {\em Electron. J. Probab.}, 19:28, 2014.
    Id/No 18.

  \bibitem{Bramson78}
  M. Bramson.
    Maximal displacement of branching {B}rownian motion.
    {\em Comm.  Pure  Appl. Math.}, 31(5):531--581, 1978.

  \bibitem{Bramson83}
  M. Bramson.
    {\em Convergence of solutions of the {Kolmogorov} equation to
    travelling waves}, volume 285 of {\em Mem. Am. Math. Soc.}
    Providence, RI: American Mathematical Society (AMS), 1983.


  \bibitem{CR88}
  B.~Chauvin and A.~Rouault.
    {KPP} equation and supercritical branching {Brownian} motion in the
    subcritical speed area. {Application} to spatial trees.
    {\em Probab. Theory Relat. Fields}, 80(2):299--314, 1988.

  \bibitem{CHL19}
  A.~Cortines, L.~Hartung, and O.~Louidor.
    The structure of extreme level sets in branching {Brownian} motion.
    {\em Ann. Probab.}, 47(4):2257--2302, 2019.

  \bibitem{Fisher37}
A. R. Fisher.
     The wave of advance of advantageous genes
    {\em Ann. Eugen.}, 7 (4): 353--369, 1937.


  \bibitem{Harris99}
  S.~C. Harris.
    Travelling-waves for the {F--KPP} equation via probabilistic arguments.
    {\em Proc. R. Soc. Edinb., Sect. A, Math.}, 129(3):503--517, 1999.


  \bibitem{HRS23}
    H. Hou, Y.-X. Ren, and R. Song.
      Extremal process for irreducible multitype branching {B}rownian motion,   {\em  arXiv:2303.12256},  2023.


  \bibitem{Kistler14}
  N.~Kistler.
    Derrida's random energy models.
  From spin glasses to the extremes of
    correlated random fields,   {\em     arXiv:1412.0958}, 2014.

  \bibitem{KPP37}
 A. Kolmogorov, I.  Petrovsky  and N. Piskounov.
    {\'E}tude de l'{\'e}quation de la diffusion avec croissance de la
    quantite de mati{\`e}re et son application {\`a} un probl{\`e}me biologique.
    Bull. {Univ}. {\'E}tat {Moscou}, {S{\'e}r}. {Int}., {Sect}. {A}:
    {Math}. et {M{\'e}can}. 1, {Fasc}. 6, 1-25 (1937), 1937.

  \bibitem{Kyprianou04}
  A.~E. Kyprianou.
    Travelling wave solutions to the {K}-{P}-{P} equation: alternatives
    to {Simon} {Harris}' probabilistic analysis.
    {\em Ann. Inst. Henri Poincar{\'e}, Probab. Stat.}, 40(1):53--72,
    2004.


\bibitem{LS87}
S. P. Lalley and T. Sellke.
  {A conditional limit theorem for the frontier of a branching Brownian
  motion}.
  {\em  Ann.  Probab.}, 15(3):1052-1061, 1987.

  \bibitem{MR23}
  H.~Ma and Y.-X. Ren.
    Double jump in the maximum of two-type reducible branching Brownian
    motion,   {\em arXiv:2305.09988}, 2023.

  \bibitem{Madaule16}
  T.~Madaule.
    First order transition for the branching random walk at the critical
    parameter.
   {\em Stochastic Process. Appl.},
  126(2):470--502, 2016.


\bibitem{Madaule16tail}
T.~Madaule.
  The tail distribution of the derivative martingale and the global
  minimum of the branching random walk,   {\em arXiv:1606.03211}, 2016.


  \bibitem{McKean75}
  H.~P. McKean.
    Application of
 Brownian motion to the equation of
    Kolmogorov-Petrovskii-Piskunov.
    {\em Comm. Pure Appl. Math.}, 28(3):323--331,
    1975.

  \bibitem{MRR22}
  L.~Mytnik, J.-M. Roquejoffre, and L.~Ryzhik.
    Fisher-{KPP} equation with small data and the extremal process of
    branching {Brownian} motion.
    {\em Adv. Math.}, 396:58, 2022.
    Id/No 108106.

\bibitem{Pain18}
M. Pain.
  {The near-critical Gibbs measure of the branching random walk}.
  {\em Ann. Inst. Henri Poincar{\'e}, Probab. Stat.}, 54(3):1622--1666, 2018.

  \bibitem{RY}
  Y.-X. Ren and T. Yang.
    Multitype branching Brownian motion and traveling waves.
    {\em Adv. Appl. Probab.}, 46(1):217--240, 2014.


  \bibitem{KS15}
M. Schmidt and N. Kistler.
  {From Derrida's random energy model to branching random walks: from 1
  to 3}.
  {\em Electron. Comm. Probab.}, 20:1--12, 2015.


  \bibitem{Zeitouni12}
  O.~Zeitouni.
    {\em Lecture notes on branching random walks and the Gaussian free
    field}.
    https://www.wisdom.weizmann.ac.il/*zeitouni/pdf/notesBRW.pdf, 2012.


  \end{thebibliography}


\end{document}